\documentclass[12pt]{article}

\usepackage{amsmath}
\usepackage{amssymb}
\usepackage{epsfig}

\usepackage{amsfonts}
\usepackage{amsthm}
\usepackage{graphicx}

\newtheorem{theorem}{Theorem}
\newtheorem{lemma}{Lemma}

\newcommand{\e}{\varepsilon}
\newcommand{\EE}{\mathsf{E}}
\newcommand{\PP}{\mathsf{P}}
\newcommand{\VV}{\mathsf{V}}
\newcommand{\XX}{\mathbb{X}}

\newcommand{\beq}{\begin{equation}}
\newcommand{\eeq}{\end{equation}}
\newcommand{\lb}{\label}

\begin{document}

\begin{center}
{\large \bf Nonlinearly Perturbed \vspace{1mm} \\ Birth-Death-Type Models} \\
\end{center}
\vspace{2mm}

\begin{center}
Dmitrii Silvestrov,\footnote{Department of Mathematics, Stockholm University, 106 91, Stockholm, Sweden, \\ 
E-mail: silvestrov@math.su.se}  Mikael Petersson,\footnote{Department of Mathematics, Stockholm University, 106 91, Stockholm, Sweden, \\ 
E-mail: mikpe@math.su.se}    Ola H\"{o}ssjer\footnote{Department of Mathematics, Stockholm University, 106 91, Stockholm, Sweden, \\ 
E-mail: ola@math.su.se} 
\end{center}
\vspace{2mm}

{\bf Abstract}:
Asymptotic expansions for stationary and conditional quasi-stationary distributions of 
nonlinearly perturbed birth-death-type semi-Mar\-kov models are presented. Applications to models of population growth, epidemic spread and population genetics are discussed. \\



{\bf 1.  Introduction} \\

In this paper, we present new algorithms  for  construction of asymptotic expansions for stationary and conditional quasi-stationary distributions of nonlinearly perturbed birth-death-type  semi-Markov processes with a finite phase space.  

We consider models, where  the phase space is one class of communicative states, for embedded Markov chains of pre-limiting perturbed birth-death-type  semi-Markov processes, while it can consist of one  closed class of communicative states or  consist of one class of communicative transient internal states and one or both absorbing end states,  for the limiting embedded Markov chain. 

The initial perturbation conditions are formulated  in the forms of Taylor asymptotic expansions for transition probabilities (of  embedded Markov chains)  
and  expectations of transition times, for perturbed semi-Markov processes.    

The algorithms are based  on special time-space screening procedures for sequential phase space reduction and algorithms for  re-calculation of asymptotic expansions, which constitute perturbation conditions for the  semi-Mar\-kov processes with reduced phase spaces. 

The final asymptotic expansions for stationary distributions of nonlinearly perturbed semi-Markov processes are given in the form of Taylor asymptotic expansions.

Models of perturbed Markov chains  and semi-Markov processes, in particular, for the most difficult cases of perturbed processes with absorption  and so-called singularly perturbed  processes, attracted attention of researchers in the mid of the 20th century.  

An interest in these models has been stimulated by  applications to control and queuing systems, information networks, epidemic models and models of mathematical genetics and population dynamics. As a rule, Markov-type processes with singular perturbations appear as  natural tools for mathematical analysis of multi-component systems with weakly interacting components.

We refer here to  the latest books containing results on  asymptotic expansions for perturbed Markov chains and semi-Markov processes,  
Stewart  (1998, 2001), Korolyuk, V.~S. and Korolyuk, V.~V. (1999),  Konstantinov,   Gu,  Mehrmann and Petkov (2003), 
Bini,  Latouche and Meini (2005), Koroliuk  and Limnios (2005), Yin and Zhang   (2005, 2013),  Gyllenberg  and Silvestrov (2008) and   Avrachenkov, Filar   and   Howlett (2013) and the research report by Silvestrov,  D.  and  Silvestrov, S. (2015), where readers can find comprehensive bibliographies of works in the area. 

In this paper, we concentrate our attention on several perturbed birth-death models of biological nature. 
%We will discuss three types of applications of the birth-death type semi-Markov processes.
The first application is population dynamics in a constant environment, where one individual at a time is born or dies, see for instance Lande et al., (2003). The second application is epidemic spread of a disease, reviewed in Hethcote (2000) and N{\aa}sell (2011). Here one individual at a time gets infected or recovers, and recovered individuals become susceptible for new infections. The third application is population genetic models, treated extensively in Crow and Kimura (1970) and Ewens (2004). We focus in particular on models with overlapping generations, introduced by Moran (1958a). These Moran type models focus on the the dynamics of the variants of a certain gene for a one-sex population, with an assumption that a copy of the gene is replaced for one individual at a time.  

The paper includes 8 sections. In Section 2, we introduce a model of perturbed semi-Markov processes, including processes of birth-death type, define conditional quasi-stationary distributions for such processes and formulate basic perturbation conditions. In Section 3, we describe examples of perturbed population dynamics, epidemic and population genetic models, which can be described in the framework of birth-death-type Markov chains and semi-Markov processes. In Section 4, we present time-space screening procedures of phase space reduction for perturbed semi-Markov processes and recurrent algorithms for computing expectations of hitting times and stationary and conditional quasi-stationary distributions for birth-death semi-Markov processes. In Section 5, we get the first and the second order asymptotic expansions  for  stationary and conditional quasi-stationary distributions of perturbed semi-Markov processes and give explicit formulas for  coefficients in these expansions. In Section 6, we describe general recurrent algorithms for construction of high order asymptotic expansions  for  stationary and conditional quasi-stationary distributions of perturbed birth-death-type semi-Markov processes. In Section 7, we apply the above asymptotic results  to the perturbed birth-death models of biological nature presented in Section 3 and present results of related numerical studies. In Section 8,  we give concluding remarks and comments. \\

{\bf 2. Nonlinearly perturbed  semi-Markov processes} \\

In this section, we introduce a model of perturbed semi-Markov processes, including processes of birth-death type, define conditional quasi-stationary distributions for such processes and formulate basic perturbation conditions. \\

{\bf 2.1. Perturbed semi-Markov processes}.\\

Let  ${\mathbb X} = \{0, \ldots, N \}$  and $(\eta^{(\e)}_n, \kappa^{(\e)}_n), n = 0, 1, \ldots$ be, for every value of a perturbation parameter $\e \in (0, \e_0]$, where $0  < \e_0  \leq 1$, a Markov renewal process, i.e., a homogeneous Markov chain with the phase space ${\mathbb X} \times [0, \infty)$, an initial distribution $\bar{p}^{(\e)} = \langle p^{(\e)}_i = \PP \{\eta^{(\e)}_0 = i, \kappa^{(\e)}_0 = 0 \} = \PP \{\eta^{(\e)}_0 = i \}, i \in {\mathbb X} \rangle$ and transition probabilities, defined for  $(i, s), (j, t) \in  {\mathbb X} \times [0, \infty)$,
\begin{equation}\label{semika}
Q^{(\e)}_{ij}(t) = \PP \{ \eta^{(\e)}_{1} = j, \kappa^{(\e)}_{1} \leq t / \eta^{(\e)}_{0} = i, \kappa^{(\e)}_{0} = s \}. 
\end{equation}

In this case,  the random sequence $\eta^{(\e)}_n$ is also a homogeneous (embedded) Markov chain with 
the  phase space $\XX$ and the transition probabilities, defined for  $i, j \in \XX$,
\begin{equation}\label{embed}
p_{ij}(\e) = \PP \{ \eta^{(\e)}_{1} = j / \eta^{(\e)}_{0} = i \} = Q^{(\e)}_{ij}(\infty). 
\end{equation}

We assume that the following condition holds:
\begin{itemize}
\item [${\bf A}$:] $\XX$ is a communicative class of states for the embedded Markov chain $\eta^{(\e)}_n$, for every $\e \in (0, \e_0]$.
\end{itemize}

We exclude instant transitions and assume that the following condition holds:
\begin{itemize}
\item [${\bf B}$:]  $Q_{ij}^{(\e)}(0) = 0, \ i, j \in \XX$, for every $\e \in (0, \e_0]$.
\end{itemize}

Let us now introduce  a semi-Markov process,
\begin{equation}\label{sepr}
\eta^{(\e)}(t) = \eta^{(\e)}_{\nu^{(\e)}(t)},  \ t \geq 0,
\end{equation}
where $\nu^{(\e)}(t) = \max(n \geq 0: \zeta^{(\e)}_n \leq t)$
is a number of jumps in the  time interval  $[0, t]$, for $t \geq 0$, and 
\begin{equation}\label{zetan}
\zeta^{(\e)}_n = \kappa^{(\e)}_1 + \cdots + \kappa^{(\e)}_n,
\end{equation}
$n = 0, 1, \ldots$, are sequential moments of jumps, for the semi-Markov process $\eta^{(\e)}(t)$.

This process has the phase space $\XX$, the initial distribution $\bar{p} = \langle p_i = \PP \{\eta^{(\e)}(0) = i \}, 
i \in {\mathbb X} \rangle$ and transition probabilities $Q^{(\e)}_{ij}(t), t \geq 0,  i, j \in \XX$. 

If  $Q^{(\e)}_{ij}(t) = {\rm I}(t \geq 1)p_{ij}(\e), t \geq 0, i, j \in \XX$,  then $\eta^{(\e)}(t) = \eta^{(\e)}_{[t]}, t \geq 0$ is a discrete time homogeneous Markov chain embedded in continuous time.

If $Q^{(\e)}_{ij}(t) = (1 - e^{- \lambda_i(\e) t})p_{ij}(\e), t \geq 0, i, j \in \XX$ (here,   $0 <  \lambda_i(\e) < \infty, i \in \XX$), then $\eta^{(\e)}(t), t \geq 0$ is a continuous time homogeneous Markov chain. 

Let us denote,  for  $t \geq 0, i \in \XX$,
\begin{equation}\label{embedada}
F^{(\e)}_{i}(t)  = \PP_i \{ \kappa^{(\e)}_{1} \leq t \} = \sum_{j \in \XX} Q^{(\e)}_{ij}(t). 
\end{equation}
and 
\begin{equation}\label{embedadam}
e_i(\e)  = \EE_i \kappa^{(\e)}_{1} = \int_0^\infty t  F^{(\e)}_{i}(dt), 
\end{equation}

Here and henceforth,  notations $\PP_i $ and $ \EE_i $ are used for conditional probabilities and expectations under 
condition $\eta^{(\e)}(0) = i$.

Let us also introduce the conditional distributions of transition times  $\kappa^{(\e)}_n$, defined for  $t \geq 0, i, j \in \XX$, 
\begin{equation*}
F^{(\e)}_{ij}(t) = \PP \{ \kappa^{(\e)}_{1} \leq t / \eta^{(\e)}_{0} = i, \eta^{(\e)}_{1} = j  \} \makebox[20mm]{}
\end{equation*}
\begin{equation}\label{embedas}
= \left\{
\begin{array}{lll}
Q^{(\e)}_{ij}(t)/p_{ij}(\e) & \text{if} \ \ p_{ij}(\e)  > 0,  \\
F^{(\e)}_{i}(t)  &  \text{if} \ \  p_{ij}(\e)  = 0. 
\end{array}
\right.
\end{equation}
and also denote, for $i, j \in \XX$,
\begin{equation}\label{embedadam}
f_{ij}(\e)  = \EE_i \{ \kappa^{(\e)}_{1} / \eta^{(\e)}_{1} = j \} = \int_0^\infty t  F^{(\e)}_{ij}(dt), 
\end{equation}
and
\begin{equation}\label{embedadam}
e_{ij}(\e)  = \EE_i \{\kappa^{(\e)}_{1} I( \eta^{(\e)}_{1} = j)\} = \int_0^\infty t  Q^{(\e)}_{ij}(dt),  
\end{equation}

We also assume that the following condition holds: 
\begin{itemize}
\item [${\bf C}$:] $e_{ij}(\e) < \infty, \ i, j \in \XX$, for   $\e \in (0, \e_0]$.
\end{itemize}

Obviously, the above expectations are connected by the following relations, for $i, j \in \XX$, 
\begin{equation}\label{embedadamad}
e_{ij}(\e)  = f_{ij}(\e) p_{ij}(\e),   
\end{equation}
and
\begin{equation}\label{embedadamad}
e_i(\e)  = \sum_{j \in \XX} f_{ij}(\e) p_{ij}(\e)   = \sum_{j \in \XX} e_{ij}(\e).  
\end{equation}

In the case of discrete time Markov chain $f_{ij}(\e)  = 1, i, j \in \XX$. 

In the case of continuous time Markov chain 
$f_{ij}(\e)  = \lambda^{-1}_{i}(\e), i, j \in \XX$.  

Conditions ${\bf A}$ -- ${\bf C}$ imply that the semi-Markov process $\eta^{(\e)}(t)$ 
is, for every $\e \in (0, \e_0]$, ergodic in the sense that the following asymptotic relation holds, 
\begin{equation}\label{statik}
\mu^{(\e)}_i(t) = \frac{1}{t} \int_0^t I(\eta^{(\e)}(s) = i)ds \stackrel{\PP}{\longrightarrow} \pi_i(\e)  \ {\rm as} \ t \to \infty, \ i \in \XX.
 \end{equation}
 
The ergodic relation (\ref{statik}) holds for any initial distribution $\bar{p}^{(\e)}$,  and the stationary probabilities  
$\pi_i(\e), i \in \XX$  do not depend on the initial distribution. Moreover, $\pi_i(\e) > 0, i \in \XX$ and $\sum_{i \in \XX} \pi_i(\e) = 1$. \\

{\bf 2.2. Perturbed semi-Markov processes of birth-death type} \\

The semi-Markov process $\eta^{(\e)}(t)$ is of birth-death type if the following relation holds, for $t \geq 0$, 
\begin{equation}\label{emsemi}
Q^{(\e)}_{ij}(t) = \left\{
\begin{array}{lll}
F^{(\e)}_{0, \pm}(t) p_{0, \pm}(\e)  &  \text{if} \ j =  0 + \frac{1 \pm  1}{2},  \ \text{for} \ i = 0, \\ 
F^{(\e)}_{i, \pm}(t) p_{i, \pm}(\e)    & \text{if} \ j = i \pm 1, \quad \quad \text{for} \ 0 < i < N,   \\ 
F^{(\e)}_{N, \pm}(t) p_{N, \pm}(\e)     &  \text{if} \ j =  N - \frac{1 \mp  1}{2},  \ \text{for} \ i = N, \\
0  & \text{otherwise}.
\end{array}
\right.
\end{equation}
where: (a)  $F^{(\e)}_{i, \pm}(t), i \in \XX$ are, for every $\e \in (0, \e_0]$,  distribution functions concentrated on $[0, \infty)$ such that 
$F^{(\e)}_{i, \pm}(0) = 0,  i \in \XX$; (b) $p_{i, \pm}(\e)  \geq 0, p_{i, -}(\e)  + p_{i, +}(\e) = 1$, for every   $\e \in (0, \e_0]$. 

Let us  also denote, for $i, j \in \XX$,
\begin{equation}\label{emdam}
f_{i, \pm}(\e) = \int_0^\infty t  F^{(\e)}_{i, \pm}(dt), \ e_{i, \pm}(\e) =  f_{i, \pm}(\e)  p_{i, \pm}(\e).
\end{equation}

The following relations take place;
\begin{equation}\label{emsemi}
p_{ij}(\e) = \left\{
\begin{array}{lll}
p_{0, \pm}(\e)  &  \text{if} \ j =  0 + \frac{1 \pm  1}{2},  \ \text{for} \ i = 0, \\ 
p_{i, \pm}(\e)  & \text{if} \ j = i \pm 1, \quad \quad \text{for} \ 0 < i < N,   \\ 
p_{N, \pm}(\e)   &  \text{if} \ j =  N - \frac{1 \mp  1}{2},  \ \text{for} \ i = N, \\
0  & \text{otherwise},
\end{array}
\right.
\end{equation}
and
\begin{equation}\label{emsemi}
e_{ij}(\e) = \left\{
\begin{array}{lll}
e_{0, \pm}(\e)  &  \text{if} \ j =  0 + \frac{1 \pm  1}{2},  \ \text{for} \ i = 0, \\ 
e_{i, \pm}(\e)  & \text{if} \ j = i \pm 1, \quad \quad \text{for} \ 0 < i < N,   \\ 
e_{N, \pm}(\e)   &  \text{if} \ j =  N - \frac{1 \mp  1}{2},  \ \text{for} \ i = N, \\
0  & \text{otherwise}.
\end{array}
\right.
\end{equation}

Let us assume that there exist some integer $0 \leq L < \infty$ such that the  following perturbation conditions hold:
\begin{itemize}
\item [${\bf D}_L$:] $p_{i, \pm}(\e) =  \sum_{l = 0}^{L  + l_{i, \pm}} a_{i, \pm}[l] \e^l + o_{i, \pm}(\e^{L + l_{i, \pm}}), \, \e \in (0, \e_0]$, for $i \in \XX$, where: {\bf (a)}  
$|a_{i, \pm}[l]| < \infty$, for $0 \leq l \leq  L + l_{i, \pm}, i \in \XX$;  {\bf (b)} $l_{i, \pm} = 0$ and
$a_{i, \pm }[0] > 0$, for $0 <  i < N$;  {\bf (c)}  $l_{i, \pm} = 0$  and   $a_{i, \pm }[0] > 0$ or  $l_{i, \pm} = 1$  and $a_{i, \pm }[0] = 0, a_{i, \pm }[1] > 0 $, for $i = 0, N$;      {\bf (d)} 
$o_{i, \pm}(\e^{L + l_{i, \pm}})/\e^{L+ l_{i, \pm}} \to 0$ as $\e \to 0$, for $i \in \XX$. 
\end{itemize}
and
\begin{itemize}
\item [${\bf E}_L$:]  $e_{i, \pm}(\e) =  \sum_{l = 0}^{L + l_{i, \pm}} b_{i, \pm}[l]\e^l + \dot{o}_{i, \pm}(\e^{L + l_{i, \pm}}), \, \e \in (0, \e_0]$, for $i \in \XX$,  where: {\bf (a)} 
$|b_{i, \pm}[l]| < \infty$, for $0 \leq l \leq L + l_{i, \pm}, i \in \XX$;  {\bf (b)} $l_{i, \pm} = 0$ and $b_{i, \pm}[0] >0$, for $0 < i < N$;  {\bf (c)}  $l_{i, \pm} = 0$  and   $b_{i, \pm }[0] > 0$ or 
$l_{i, \pm} = 1$  and   $b_{i, \pm }[0] = 0, b_{i, \pm }[1] > 0 $, for $i = 0, N$;  {\bf (d)} $ \dot{o}_{i}(\e^{L + l_{i, \pm}})/\e^{L + l_{i, \pm}} \to 0$ as $\e \to 0$, for $i \in \XX$.
\end{itemize}

It is useful to explain what role is played by parameter $l_i$ in conditions  ${\bf D}_L$ and ${\bf E}_L$. This parameter equalizes the so-called length of asymptotic expansions penetrating 
these conditions, which is defined as number of coefficients for powers of $\e$ in the corresponding expansions, beginning from the first non-zero coefficients and up to the  
coefficients for the largest powers of $\e$ in the corresponding asymptotic expansions. All expansions penetrating conditions  ${\bf D}_L$ and ${\bf E}_L$ have the length $L$.

Note that conditions ${\bf D}_L$ and ${\bf E}_L$ imply that there exist $\e'_0 \in (0, \e_0]$ such that  probabilities $p_{i, \pm}(\e) > 0, i \in \XX$ and expectations  $e_{i, \pm}(\e)  > 0, i \in \XX$ for  $\e \in (0, \e'_0]$. This let us just assume that $\e'_0 = \e_0$.  

The model assumption, $p_{i, -}(\e)  + p_{i, +}(\e) = 1, \e \in (0, \e_0]$,  also implies that the following condition should hold:
\begin{itemize}
\item [${\bf F}_L$:] $a_{i, -}[0] + a_{i, +}[0] = 1, a_{i, -}[l] + a_{i, +}[l] =  0, 1 \leq l \leq  L + l_{i, +} \wedge l_{i, -}$, for $i \in \XX$.
\end{itemize}

We also assume that the following natural consistency condition for asymptotic expansions penetrating perturbation conditions  ${\bf D}$ and ${\bf E}$  hold:
\begin{itemize}
\item [${\bf G}$:] $b_{i, \pm }[0] > 0$ if and only if $a_{i, \pm }[0] > 0$, for $i = 0, N$.  
\end{itemize}

Condition ${\bf D}_L$ implies that there exist $ \lim_{\e \to 0} p_{i, \pm}(\e) = p_{i, \pm}(0), i \in \XX$ and, thus, there also exist $\lim_{\e \to 0} p_{ij}(\e) = p_{ij}(0), i, j \in \XX$. 

Condition  ${\bf E}_L$ implies that there exist $ \lim_{\e \to 0} e_{i, \pm}(\e) = e_{i, \pm}(0), i \in \XX$ and, thus, there also exist $\lim_{\e \to 0} e_{ij}(\e) = e_{ij}(0), i, j \in \XX$. 

There are tree basic variants of the model, where one of the following conditions hold:
\begin{itemize}
\item [${\bf H_1}$:]  $a_{0, + }[0] > 0,  a_{N, - }[0] > 0$.
\end{itemize}
\begin{itemize}
\item [${\bf H_2}$:]  $a_{0, + }[0] = 0,  a_{N, - }[0] > 0$.
\end{itemize}
\begin{itemize}
\item [${\bf H_3}$:]  $a_{0, + }[0] = 0,  a_{N, - }[0] = 0$. 
\end{itemize}

The limiting birth-death type Markov chain $\eta^{(0)}_n$ with the matrix of transition probabilities $\| p_{ij}(0) \|$ has: {\bf (a)} one class of communicative states $\XX$, if condition ${\bf H_1}$ holds, {\bf (b)} one communicative class of transient states $_0\XX = \XX \setminus \{ 0 \}$ and an absorbing state $0$, if condition ${\bf H_2}$ holds, and {\bf (c)} one communicative class of transient states $_{0, N}\XX =  \XX \setminus \{ 0, N \}$ and two absorbing states $0$ and $N$, if condition ${\bf H_3}$ holds.

The case  $a_{0, + }[0] > 0,  a_{N, - }[0] = 0$ is analogous to the case where condition ${\bf H_2}$ holds, and we omit its consideration.

In this paper, we get, under conditions  ${\bf A}$ --  ${\bf G}$ and ${\bf H}_i$ (for $i = 1, 2, 3$), asymptotic expansions for stationary 
probabilities,  as $\e \to 0$, 
\begin{equation} \label{sta1} 
\pi_i(\e) =   \sum_{l = 0}^{L} c_{i}[l]\e^l + o_{i}(\e^{L}), \ i \in \XX.
\end{equation}

Moreover, we shall show that the limiting stationary probabilities $\pi_i(0)  > 0, i \in \XX$, if condition ${\bf H_1}$ holds,  $\pi_0(0) =  1,  \pi_i(0) =  0, i \in \, _0\XX$,  if condition ${\bf H_2}$ holds,  and $\pi_0(0), \pi_N(0) > 0, \pi_0(0) + \pi_N(0) = 1,   \pi_i(0) =  0, i \in  \, _{0, N}\XX$,  if condition ${\bf H_3}$ holds.

This implies that there is sense to consider so-called conditional quasi-stationary stationary probabilities, which are defined as, 
\begin{equation}\label{quasiF2}
\tilde{\pi}_i(\e) = \frac{\pi_i(\e)}{1 - \pi_0(\e)} = \frac{\pi_i(\e)}{\sum_{j \in \, _0\XX} \pi_j(\e)}, \ i \in \, _0\XX,
\end{equation} 
in the case where condition ${\bf H_2}$ holds, or as, 
\begin{equation}\label{quasiF3}
\hat{\pi}_i(\e) = \frac{\pi_i(\e)}{1 - \pi_0(\e) - \pi_N(\e)} = \frac{\pi_i(\e)}{\sum_{j \in  \, _{0, N}\XX} \pi_j(\e)}, \  \ i \in \ _{0, N}\XX, 
\end{equation}
in the case where condition ${\bf H_3}$ holds.

We also get, under conditions  ${\bf A}$ --  ${\bf G}$ and ${\bf H_2}$, asymptotic expansions for conditional quasi-stationary probabilities, 
\begin{equation} \label{sta2} 
\tilde{\pi}_i(\e) =   \sum_{l = 0}^{L} \tilde{c}_{i}[l]\e^l + \tilde{o}_{i}(\e^{L}), \ i \in \, _o\XX, 
\end{equation} 
and, under conditions  ${\bf A}$ --  ${\bf G}$ and ${\bf H_3}$, asymptotic expansions for conditional quasi-stationary probabilities,  
\begin{equation} \label{sta3} 
\hat{\pi}_i(\e) =   \sum_{l = 0}^{L} \hat{c}_{i}[l]\e^l + \hat{o}_{i}(\e^{L}), \ i \in \XX.  
\end{equation} 

The coefficients in the above asymptotic expansions are given by explicit recurrent formulas via coefficients in asymptotic expansions given in initial perturbation 
conditions ${\bf D}_L$  and ${\bf E}_L$. 

The first coefficients  $\pi_i(\e)  = c_{i}[0], \tilde{\pi}_i(0) = \tilde{c}_{i}[0]$ and $\hat{\pi}_i(0) = \hat{c}_{i}[0]$ describe the asymptotic behavior of stationary and quasi-stationary probabilities and their continuity properties with respect to small perturbations of  transition characteristics of the corresponding semi-Markov processes.

The second coefficients $c_{i}[1], \tilde{c}_{i}[1]$ and $\hat{c}_{i}[1]$ determine sensitivity of stationary and quasi-stationary probabilities with respect to small perturbations  of transition characteristics.

The high order coefficients can be useful and improve accuracy of the corresponding numerical computations based on the corresponding asymptotic expansions,  especially, for the models, where actual values of perturbation parameter 
$\e$ are not small enough to neglect the high order terms in the corresponding asymptotic expansions. 

We also would like to comment the use of the term ``conditional quasi-stationary probability''  for quantities defined in relations (\ref{quasiF2}) and (\ref{quasiF3}). As a matter of fact, the term ``quasi-stationary probability (distribution)'' is traditionally used  for limits of probabilities,  
\beq
q_j(\e) = \PP_i\{ \eta^{(\e)}(t) = j / \eta^{(\e)}(s) \notin A, 0 \leq s \leq t \}
\lb{quasi}
\eeq
as  $t \to \infty$, where  $A$ is some special subsets of $\XX$. We refer to the book by Gyllenberg and Silvestrov (2008), where one can find results concerned asymptotic expansions 
for such quasi-stationary distributions for perturbed semi-Markov process. A detailed presentation of results concerned  quasi-stationary distributions and comprehensive bibliographies of works in this are can be found in the above book as well as in the recent books by 
N{\aa}sell (2011) and Collet, Mart\'{\i}nez and San Mart\'{\i}n (2013).  \\

{\bf 3. Examples of perturbed birth-death type models} \\

In this section we consider a number of applications of perturbed semi-Markov processes of birth-death type.
We will assume that the conditional distribution of the transition time $\kappa_n^{(\e)}$ between two jumps only depends on the state $i$ before the jump, not the state $j$ to which a jump occurs, i.e.\ $F^{(\e)}_{ij}(t)=F^{(\e)}_{i}(t)$. Formula (\ref{embedas}) then simplifies to 
\begin{equation}\label{Qij}
Q^{(\e)}_{ij}(t) = F^{(\e)}_{i}(t)p_{ij}(\e),
\end{equation}
for $t \geq 0$ and $i, j \in \XX$.
 
An important special case of (\ref{Qij}) is a geometrically distributed transition time
\begin{equation}\label{FGeo}
F_i^{(\e)} \sim \mbox{Ge}\left[\lambda_i(\e)\right] \Longrightarrow 
Q^{(\e)}_{ij}(t) = \left\{1-\left[1-\lambda_i(\e)\right]^{[t]}\right\}p_{ij}(\e)
\end{equation}
for $0<\lambda_i(\e)\le 1$, and with $[t]$ the integer part of $t$. In particular, $\lambda_i(\e)=1$ corresponds to $\eta^{(\e)}(t) = \eta^{(\e)}_{[t]}, t \geq 0$; a discrete time homogeneous Markov chain embedded in continuous time.

A second special case of (\ref{Qij}) is exponentially distributed transition times
\begin{equation}\label{FExp}
F_i^{(\e)} \sim \mbox{Exp}\left[\lambda_i(\e)\right] \Longrightarrow
Q^{(\e)}_{ij}(t) = \left[1 - e^{- \lambda_i(\e) t}\right] p_{ij}(\e)
\end{equation}
for $0<\lambda_i(\e)<\infty$. Then $\eta^{(\e)}(t), t \geq 0$ is a continuous time homogeneous Markov chain. 

For a semi-Markov process of birth-death type, we refer to $l(i)=i-1+I(i=0)$ and $u(i)=i+1-I(i=N)$ as the lower and upper state to which a transition from $i$ is possible. In all the examples below the transition time distribution is given by (\ref{FGeo}) or (\ref{FExp}). For both of these models, it is convenient to introduce 
\begin{equation}\label{laDef}
\begin{array}{rcl}
\lambda_{i,+}(\e) &=& \lambda_{i,u(i)}(\e),\\
\lambda_{i,-}(\e) &=& \lambda_{i,l(i)}(\e),\\
\lambda_i{(\e)} &=& \lambda_{i,-}{(\e)} + \lambda_{i,+}{(\e)},
\end{array}
\end{equation}
so that the transition probabilities of the imbedded  Markov chain satisfy
\begin{equation}\label{pDef}
p_{i,+}(\e) = 1-p_{i,-}(\e) = \frac{\lambda_{i,+}{(\e)}}{\lambda_i{(\e)}},
\end{equation}
in accordance with requirement b) below equation (\ref{emsemi}). 
Here $\lambda_{ij}(\e)$ is either the probability by which a transition from $i$ to $j$ occurs for a discrete transition time (\ref{FGeo}), or the rate of transition from $i$ to $j$, for a continuous transition time (\ref{FExp}). The stationary distribution (\ref{statik}) then has the exact expression 
\begin{equation}\label{pii}
\pi_i(\e) \propto \left\{\begin{array}{ll}
1, & i=0,\\
\frac{\lambda_{0,+}(\e)\cdot \ldots \cdot \lambda_{i-1,+}(\e)}
{\lambda_{1,-}(\e)\cdot \ldots \cdot \lambda_{i,-}(\e)}, & i=1,\ldots,N,
\end{array}\right.
\end{equation}
for $0<\e\le\e_0$, both in discrete and continuous time (\ref{FGeo})-(\ref{FExp}), with a proportionality constant chosen so that $\sum_{i=0}^N \pi_i(\e)=1$. 

Suppose $\lambda_{i,+}(\e)$ and $\lambda_{i,+}(\e)$ admit series 
\beq
\lambda_{i, \pm}(\e) =  \sum_{l = 0}^{L_{i,\pm}} g_{i, \pm}[l] \e^l + o_{i, \pm}(\e^{L + l_{i, \pm}}) 
\lb{lambdaexp}
\eeq
for $\e \in (0, \e_0]$. 
%Assume also that 
%$$
%L_i=\min\{l;\, 0\le l \le \min(L_{i,-},L_{i,+}), g_{i,-}[l]+g_{i,+}[l]>0\}
%$$ 
%exists. 
From equations (\ref{embedadam})-(\ref{embedadamad}), (\ref{FGeo})-(\ref{FExp}) and (\ref{pDef}) we deduce that
\beq
e_{i,\pm}(\e) = \frac{1}{\lambda_i(\e)}\cdot \frac{\lambda_{i,\pm}(\e)}{\lambda_i(\e)}.
\lb{eipm}
\eeq
Inserting (\ref{lambdaexp}) into (\ref{eipm}), we find that 
\beq
g_{i,-}[0]+g_{i,+}[0]>0
\lb{gsum}
\eeq 
must hold for all $i\in\XX$ in order for the series expansion of $e_{i,\pm}(\e)$ to satisfy Condition $\mbox{\bf E}_L$. It therefore follows from (\ref{pDef}) that $p_{i,\pm}(\e)$ will satisfy perturbation condition $\mbox{\bf D}_L$, with $L+l_i = \min(L_{i,-},L_{i,+})$, and
\beq
a_{i,\pm}[0] = \frac{g_{i,\pm}[0]}{g_{i,-}[0]+g_{i,+}[0]}.
\lb{apm}
\eeq
Because of (\ref{gsum}) and (\ref{apm}), we can rephrase the three perturbation scenarios $\mathop{\mbox{\bf H}}_1$-$\mathop{\mbox{\bf H}}_3$ of Subsection 2.2 as
\beq
\begin{array}{ll}
\mathop{\mbox{\bf H}}_1: & g_{0,+}[0]>0, g_{N,-}[0]>0,\\  
\mathop{\mbox{\bf H}}_2: & g_{0,+}[0]=0, g_{N,-}[0]>0,\\  
\mathop{\mbox{\bf H}}_3: & g_{0,+}[0]=0, g_{N,-}[0]=0.
\end{array}
\lb{H1H3}
\eeq
This will be utilized in Subsections 3.1-3.3 in order to characterize the various perturbed models that we propose. 
Under $\mathop{\mbox{\bf H}}_2$, the exact expression for the conditional quasi stationary distribution (\ref{quasiF2}) is readily obtained from (\ref{pii}). It equals 
\begin{equation}\label{tpii}
\tilde{\pi}_i(\e) \propto \frac{\lambda_{1,+}(\e)\cdot \ldots \cdot \lambda_{i-1,+}(\e)}
{\lambda_{1,-}(\e)\cdot \ldots \cdot \lambda_{i,-}(\e)}
\end{equation}
for $i=1,\ldots,N$ and $0<\e\le\e_0$, with the numerator equal to 1 when $i=1$, 
and a proportionality constant chosen so that $\sum_{i=1}^N \tilde{\pi}_i(\e)=1$. As $\e\to 0$, this expression converges to
\begin{equation}\label{tpii0}
\tilde{\pi}_i(0) \propto \frac{\lambda_{1,+}(0)\cdot \ldots \cdot \lambda_{i-1,+}(0)}
{\lambda_{1,-}(0)\cdot \ldots \cdot \lambda_{i,-}(0)}.
\end{equation}
If Scenario $\mathop{\mbox{\bf H}}_3$ holds, we find analogously that the conditional quasi stationary distribution (\ref{quasiF3}) is given by   
\beq
\hat{\pi}_i(\e) \propto \frac{\lambda_{0,+}(\e)\cdot \ldots \cdot \lambda_{i-1,+}(\e)}
{\lambda_{1,-}(\e)\cdot \ldots \cdot \lambda_{i,-}(\e)}
\lb{pihat}
\eeq
for $i=1,\ldots,N-1$, with a limit 
\begin{equation}
\hat{\pi}_i(0) \propto \frac{\lambda_{1,+}(0)\cdot \ldots \cdot \lambda_{i-1,+}(0)}
{\lambda_{1,-}(0)\cdot \ldots \cdot \lambda_{i,-}(0)}.
\label{CondQuasiLimit}
\end{equation} \vspace{2mm}

{\bf 3.1.  Perturbed population dynamics models} \\

Let $N$ denote the maximal size of a population, and let $\eta^{(\e)}(t)$ be its size at time $t$.
In order to model the dynamics of the population, we introduce births, deaths and immigration from outside, according to a parametric model with
\begin{equation}\label{la+PopDyn}
\lambda_{i,+}(\e) = \lambda i \left[1-\alpha_1 \left(\frac{i}{N}\right)^{\theta_1}\right] 
+ \nu \left[ 1-\left(\frac{i}{N}\right)^{\theta_2}\right]
\end{equation}
and 
\begin{equation}\label{la-PopDyn}
\lambda_{i,-}(\e) = \mu i \left[1+\alpha_2\left(\frac{i}{N}\right)^{\theta_3}\right].
% + \mbox{I}(i=0).
\end{equation}
For a small population ($i\ll N$), we interpret the three parameters $\lambda>0$, $\mu>0$ and $\nu>0$ as a birth rate per individual, a death rate per individual and an immigration rate, whereas $\alpha_k,\theta_k$ are density regulation parameters that model decreased birth/immigration and increased death for a population close to its maximal size. They satisfy $\theta_k>0$, $\alpha_1\le 1$ and $\alpha_1,\alpha_2\ge 0$, where the last inequality is strict for at least one of $\alpha_1$ and $\alpha_2$. A more general model would allow birth, death and immigration rates to vary non-parametrically with $i$.  
%The extra term $\mbox{I}(i=0)$ of (\ref{la-PopDyn}) is included for convenience, in order to satisfy condition {\bf F(a)} when $i=0$.  

The expected growth rate of the population, when $0<i<N$, is 
$$
\begin{array}{l}
\EE\left[\eta^{(\e)}(t+\Delta t)-\eta^{(\e)}(t)|\eta^{(\e)}(t)=i\right] = \Delta t \left[\lambda_{i,+}(\e)-\lambda_{i,-}(\e)\right]\\
\,\,\,\,\,\,\,\,\,\, = \Delta t \left\{ \lambda i \left[1-\alpha_1 \left(\frac{i}{N}\right)^{\theta_1}\right]
+ \nu \left[ 1-\left(\frac{i}{N}\right)^{\theta_2}\right] - \mu i \left[1+\alpha_2\left(\frac{i}{N}\right)^{\theta_3}\right]\right\},
\end{array}
$$
where $\Delta t=1$ in discrete time (\ref{FGeo}), and $\Delta t >0$ is infinitesimal in continuous time (\ref{FExp}).
When $\theta_1=\theta_2=\theta_3=\theta$, this expression simplifies to 
\begin{equation}\label{EetaPop}
\begin{array}{l}
\EE\left[\eta^{(\e)}(t+\Delta t)-\eta^{(\e)}(t)|\eta^{(\e)}(t)=i\right] \\
\,\,\,\,\,\, = \Delta t \left\{ \lambda i \left[1-\alpha_1 \left(\frac{i}{N}\right)^{\theta}\right] 
- \mu i \left[1+\alpha_2\left(\frac{i}{N}\right)^{\theta}\right]\right\}
+ \nu \left[ 1-\left(\frac{i}{N}\right)^{\theta}\right],\\
\,\,\,\,\,\, = 
\Delta t \cdot \left\{\begin{array}{ll} 
(\lambda-\mu)i\left[1- \frac{\alpha_1\lambda + \alpha_2\mu}{\lambda-\mu}\left(\frac{i}{N}\right)^\theta \right] 
+ \nu \left[ 1-\left(\frac{i}{N}\right)^{\theta}\right], & \lambda\ne\mu\\
-\mu i (\alpha_1+\alpha_2)\left(\frac{i}{N}\right)^\theta 
+ \nu \left[ 1-\left(\frac{i}{N}\right)^{\theta}\right], & \lambda=\mu.
\end{array}\right.
\end{array}
\end{equation}
We will consider two perturbation scenarios. The first one has 
\begin{equation}\label{PopDynPert1}
\nu = \nu(\e) = \e,
\end{equation}
whereas all other parameters are kept fixed, not depending on $\e$. It is also possible to consider more general nonlinear functions $\nu(\e)$, but this will hardly add more insight to how immigration affects population dynamics. The unperturbed $\e=0$ model corresponds to an isolated population that only increases through birth events. Since $\lambda_{0,-}(\e)=0$ and $\lambda_{0,+}(\e)=\e$, it follows that $g_{0,-}[0]=g_{0,+}[0]=0$, and therefore formula (\ref{gsum}) is violated for $i=0$.
But the properties of $\eta^{(\e)}$ remain the same if we put $\lambda_{0,-}(\e)=1$ instead. With this modification,
formula (\ref{H1H3}) implies that Condition $\mathop{\mbox{\bf H}}_2$ of Subsection 2.2 holds, 
and hence the $\e\to 0$ limit of the stationary distribution in (\ref{statik}) and (\ref{pii}) is concentrated at state 0 ($\pi_0(0)=1$). For small $\e$, we can think of a population that resides on an island and faces subsequent extinction and recolonization events. After the population temporarily dies out, the island occasionally receives new immigrants at rate or probability $\e$. 
Let $\tau_0^{(\e)}$ be the time it takes for the population to get temporarily extinct again, after an immigrant has entered and empty island. It then follows from a slight modificaton of equation (\ref{stationary}) in Subsection 4.2 or Theorem \ref{theoremH2} of Subsection 5.2 (which disregards transitions $0\to 0$) and the relation $\lambda_{0,+}(\e)=\e$, that  
a first order expansion of the probability that the island is empty at stationarity, is
\begin{equation}\label{NonExtinctProb}
\pi_0(\e) = \frac{1/\lambda_{0,+}(\e)}{1/\lambda_{0,+}(\e) + \EE_1(\tau_0^{(\e)})} = \frac{1/\e}{1/\e + E_{10}(\e)} 
= 1-E_{10}(\e)\e + o(\e).
\end{equation}
This expansion is accurate when the perturbation parameter is small ($\e\ll 1/E_{10}(\e)$), otherwise higher order terms (\ref{statar}) are needed. The value of $E_{10}(\e)$ will be highly dependent on the value of the basic reproduction number $R_0=\lambda/\mu$. When $R_0>1$, the expected time to extinction will be very large, and $\pi_0(\e)$ will be close to 0 for all but very small $\e$. On the other hand, (\ref{NonExtinctProb}) is accurate for a larger range of $\e$ when $R_0<1$, since $E_{10}(\e)$ is then small.     

In order to find useful approximations of the conditional quasi stationary distribution $\tilde{\pi}_i(\e)$ in (\ref{tpii}), we will distinguish between whether $R_0$ is larger than or smaller than 1. When $R_0> 1$, or equivalently $\lambda>\mu$, we can rewrite (\ref{EetaPop}) as
\begin{equation}\label{ExpGrowthPopDyn}
\EE\left[\eta^{(\e)}(t+\Delta t)-\eta^{(\e)}(t)|\eta^{(\e)}(t)=i\right]
= \Delta t \cdot N m\left(\frac{i}{N}\right),
\end{equation}
where 
\beq
m(x) = rx + \frac{\e}{N} - \left[rx(0)^{-\theta}\cdot x + \frac{\e}{N}\right]x^\theta
\lb{mx}
\eeq
is a rescaled mean function of the drift, 
$r = \mu(R_0-1)$ is the intrinsic growth rate, or growth rate per capita, of a small population without immigration ($\e=0$), and
$$
x(0) = \left(\frac{R_0-1}{\alpha_1R_0+\alpha_2}\right)^{1/\theta}.
$$
We assume that $\alpha_1$ and $\alpha_2$ are large enough so that $x(0)<1$. A sufficient condition for this
is $\alpha_1+\alpha_2=1$. The carrying capacity $K(\e)=Nx(\e)$ of the environment is the value of $i$ such that the right hand side of (\ref{ExpGrowthPopDyn}) equals zero. We can write $x=x(\e)$ as the unique solution of $m(x)=0$, or equivalently
$$
x^\theta = \frac{rx+\e N^{-1}}{rx(0)^{-\theta}x+\e N^{-1}},
$$
with $x(\e)\searrow x(0)$ as $\e\to 0$. The conditional quasi stationarity distribution (\ref{tpii}) will be centered around $K(\e)$. In order to find a good approximation of this distribution we look at the second moment
$$
\begin{array}{l}
\EE\left\{\left[\eta^{(\e)}(t+\Delta t)-\eta^{(\e)}(t)\right]^2|\eta^{(\e)}(t)=i\right\}
= \Delta t \left[\lambda_{i,+}(\e)+\lambda_{i,-}(\e)\right]\\
\,\,\,\,\,\,\,\,\,\,\,\, = \Delta t \cdot N v\left(\frac{i}{N}\right),
\end{array}
$$
of the drift of $\eta^{(\e)}$, with 
\beq
v(x) = \lambda x(1-\alpha_1 x^\theta) + \frac{\e}{N}(1-x^\theta) + \mu x(1+\alpha_2 x^\theta).
\lb{vx}
\eeq
When $N$ is large we may approximate the conditional quasi stationary distribution
\beq
\begin{array}{rcl}
\tilde{\pi}_i(\e) &\approx & \int_{i-}^{i+} f^{(\e)}(k)dk\\
&=& \int_{i-}^{i+} f^{(0)}(k)dk + \int_{i-}^{i+} \left.\frac{df^{(\e)}(k)}{d\e}\right|_{\e=0} dk \cdot \e + o(\e),
\end{array}
\lb{tpiSens}
\eeq
by integrating a density function $f^{(\e)}$ on $[0,N]$ between $i_-=\max(0,i-1/2)$ and $i_+=\min(N,i+1/2)$. 
This density function can be found through a diffusion argument as the stationary density 
\beq
\begin{array}{rcl}
f^{(\e)}(k) &\propto & \frac{1}{Nv(\frac{k}{N})}\exp\left(2\int_{K(\e)}^k \frac{Nm(\frac{y}{N})}{Nv(\frac{y}{N})}dy\right)\\
&\propto & \frac{1}{v(\frac{k}{N})}\exp\left(2\int_{K(\e)}^k \frac{m(\frac{y}{N})}{v(\frac{y}{N})}dy\right)
\end{array}
\lb{fek}
\eeq
of Kolmogorov's forward equation, with a proportionality constant chosen so that $\int_0^N f^{(\e)}(k)dk = 1$ (see for instance Chapter 9 of Crow and Kimura, 1970). A substitution of variables $x=y/N$ in (\ref{fek}), and a Taylor expansion of $m(x)$ around $x(\e)$ reveals that 
the diffusion density is approximately normally distributed 
\beq
f^{(\e)} \sim N\left(K(\e),N\frac{v\left[x(\e)\right]}{2|m^\prime\left[x(\e)\right]|}\right).
\lb{Normal}
\eeq
Expansion (\ref{tpiSens}) is valid for small migration rates $\e$, and its linear term quantifies how sensitive the conditional quasi stationary distribution is to a small amount of immigration. 

It follows from (\ref{ExpGrowthPopDyn}) that the expected population size 
$$
\EE\left[\eta^{(0)}(t+\Delta t)-\eta^{(0)}(t)|\eta^{(0)}(t)=i\right] = \Delta t \cdot ri \left[1 -\left(\frac{i}{K(0)}\right)^{\theta}\right]
$$
of the null model $\e=0$ follows a so called theta logistic model (Gilpin and Ayala, 1973), which is a special case of the generalized growth curve model in Tsoularis and Wallace (2002). The theta logistic model has  a carrying capacity $K(0)$ of the environment to accommodate new births. When $\theta=1$, we obtain the logistic growth model of Verhulst (1838). Pearl (1920) used such a curve to approximate population growth in the United States, and Feller (1939) introduced a stochastic version of the logistic model in terms of a Markov birth-death process (\ref{FExp}) in continuous time. Feller's approach has been extended for instance by Kendall (1949), Whittle (1957), and N{\aa}sell (2001,2003). In particular, N{\aa}sell studied the quasi stationary distribution (\ref{quasi}) of $\eta^{(\e)}$, with $A=\{1,\ldots,N\}$. In this paper the previously studied population growth models are generalized in two directions; we consider semi-Markov processes and allow for theta logistic expected growth.    

When $0<R_0<1$, or equivalently $0<\lambda<\mu$, we rewrite (\ref{EetaPop}) as
\begin{equation}\label{ExpDecrPopDyn}
\begin{array}{l}
\EE\left[\eta^{(\e)}(t+\Delta t)-\eta^{(\e)}(t)|\eta^{(\e)}(t)=i\right]\\
\,\,\,\,\,\,\,\, = \Delta t \cdot \left[\nu-ri -(\nu+ri\tilde{x}^{-\theta})\left(\frac{i}{N}\right)^\theta\right],
\end{array}
\end{equation}
where $r=(1-R_0)\mu$ quantifies per capita decrease for a small population without immigration, and 
$\tilde{x}= \left[(1-R_0)/(\alpha_1R_0+\alpha_2)\right]^{1/\theta}$ is the fraction of the maximal population size at which the per capita decrease of an isolated $\e=0$ population has doubled to $2r$. For large $N$, we can neglect all $O(N^{-\theta})$ terms, and it follows from (\ref{tpii0}) that
$$
\tilde{\pi}_i(\e) \approx \frac{1}{\log(1-R_0)}\cdot\frac{R_0^i}{i} + \tilde{c}_i[1]\e + o(\e),
$$
for $i=1,\ldots,N$. 

A second perturbation scenario has a birth rate
\begin{equation}\label{PopDynPert2}
\lambda = \lambda(\e) = \e
\end{equation}
that equals $\e$, whereas all other parameters are kept fixed, not depending on $\e$. Again, more general nonlinear functions $\lambda(\e)$ can be studied, but for simplicity assume that (\ref{PopDynPert2}) holds. The unperturbed $\e=0$ model corresponds to a sink population that only increases through immigration. In view of (\ref{H1H3}), it satisfies Condition $\mathop{\mbox{\bf H}}_1$ of Subsection 2.2. Suppose $N$ is large. If $\nu =o(N)$, it follows from (\ref{pii}) that the stationary distribution for small values of $\e$ is well approximated by 
$$
\pi_i(\e) \approx \frac{(\nu/\mu)^i}{i!}e^{-\nu/\mu} + c_i[1]\e + o(\e)
$$
for $i=0,\ldots,N$, a Poisson distribution with mean $\nu/\mu$, corrupted by a sensitivity term $c_i[1]\e$ due to births. If $\nu=VN$, the carrying capacity of the environment is $K(\e)=Nx(\e)$, where $x=x(\e)$ is the value of $i/N$ in (\ref{ExpDecrPopDyn}) such that the right hand side vanishes, i.e.\ the unique solution of the equation
$$
rx +V x^\theta+r\tilde{x}^{-\theta}x^{\theta+1} = V,
$$
with $r=r(\e)=\mu-\e$. The stationary distribution (\ref{pii}) is well approximated by a discretized normal distribution (\ref{tpiSens})-(\ref{Normal}), but with a mean drift function $m(x)$ obtained from (\ref{ExpDecrPopDyn}), and a variance function $v(x)$ derived similarly.  \\

{\bf 3.2. Perturbed epidemic models} \\

In order to model an epidemic in a population of size $N$, we let $\eta^{(\e)}(t)$ refer to the number of infected individuals at time $t$, whereas the remaining $N-\eta^{(\e)}(t)$ are susceptible. We assume that   
\begin{equation}\label{la+Epid}
\lambda_{i,+}(\e) = \lambda i \left(1-\frac{i}{N}\right) + \nu (N-i),
\end{equation}
and 
\begin{equation}\label{la-Epid}
\lambda_{i,-}(\e) = \mu i,
% + \mbox{I}(i=0),
\end{equation}
where $\lambda(N-1)/N$ is the contact rate between each individual and other members of the population. 
This term may also be written as the product of the force of infection $\lambda i/N$ caused by $i$ infected individuals, and the number of susceptibles $N-i$.  
The second parameter $\nu$ is the contact rate between each individual and the group of infected ones outside of the population. The third parameter $\mu$ is the recovery rate per individual. It may also include a combined death and birth of an infected and susceptible individual. 
%The extra term $\mbox{I}(i=0)$ of (\ref{la-Epid}) guarantees that Condition {\bf F(a)} is satisfied when $i=0$.  
The model in (\ref{la+Epid})-(\ref{la-Epid}) is an SIS-epidemic, since infected individuals become susceptible after recovery. It is essentially a special case of (\ref{la+PopDyn})-(\ref{la-PopDyn}), with $\theta_1=\theta_2=\theta_3=1$, $\alpha_1=1$ and $\alpha_2=0$, although immigration is parametrized differently in (\ref{la+PopDyn}) and (\ref{la+Epid}).  

Assume that the external contact rate
\begin{equation}
\nu = \nu(\e) = \e
\label{ExtRate}
\end{equation}
equals the perturbation parameter, whereas all other parameters are kept fixed, not depending on $\e$. The unperturbed $\e=0$ model refers to an isolated population without external contagion. Sooner or later the epidemic will then die out and reach the only absorbing state 0. This corresponds to Condition $\mathop{\mbox{\bf H}}_2$ of Subsection 2.2. 

Weiss and Dishon (1971) first formulated the SIS-model as a continuous time birth-death Markov process (\ref{FExp}) without immigration ($\e=0$). It has since then been extended in a number of directions, see for instance Cavender (1978), Kryscio and Lef{\'e}vre (1989), Jacquez and O'Neill (1991), Jacquez and Simon (1993), N{\aa}sell (1996,1999) and Allen and Burgin (2000). The quasi stationary distribution (\ref{quasi}) of $\eta^{(\e)}$ is studied in several of these papers. 

In this work we generalize previously studied models of epidemic spread by treating discrete and continuous time in a unified manner through semi Markov processes.

The expected growth rate of the null model $\e=0$ satisfies
\begin{equation}\label{ExpGrowthEpid}
\EE\left[\eta^{(\e)}(t+\Delta t)-\eta^{(\e)}(t)|\eta^{(\e)}(t)=i\right] = \Delta t \cdot ri \left(1-\frac{i}{K(0)}\right),
\end{equation}
if $0<i<N$, when the basic reproduction ratio $R_0=\lambda/\mu$ exceeds 1. This implies that the expected number of infected individuals follows Verhult's logistic growth model, with intrinsic growth rate $r=\mu(R_0-1)$, and a carrying capacity $K(0)=N(1-R_0^{-1})$ of the environment. This is a special case of the theta logistic mean growth curve model (\ref{ExpGrowthPopDyn}), with $\theta=1$. When $R_0<1$, we similarly write the expected population decline as in (\ref{ExpDecrPopDyn}), with $\theta=1$. Since the SIS model is a particular case of the population dynamic models of Subsection 3.1, the stationary and conditional quasi stationary distributions are obtained in the same way. \\

{\bf 3.3.  Perturbed models of population genetics} \\

Let $N$ be a positive even integer, and consider a one-sex population with $N/2$ individuals, each one of which carries two copies of a certain gene. This gene exists in two variants (or alleles); $A_1$ and $A_2$. Let $\eta^{(\e)}(t)$ be the number of gene copies with allele $A_1$ at time $t$. Consequently, the remaining $N-\eta^{(\e)}(t)$ gene copies have the other allele $A_2$ at time $t$. At each moment $\zeta^{(\e)}_n$ of jump in (\ref{zetan}), a new gene copy replaces an existing one, so that 
\begin{equation}\label{etaJump}
\eta^{(\e)}(\zeta^{(\e)}_n) = \left\{ \begin{array}{ll}
\eta^{(\e)}(\zeta^{(\e)}_n-) + 1, & \mbox{if $A_1$ replaces $A_2$},\\
\eta^{(\e)}(\zeta^{(\e)}_n-), & \mbox{if $A_k$ replaces $A_k$},\\
\eta^{(\e)}(\zeta^{(\e)}_n-) - 1, & \mbox{if $A_2$ replaces $A_1$}.
\end{array}
\right. 
\end{equation}
In discrete time (\ref{FGeo}) we define $\lambda_{ij}(\e)$ as the probability that the number of $A_1$ alleles changes from $i$ to $j$ when a gene copy is replaced, at each time step. In continuous time (\ref{FExp}) we let $\lambda_{ij}(\e)$ be the rate at which the number of $A_1$ alleles changes from $i$ to $j$ when a gene copy replacement occurs. Let $x^{\ast\ast}$ refer to the probability that the new gene copy has variant $A_1$ when the fraction of $A_1$-alleles before replacement is $x=i/N$. We further assume that the removed gene copy is chosen randomly among all $N$ gene copies, with equal probabilities $1/N$, so that
\begin{equation}\label{laij}
\lambda_{ij}(\e) = \left\{\begin{array}{ll}
x^{\ast\ast}(1-x), & j=i+1,\\
(1-x^{\ast\ast})x, & j=i-1,\\
1-x^{\ast\ast}(1-x)-(1-x^{\ast\ast})x, & j=i.
\end{array}
\right.
\end{equation}
Notice that in order to make $\eta^{(\e)}(t)$ a semi-Markov process of birth-death type that satisfies (\ref{laDef})-(\ref{pDef}), we do not regard instances when the new gene copy replaces a gene copy with the same allele as a moment of jump, if the current number $i$ of $A_1$ alleles satisfies $0<i<N$. That is, the second line on the right hand side of (\ref{etaJump}) is only possible in a homogeneous population where all gene copies have the same allele $A_1$ or $A_2$, and therefore $\lambda_{ii}(\e)$ is not included in the probability or rate $\lambda_i(\e)$ to leave state $i$ in (\ref{laDef}), when $0<i<N$. 

The choice of $x^{\ast\ast}$ will determine the properties of the model. 
%It involves the selective fitness of pairs of two alleles as well as the mutation rates between them. 
The new gene copy is formed in two steps.
In the first step a pair of genes is drawn randomly with replacement, so that its genotype is $A_1A_1$, $A_1A_2$ and $A_2A_2$ with probabilities $x^2$, $2x(1-x)$ and $(1-x)^2$ respectively. Since the gene pair is drawn with replacement, this corresponds to a probability $2/N$ that the two genes originate from the same individual (self fertilization). A gene pair survives with probabilities proportional to $1+s_1$, $1$ and $1+s_2$ for these three genotypes, where $1+s_1\ge 0$ and $1+s_2\ge 0$ determine the fitnesses of genotypes $A_1A_1$ and $A_2A_2$ relative to that of genotype $A_1A_2$. This is repeated until a surviving gene pair appears, from which a gene copy is picked randomly. Consequently, the probability is 
\begin{equation}\label{xast}
x^\ast = \frac{1\cdot (1+s_1)x^2 + \frac{1}{2}\cdot 2x(1-x)}{(1+s_1)x^2 + 2x(1-x) + (1+s_2)(1-x)^2}
\end{equation}
that the chosen allele is $A_1$. In the second step, before the newly formed gene copy is put into the population, an $A_1$ allele mutates with probability $u_1=\PP(A_1\to A_2)$, and an $A_2$ allele with probability $u_2=\PP(A_2\to A_1)$. This implies that
\begin{equation}\label{xastast}
x^{\ast\ast} = (1-u_1)x^\ast + u_2(1-x^\ast).
\end{equation}
By inserting (\ref{xastast}) into (\ref{laij}), and (\ref{laij}) into (\ref{laDef})-(\ref{pDef}) we get a semi-Markov process of Moran type that describes the time dynamics of two alleles in a one-sex population in the presence of selection and mutation. A special case of it was originally introduced by Moran (1958a), and some of its properties can be found, for instance, in Karlin and McGregor (1962) and Durrett (2008). The model incorporates a number of different selection scenarios. A selectively neutral model corresponds to all three geneotypes having the same fitness ($s_1=s_2=0$), for directional selection, one of the two alleles is more fit than the other ($s_1<0<s_2$ or $s_1>0>s_2$), an underdeominant model has a heterozygous genotype $A_1A_2$ with smaller fitness than the two homozygous genotypes $A_1A_1$ and $A_2A_2$ ($s_1,s_2>0$), whereas overdominance or balancing selection means that the heterozygous genotype is the one with highest fitness ($s_1,s_2<0$).    

In continuous time (\ref{FExp}), the expected value of the Moran model satisfies a differential equation
\begin{equation}\label{Eeta}
\begin{array}{l}
\EE\left[\eta^{(\e)}(t+\Delta t)-\eta^{(\e)}(t)|\eta^{(\e)}(t)=Nx\right] = \Delta t \left[\lambda_{i,+}(\e)-\lambda_{i,-}(\e)\right]\\
\,\,\,\,\,\,\,\,\,\, = \Delta t \left[x^{\ast\ast}(1-x)-x(1-x^{\ast\ast})\right]\\
\,\,\,\,\,\,\,\,\,\, = \Delta t (x^{\ast\ast}-x)\\
\,\,\,\,\,\,\,\,\,\, = \Delta t \left[(1-u_1-u_2)x^{\ast}+u_2-x\right]\\
\,\,\,\,\,\,\,\,\,\, = \Delta t \left[(1-u_1-u_2)\frac{x+s_1x^2}{1+s_1x^2+s_2(1-x)^2}+u_2-x\right]\\
\,\,\,\,\,\,\,\,\,\, =: \Delta t \left[N^{-1}m(x) + o(N^{-1})\right],
\end{array}
\end{equation}
whenever $0<x<1$, with $\Delta t>0$ infinitesimal. The discrete time Moran model (\ref{FGeo}) also satisfies (\ref{Eeta}), interpreted as a difference equation, with $\Delta t=1$. In the last step of (\ref{Eeta}) we assumed that all mutation and selection parameters are inversely proportional to population size;
\begin{equation}\label{SmallPar}
\begin{array}{rcl}
u_1 &=& U_1/N,\\
u_2 &=& U_2/N,\\
s_1 &=& S_1/N,\\
s_2 &=& S_2/N,
\end{array}
\end{equation}
and introduced an infinitesimal drift function 
$$
m(x) = U_2(1-x)-U_1x + \left[(S_1+S_2)x-S_2\right]x(1-x).
$$
The corresponding infinitesimal variance function $v(x)=2x(1-x)$ follows similarly from (\ref{SmallPar}), according to 
\begin{equation}\label{Veta}
\begin{array}{l}
\VV\left[\eta^{(\e)}(t+\Delta t)|\eta^{(\e)}(t)=Nx\right] 
= \Delta t \left[\lambda_{i,+}(\e)+\lambda_{i,-}(\e)+O(N^{-1})\right]\\
\,\,\,\,\,\,\,\,\,\, = \Delta t \left[x^{\ast\ast}(1-x)+x(1-x^{\ast\ast})+O(N^{-1})\right]\\
\,\,\,\,\,\,\,\,\,\, = \Delta t \left[2x(1-x) + O(u_1+u_2+|s_1|+|s_2|)+O(N^{-1})\right]\\
\,\,\,\,\,\,\,\,\,\, =: \Delta t \left[v(x) +O(N^{-1})\right].
\end{array}
\end{equation}
The stationary distribution is found by first inserting (\ref{laij}) into (\ref{laDef}), and then (\ref{laDef}) into (\ref{pii}). However, for large $N$, it is often convenient to use a diffusion approximation
\beq
\pi_{i}(\e) \approx \int_{x_{i,-}}^{x_{i,+}} f^{(\e)}(x)dx,
\lb{piNx}
\eeq
by integrating the density function 
\beq
\begin{array}{rcl}
f^{(\e)}(x) &\propto &\frac{1}{v(x)}\exp\left(2\int_{1/2}^x \frac{m(y)}{v(y)}dy\right)\\
&\propto & (1-x)^{-1+U_1}x^{-1+U_2}\exp\left[\frac{1}{2}(S_1+S_2)x^2-S_2x\right]
\end{array}
\lb{fex}
\eeq
between $x_{i,-}=\max\left[0,(i-1/2)/N\right]$ and $x_{i,+}=\min\left[1,(i+1/2)/N\right]$. This density function is defined terms of the infinitesimal drift and variance functions $m(x)$ and $v(x)$ in (\ref{Eeta})-(\ref{Veta}), with a constant of proportionality chosen so that $\int f^{(\e)}(x)dx=1$. See for instance Chapter 9 of Crow and Kimura (1970) and Chapter 7 of Durrett (2008) for details. 
 
Assume that $N$ is fixed, whereas the perturbation parameter $\e$ varies. 
We let the two selection parameters $s_1$ and $s_2$, and hence also the rescaled selection parameters $S_1$ and $S_2$, be independent of $\e$, whereas the rescaled mutation parameters satisfy
\begin{equation}\label{ue}
\begin{array}{rcl}
U_1 &=& U_1(\e) = C_{1} + D_{1}\e,\\
U_2 &=& U_2(\e) = C_{2} + D_{2}\e,
\end{array}
\end{equation}
for some non-negative constants $C_{1},D_{1},C_{2},D_{2}$, where at least one of $D_1$ and $D_2$ is strictly positive.  
It follows from (\ref{laDef}), (\ref{H1H3}) and (\ref{laij}) that the values of $0\le C_{1},C_{2} < 1$ will determine the properties of the unperturbed $\e=0$ model, according to the three distinct scenarios
$$
\begin{array}{ll}
\mathop{\mbox{\bf H}}_1: & C_{1}>0, C_{2}>0,\\  
\mathop{\mbox{\bf H}}_2: & C_{1}>0, C_{2}=0,\\  
\mathop{\mbox{\bf H}}_3: & C_{1}=0, C_{2}=0.
\end{array}
$$
The null model $\e=0$ incorporates two-way mutations $A_1\to A_2$ and $A_2\to A_1$ for Perturbation scenario $\mathop{\mbox{\bf H}}_1$, with no absorbing state, it has one-way mutations $A_1\to A_2$ for Perturbation scenario $\mathop{\mbox{\bf H}}_2$, with $i=0$ as absorbing state, and no mutations for Perturbation scenario $\mathop{\mbox{\bf H}}_3$, with $i=0$ and $i=N$ as the two absorbing states. 

For $\mathop{\mbox{\bf H}}_1$ we find an approximate first order series expansion
$$
\pi_{i}(\e) \approx \int_{x_{i,-}}^{x_{i,+}} f^{(0)}(x)dx
+ \int_{x_{i,-}}^{x_{i,+}} \left. \frac{df^{(\e)}(x)}{d\e}\right|_{\e=0}dx \cdot \e + o(\e)
$$
(\ref{statar}) of the stationary distribution
by inserting (\ref{ue}) into (\ref{piNx})-(\ref{fex}). The null density $f^{(0)}(x)$ is defined 
(\ref{fex}), with $C_{1}$ and $C_{2}$ instead of $U_1$ and $U_2$. 
For a neutral model ($S_1=S_2=0$), the stationary null distribution is approximately beta with parameters $C_1$ and $C_2$, and expected value $C_2/(C_1+C_2)$. 
A model with $S_1>0>S_2$ corresponds to directional selection, with higher fitness for $A_1$ compared to $A_2$. 
It can be seen from (\ref{fex}) that the stationary null distribution is further skewed to the right compared that of a neutral model. A model with balancing selection or overdominance has negative $S_1$ and $S_2$, so that the heterozygous genotype $A_1A_2$ has a selective advantage. The stationary null distribution will then have a peak around $S_2/(S_1+S_2)$. On the other hand, for an underdominant model where $S_1$ and $S_2$ are both positive, the heterozygous genotype will have a selective disadvantage. Then $S_2/(S_1+S_2)$ functions as a repelling point of the stationary null distribution.  

For Scenario $\mathop{\mbox{\bf H}}_2$, the null model has one absorbing state 0. In analogy with (\ref{NonExtinctProb}), we find that the series expansion (\ref{statar}) of the stationary probability of no $A_1$ alleles in the population, is 
$$
\pi_0(\e) = 1-\EE_1(\tau_0^{(\e)})\cdot\frac{D_2\e}{N} + o(\e),
$$
for small values of the perturbation parameter, if $D_2>0$. Here $D_2\e/N$ is the probability that a mutation $A_2\to A_1$ occurs in a homogeneous $A_2$ population, and $\tau_0^{(\e)}$ is the time it takes for the $A_1$ allele to disappear again. 

The conditional quasi stationary distribution (\ref{quasiF2}) is found by inserting (\ref{laij}) into (\ref{laDef}), and then (\ref{laDef}) into (\ref{tpii})-(\ref{tpii0}). After some computations this leads to
\beq
\begin{array}{rcl}
\tilde{\pi}_i(\e) &\approx & \tilde{c}_1[0]i^{-1}\left(1-\frac{i-1}{N}\right)^{C_1-1}
\exp\left[\frac{1}{2}(S_1+S_2)\frac{i-1}{N}\frac{i}{N}-S_2\frac{i-1}{N}\right]\\
&+& \tilde{c}_1[1]\e + o(\e)
\end{array}
\lb{tpiiMoran}
\eeq
for $i=1,\ldots,N$, where $\tilde{c}_1[0]$ is chosen so that $\sum_{i=1}^N \tilde{\pi}_i(0)=1$, and $\tilde{c}_1[1]$ will additionally involve $D_1$ and $D_2$. If $D_2=0$, we have that $\pi_0(\e)=1$ for all $0<\e\le\e_0$, so that the conditional quasi stationary distribution (\ref{quasiF2}) is not well defined. However, the time to reach absorption is very large for small $U_1>0$. It is shown in H\"{o}ssjer et al.\ (2016) that $\eta^{(\e)}$ may be quasi-fixed for a long time at the other boundary point $i=N$, before eventual absorption at $i=0$ occurs. 
 
For Scenario $\mathop{\mbox{\bf H}}_3$, the null model is mutation free, and the asymptotic distribution 
$$
P_j(0) = \lim_{t\to\infty} \PP_i(\eta^{(0)}(t)=j)
$$
is supported on the two absorbing states ($j\in\{0,N\}$). For a neutral model ($s_1=s_2=0$), we have that 
\beq
P_N(0) = 1-P_0(0) = \frac{i}{N}.
\lb{pi0Neutr}
\eeq
A particular case of directional selection is  multiplicative fitness, with $1+s_1=(1+s_2)^{-1}$. It is mathematically simpler since selection operates directly on alleles, not on genotypes, with selective advantages $1$ and $1+s_2$ for $A_1$ and $A_2$. It follows for instance from Section 6.1 of Durrett (2008) that  
\beq
P_N(0) = 1-P_0(0) = \frac{1-(1+s_2)^i}{1-(1+s_2)^N}
\lb{pi0Mult}
\eeq
for multiplicative fitness. Notice that $P_0(0)$ and $P_N(0)$ will differ from 
$\pi_j(0)=\lim_{\e\to 0}\pi_0(\e)$ at the two boundaries. Indeed, by ergodicity (\ref{statik}) for each $\e>0$, 
the latter two probabilities are not functions of $i=\eta^{(0)}(0)$. 
From (\ref{piNx})-(\ref{fex}) we find that
\beq
\pi_N(0) = 1-\pi_0(0) \approx \frac{D_2}{\exp\left[-\frac{1}{2}(S_1-S_2)\right]D_1+D_2}.
\lb{piFix}
\eeq
Similarly as in (\ref{tpiiMoran}), we find after some computations that the conditional quasi stationary distribution (\ref{quasiF3}) and (\ref{pihat}) admits an approximate expansion 
\beq
\begin{array}{rcl}
\hat{\pi}_i(\e) &\approx & \hat{c}_1[0]i^{-1}\left(1-\frac{i-1}{N}\right)^{-1}
\exp\left[\frac{1}{2}(S_1+S_2)\frac{i-1}{N}\frac{i}{N}-S_2\frac{i-1}{N}\right]\\
&+& \hat{c}_1[1]\e + o(\e)
\end{array}
\lb{hpiiMoran}
\eeq
for $i=1,\ldots,N-1$, where $\hat{c}_1[0]$ is chosen so that $\sum_{i=1}^{N-1} \hat{\pi}_i(0)=1$, and $\hat{c}_1[1]$ will additionally involve $D_1$ and $D_2$. Notice that the limiting fixation probabilities in (\ref{piFix}) are functions of the mutation probability ratio $D_1/(D_1+D_2)$, but the limiting conditional quasi stationary distribution $\hat{\pi}_i(0)$ in (\ref{hpiiMoran}) does not involve any of $D_1$ or $D_2$.  \\

{\bf 4. Reduced semi-Markov processes} \\

In this section, we present an  time-space screening procedures of phase space reduction for perturbed semi-Markov processes and recurrent algorithms for computing expectations of hitting times and stationary and conditional quasi-stationary distributions for birth-death semi-Markov processes. \\

{\bf 4.1. An algorithm of excluding one state} \\

Let us assume that $N \geq 1$. Let us choose 
some state $r \in \XX$.  We can  consider the reduced phase space $_r\XX = \XX \setminus \{ r \}$,  with the state $r$ excluded from the 
phase space $\XX$. 

Let us define  the sequential moments of hitting the reduced space 
$_r\XX$,  by the embedded Markov chain $\eta^{(\e)}_n$,
\begin{equation}\label{sequen}
_r\xi^{(\e)}_n = \min(k > \, _r\xi^{(\e)}_{n-1}, \  \eta^{(\e)}_k \in \, _r\XX), \ n = 1, 2, \ldots, \  _r\xi^{(\e)}_0 = 0.
\end{equation}

Now, let us define the random sequence,
\begin{equation}\label{sequena}
(_r\eta^{(\e)}_n, \, _r\kappa^{(\e)}_n) = \left\{
\begin{array}{ll}
(\eta^{(\e)}_0, 0) & \ \text{for}  \ n = 0, \vspace{2mm} \\
(\eta^{(\e)}_{_r\xi^{(\e)}_n} \,, 
\sum_{k = \, _r\xi^{(\e)}_{n-1} +1}^{_r\xi^{(\e)}_n} \kappa^{(\e)}_k) & \ \text{for}  \ n = 1, 2, \ldots. 
\end{array}
\right.
\end{equation}

This sequence is also a Markov renewal process with a phase space $\XX \times [0, \infty)$,   the initial distribution
with probability $\bar{p}^{(\e)}$, and transition probabilities defined for 
$(i, s), (j, t) \in \, \XX \times [0, \infty)$, 
\begin{equation}\label{semir}
_rQ^{(\e)}_{ij}(t) = \PP \{ \, _r\eta^{(\e)}_{1} = j, \, _r\kappa^{(\e)}_{1} \leq t / \, _r\eta^{(\e)}_{0} = i, \, _r\kappa^{(\e)}_{0} = s \}. 
\end{equation}

Respectively, one can define the  transformed semi-Markov process, 
\begin{equation}\label{sepra}
_r\eta^{(\e)}(t) = \, _r\eta^{(\e)}_{_r\nu^{(\e)}(t)}, \ t \geq 0, 
\end{equation}
where $_r\nu^{(\e)}(t) = \max(n \geq 0: \, _r\zeta^{(\e)}_n \leq t)$ 
is a number of jumps at time interval  $[0, t]$, for $t \geq 0$, and $_r\zeta^{(\e)}_n = \, _r\kappa^{(\e)}_1 + \cdots + \, _r\kappa^{(\e)}_n, \ n = 0, 1, \ldots$ 
are sequential moments of jumps,  for the semi-Markov process $_r\eta^{(\e)}(t)$. 

The transition probabilities $_rQ^{(\e)}_{ij}(t)$ are expressed, for every $\e \in (0, \e_0]$,  via the transition probabilities $Q^{(\e)}_{ij}(t)$ by the 
following formula, for $t \geq 0, \, i, j \in \, \XX$, 
\begin{equation}\label{numjabo}
_rQ^{(\e)}_{ij}(t) = Q^{(\e)}_{ij}(t) + \sum_{n = 0}^\infty Q^{(\e)}_{ir}(t) * Q^{(\e) *n}_{rr}(t) * Q^{(\e)}_{rj}(t). 
\end{equation}

Here, symbol $*$ is used to denote the convolution of distribution functions (possibly improper),  and 
$Q^{(\e) *n}_{rr}(t)$ is the $n$ times convolution of the distribution function  $Q^{(\e)}_{rr}(t)$.  

Relation (\ref{numjabo}) directly implies, for every $\e \in (0, \e_0]$,  the following formula for transition probabilities of the reduced embedded Markov chain $_r\eta^{(\e)}_n$, for $i, j \in \XX$,
\begin{align}\label{transit}
_rp_{ij}(\e)  = \, _rQ^{(\e)}_{ij}(\infty)  & = p_{ij}(\e)  + \sum_{n = 0}^\infty p_{ir}(\e)  p_{rr}(\e)^n p_{rj}(\e)\nonumber \\
&  = p_{ij}(\e) + p_{ir}(\e) \frac{p_{rj}(\e)}{1 - p_{rr}(\e)}.
\end{align}

Note that condition ${\bf A}$ implies that  probabilities $p_{rr}(\e) \in [0, 1), \, r \in \XX, \, \e \in (0, \e_0]$.

The transition distributions for the Markov chain  $_r\eta^{(\e)}_n $  are concentrated, for every $\e \in (0, \e_0]$,  on the reduced phase space $_r\XX$, i.e., for every $i \in \XX$,
\begin{align}\label{sequenani}
\sum_{j \in \, _r\XX} \, _rp_{ij}(\e) & =  \sum_{j \in \, _r\XX} p_{ij}(\e) + p_{ir}(\e)  \sum_{j \in \, _r\XX} \frac{p_{rj}(\e) }{1 - p_{rr}(\e)} \nonumber \\
& =  \sum_{j \in \, _r\XX} p_{ij}(\e)  + p_{ir}(\e)   = 1.
\end{align}

If the initial distribution $\bar{p}^{(\e)}$ is concentrated on the phase space $_r\XX$, i.e., $p_r = 0$, then the random sequence  $(_r\eta^{(\e)}_n, \, _r\kappa^{(\e)}_n)$ 
can be considered as a Markov renewal process with the reduced  phase  $_r\XX \times [0, \infty)$, the initial distribution $_r\bar{p}^{(\e)} = \langle \, p_i^{(\e)} = 
\PP \{ _r\eta^{(\e)}_0 = i, \, _r\kappa^{(\e)}_0 = 0 \}  =  \PP \{   _r\eta^{(\e)}_0  = i \} , i \in \, _r\XX \rangle$ and transition probabilities $_rQ^{(\e)}_{ij}(t), t \geq 0, i, j \in \, _r\XX$.

If the initial distribution $\bar{p}^{(\e)}$ is not concentrated on the phase space $_r\XX$, i.e., $p_r > 0$, then the random sequence   $(_r\eta^{(\e)}_n, \, _r\kappa^{(\e)}_n)$  
can be interpreted as a Markov renewal process with so-called transition period.

As follows from the above remarks, the  semi-Markov process $_r\eta^{(\e)}(t)$ has transition probabilities $_rQ^{(\e)}_{ij}(t), t \geq 0, i, j \in \, \XX$ 
concentrated on the reduced phase space  $_r\XX$, which can be interpreted as the ``actual  reduced'' phase space of this semi-Markov process $_r\eta^{(\e)}(t)$. 

If the initial distribution $\bar{p}^{(\e)}$ is concentrated on the phase space $_r\XX$, then process $_r\eta^{(\e)}(t)$ can be considered as the semi-Markov process  
with the reduced  phase  $_r\XX$, the initial distribution $_r\bar{p}^{(\e)} = \langle \, _rp^{(\e)}_i = \PP \{_r\eta^{(\e)}(0)  = i \}, i \in \, _r\XX \rangle$ and transition probabilities 
$_rQ^{(\e)}_{ij}(t), t \geq 0,  i, j \in \, _r\XX$.

According to the above remarks, we can refer to the process $_r\eta^{(\e)}(t)$ as a reduced semi-Markov process. 

If the initial distribution $\bar{p}^{(\e)}$ is not concentrated on the phase space $_r\XX$, then the process $_r\eta^{(\e)}(t)$ can be interpreted 
as a reduced semi-Markov process with transition period. \\

{\bf 4.2. Expectations of hitting times for reduced semi-Markov  \\
\makebox[14mm]{} processes and stationary distributions} \\

 Let us now introduce hitting times for semi-Markov process $\eta^{(\e)}(t)$, 
Let us define hitting times, which are random variables given by the following relation, for $j \in \XX$,
\begin{equation}\label{hitt}
\tau^{(\e)}_j = \sum_{n = 1}^{\nu^{(\e)}_j} \kappa^{(\e)}_n,
\end{equation}
where $\nu^{(\e)}_j = \min(n \geq 1: \eta^{(\e)}_n = j)$.

Let us denote,
\begin{equation}\label{hitaex}
E_{ij}(\e) = \EE_i \tau^{(\e)}_j, \ i, j \in \XX.
\end{equation}

As is known, conditions ${\bf A}$ -- ${\bf C}$ imply that, for every $\e \in (0, \e_0]$, expectations of hitting times are finite, i.e, 
\begin{equation}\label{hitaexas}
0 < E_{ij}(\e) < \infty, \ i, j \in \XX. 
\end{equation}

We also denote by $_r\tau^{(\e)}_j$ the hitting time to the state $j \in \, _r\XX$ for the reduced semi-Markov process $_r\eta^{(\e)}(t)$. 

The following theorem, which proof can be found, for example,  in Silvestrov and Manca (2015), plays the key role in what follows. \vspace{1mm}
 \begin{theorem} Let conditions ${\bf A}$ --  ${\bf C}$ hold for semi-Markov processes $\eta^{(\e)}(t)$. Then, for any state $j \in \, _r\XX$, the first hitting times $\tau^{(\e)}_j$ and $_r\tau^{(\e)}_j$ to the state $j$, respectively, for semi-Markov processes $\eta^{(\e)}(t)$ and $_r\eta^{(\e)}(t)$, coincide, and, thus, the expectations of hitting times $E_{ij}(\e) = \EE_i \tau^{(\e)}_j = \EE_i \, _r\tau^{(\e)}_j$, for any $i \in \XX,  j \in \, _r\XX$ and $\e \in (0, \e_0]$.
\end{theorem}

In what follows, we use the following well known formula for stationary probabilities $\pi_i(\e), i \in \XX$, which takes place, for every $\e \in (0, \e_0]$, 
\begin{equation} \label{stationary}
\pi_i(\e) = \frac{e_i(\e)}{E_{ii}(\e)}, \ i \in \XX.
\end{equation} \vspace{2mm}

{\bf 4.3. Reduced semi-Markov processes of birth-death type} \\

Let the initial semi-Markov process  $\eta^{(\e)}(t)$ has a birth-death type. 

First, let consider the case, where the state $0$ is excluded from the phase space $\XX$. 
In this case, the reduced phase space $_0\XX = \ \{ 1, \ldots, N \}$. 

We assume that the initial distribution of the semi-Markov process $\eta^{(\e)}(t)$ 
is concentrated on the reduced phase space  $_0\XX$.

The transition probabilities of the reduced semi-Markov process $_0\eta^{(\e)}(t)$ has, for every $\e \in (0, \e_0]$,  the following form, for $t \geq 0$, 
\begin{equation}\label{emsemitak}
_0Q^{(\e)}_{ij}(t) = \left\{
\begin{array}{lll}
%_0F^{(\e)}_{0, +}(t)  &  \text{if} \ j =  1,  i = 0, \\ 
F^{(\e)}_{1, +}(t) p_{1, +}(\e)   &  \text{if} \ j =  2,  i = 1, \\ 
_0F^{(\e)}_{1, -}(t) p_{1, -}(\e) &  \text{if} \ j =  1,   i = 1, \\ 
F^{(\e)}_{i, \pm}(t) p_{i, \pm}(\e)    & \text{if} \ j = i \pm 1,   1 < i < N,   \\ 
F^{(\e)}_{N, \pm}(t) p_{N, \pm}(\e)     &  \text{if} \ j =  N - \frac{1 \mp  1}{2},  i = N, \\
0  & \text{otherwise},
\end{array}
\right.
\end{equation}
where
\begin{equation}\label{numjabomo}
_0F^{(\e)}_{1, -}(t) =  \sum_{n = 0}^\infty \, F^{(\e)}_{1, -}(t) * F^{(\e) *n}_{0, -}(t) * F^{(\e)}_{0, +}(t) \cdot p_{0, -}(\e)^n p_{0, +}(\e).
\end{equation} 

This relation implies, for every $\e \in (0, \e_0]$,   the following relation for transition probabilities of the reduced embedded Markov chain $_0\eta^{(\e)}_n$,
\begin{equation}\label{emsemitakaska}
_0p^{(\e)}_{ij} = \left\{
\begin{array}{lll}
% 1   &  \text{if} \ j =  1,  i = 0, \\  
_0p_{1, \pm}(\e)  = p_{1, \pm }(\e)  &  \text{if} \ j =  1 +  \frac{1 \pm  1}{2},  i = 1, \\  
_0p_{i, +}(\e)  = p_{i, \pm}(\e)    & \text{if} \ j = i \pm 1,   1 < i < N,   \\ 
_0p_{N, +}(\e)  = p_{N, \pm}(\e)     &  \text{if} \ j =  N - \frac{1 \mp  1}{2},  i = N, \\
0  & \text{otherwise}.
\end{array}
\right.
\end{equation}
and the following relation for transition expectations of the reduced embedded semi-Markov process $_0\eta^{(\e)}(t)$,
{\small
\begin{equation}\label{emsemitakas}
_0e^{(\e)}_{ij} = \left\{
\begin{array}{lll}
%_0e_{0, +}(\e) = e_{0}(\e) \cdot \frac{1}{p_{0, +}(\e)}   &  \text{if} \ j =  1,   i = 0, \\
_0e_{1, +}(\e) = e_{1, +}(\e)  &  \text{if} \ j =  2,   i = 1, \\
_0e_{1, -}(\e) =  e_{1, -}(\e) & \\ 
\quad + \, e_{0}(\e) \cdot \frac{p_{1, -}(\e)}{p_{0, +}(\e)}  &  \text{if} \ j =  1,  i = 1, \\    
_0e_{i, \pm}(\e) = e_{i, \pm}(\e)    & \text{if} \ j = i \pm 1, 1 < i < N,   \\ 
_0e_{N, \pm}(\e) =e_{N, \pm}(\e)     &  \text{if} \ j =  N - \frac{1 \mp  1}{2},  i = N, \\
0  & \text{otherwise}.
\end{array}
\right.
\end{equation}
}

Note that, by Theorem 1, the following relation takes place, for every $\e \in (0, \e_0]$ and $i,  j \in \, _0\XX$, 
\begin{equation}\label{hitrer}
\EE_i \tau_j^{(\e)} = \EE_i \, _0\tau_j^{(\e)}.  
\end{equation}

Second, let us consider the case, where the state $N$ is excluded from the phase space $\XX$. In this case, the reduced phase space $_N\XX = \ \{ 0, \ldots, N -1 \}$.

The transition probabilities of the reduced semi-Markov process $_N\eta^{(\e)}(t)$ has, for every $\e \in (0, \e_0]$,  the following form, for $t \geq 0$, 
\begin{equation}\label{emsemitakba}
_NQ^{(\e)}_{ij}(t) = \left\{
\begin{array}{lll}
F^{(\e)}_{0, \pm}(t) p_{0, \pm}(\e)   &  \text{if} \ j =  0 + \frac{1 \pm 1}{2},  i = 0, \\ 
F^{(\e)}_{i, \pm}(t) p_{i, \pm}(\e)    & \text{if} \ j = i \pm 1,   1 < i < N -1,   \\ 
_NF^{(\e)}_{1, +}(t) p_{N - 1, +}(\e) &  \text{if} \ j =  N - 1,   i = N - 1, \\ 
F^{(\e)}_{N -1, -}(t) p_{N - 1, -}(\e)     &  \text{if} \ j =  N - 2, i = N -1, \\
%_NF^{(\e)}_{N, -}(t)    &  \text{if} \ j =  N -1,  i = N, \\ 
0  & \text{otherwise},
\end{array}
\right.
\end{equation}
where
\begin{equation}\label{numjabomobato}
_NF^{(\e)}_{N -1, +}(t) = \sum_{n = 0}^\infty  F^{(\e)}_{N-1, +}(t) * F^{(\e) *n}_{N, +}(t) * F^{(\e)}_{N, -}(t) \cdot p_{N, +}(\e)^n p_{N, -}(\e).
\end{equation}

This relation implies, for every $\e \in (0, \e_0]$,   the following relation for transition probabilities of the reduced embedded Markov chain $_N\eta^{(\e)}_n$,
\begin{equation}\label{emsemitakaskaba}
_Np^{(\e)}_{ij} = \left\{
\begin{array}{lll}
_Np_{0, \pm}(\e)  = p_{0, \pm }(\e)  &  \text{if} \ j =  1 +  \frac{1 \pm  1}{2},  i = 0, \\  
_Np_{i, +}(\e)  = p_{i, \pm}(\e)    & \text{if} \ j = i \pm 1,   0 < i < N -1,   \\ 
_Np_{N -1, \pm}(\e)  = p_{N - 1, \pm}(\e)     &  \text{if} \ j =  N -1 - \frac{1 \mp  1}{2},   i = N-1, \\
%1 &  \text{if} \ j =  N -1,   i = N, \\
0  & \text{otherwise}.
\end{array}
\right.
\end{equation}
and the following relation for transition expectations of the reduced embedded semi-Markov process $_0\eta^{(\e)}(t)$,
{\small
\begin{equation}\label{emsemitakasba}
_Ne^{(\e)}_{ij} = \left\{
\begin{array}{lll}
_Ne_{1, \pm}(\e) = e_{1, \pm}(\e)  &  \text{if} \ j =  0 + \frac{1 \pm 1}{2}, i = 0, \\
_Ne_{i, +}(\e) = e_{i, \pm}(\e)    & \text{if} \ j = i \pm 1, 0 < i < N- 1,   \\ 
_Ne_{N-1, +}(\e) =  e_{N-1, +}(\e) & \\ 
\quad  +  e_{N}(\e) \cdot \frac{p_{N - 1, +}(\e)}{p_{N, -}(\e)}  &  \text{if} \ j =  N - 1, i =  N - 1, \\    
_Ne_{N-1, -}(\e) = e_{N-1, -}(\e)     &  \text{if} \ j =  N - 2,  i = N -1, \\
%_Ne_{N, -}(\e)  = e_{N}(\e) \cdot \frac{1}{p_{N, -}(\e)}  &  \text{if} \ j =  N - 1, i =  N, \\    
0  & \text{otherwise}.
\end{array}
\right.
\end{equation}
}

It is readily seen that in both cases, where $r = 0$ or $r = N$, the reduced semi-Markov process $_r\eta^{(\e)}(t)$ also has a birth-death type, with the phase space, respectively $_0\XX = \{ 1, \ldots, N \}$ or  $_N\XX = \{ 0, \ldots, N -1 \}$ and transition characteristics given, respectively, by relations (\ref{emsemitak}) -- (\ref{emsemitakas}) if $r = 0$, or by relations (\ref{emsemitakba}) -- (\ref{emsemitakasba})  if $r = N$. \

\pagebreak

{\bf 4.4. Sequential reduction of states for birth-death type  \\ 
\makebox[14mm]{} semi-Markov processes}.\\
 
Let us $0 \leq k \leq m \leq r \leq N$. The states $0, \ldots, k-1$ and  $N, \ldots, r + 1$ can be sequentially excluded from the phase space $\XX$ of the 
semi-Markov process  $\eta^{(\e)}(t)$. 

In order to describe this recurrent procedure, let us denote the resulted reduced semi-Markov process  as $_{\langle k, r \rangle}\eta^{(\e)}(t)$. This process has the reduced 
phase space $_{\langle k, r \rangle}\XX = \{ k, \ldots, r \}$.

In particular, the initial semi-Markov process  $\eta^{(\e)}(t) = \, _{\langle 0, N \rangle}\eta^{(\e)}(t)$.

The reduced semi-Markov process  $_{\langle k, r \rangle}\eta^{(\e)}(t)$ can be obtained by excluding of the state $k-1$ from the phase space $_{\langle k-1, j \rangle}\XX$ of the reduced semi-Markov process $_{\langle k -1, r \rangle}\eta^{(\e)}(t)$ or by excluding state $r +1$ from the phase space $_{\langle k, r +1 \rangle}\XX$ of the reduced semi-Markov 
process $_{\langle k, r +1 \rangle}\eta^{(\e)}(t)$.

The sequential excluding of the states $0, \ldots, k -1$ and  $N, \ldots, r + 1$ can be realized in an arbitrary order of choice one of these sequences and then by excluding the corresponding next  state from the chosen sequence. 

The simplest variants for the sequences of excluded states are $0, \ldots, k -1, N, \ldots, r + 1$  and $N, \ldots, r + 1, 0, \ldots, k -1$. 

The resulted reduced semi-Markov process $_{\langle k, r \rangle}\eta^{(\e)}(t)$ will be the same and it will have a birth-death type.

Here, we also accept  the reduced semi-Markov process $_{\langle m, m \rangle}\eta^{(\e)}(t)$ with one-state phase space $_{\langle m, m \rangle}\XX =  \{ m \}$  as a birth-death semi-Markov process.

This process has transition probability for the embedded Markov chain,
\begin{equation}
_{\langle m, m \rangle}p_{mm}^{(\e)} = \, _{\langle m, m \rangle}p_{m, +}(\e) + \, _{\langle m, m \rangle}p_{m, -}(\e) =  1,
\end{equation} 
and the semi-Markov transition probabilities,
\begin{align}
_{\langle m, m \rangle}Q_{mm}^{(\e)}(t) & = \, _{\langle m, m \rangle}F_{m, +}^{(\e)}(t) \, 
_{\langle m, m \rangle}p_{m, +}(\e) + \, _{\langle m, m \rangle}F_{m, -}^{(\e)}(t) \, _{\langle m, m \rangle}p_{m, -}(\e) \nonumber \\ 
& = \PP_i \{\tau_i^{\e)} \leq t \}.
\end{align}

The following relations, which are, in fact, variants of 
relations (\ref{emsemitakaska}), (\ref{emsemitakaskaba})
 and  (\ref{emsemitakas}), (\ref{emsemitakasba}), express transition probabilities $_{\langle k, r \rangle}p^{(\e)}_{ij}$ and expectations of transition times $_{\langle k, r \rangle}e^{(\e)}_{ij}$ for the reduced embedded semi-Markov process  $_{\langle k, r \rangle}\eta^{(\e)}(t)$, via the transition probabilities $_{\langle k-1, r \rangle}p^{(\e)}_{ij}$ and the expectations of transition times 
$_{\langle k-1, r \rangle}e^{(\e)}_{ij}$ for  the reduced embedded semi-Markov process $_{\langle k-1, r \rangle}\eta^{(\e)}(t)$, 
for $1 \leq k \leq r \leq N$ and, for every $\e \in (0, \e_0]$, 
{\small 
\begin{equation}\label{takaskao}
_{\langle k, r \rangle}p^{(\e)}_{ij} = \left\{
\begin{array}{lll}
_{\langle k, r \rangle}p_{k, \pm}(\e)  & = \, _{\langle k-1, r \rangle}p_{k, \pm }(\e)   \vspace{1mm} \\
& \quad  \text{if} \ j =  k +  \frac{1 \pm  1}{2},  i = k,  \vspace{2mm} \\  
_{\langle k, r \rangle}p_{i, \pm}(\e)  & = \, _{\langle k-1, r \rangle}p_{i, \pm}(\e)   \vspace{1mm}  \\
& \quad  \ \text{if} \ j = i \pm 1,   k < i < r,    \vspace{2mm} \\ 
_{\langle k, r \rangle}p_{r, \pm }(\e)  & = \, _{\langle k-1, r \rangle}p_{r, \pm}(\e)    \vspace{1mm}  \\
& \quad \  \text{if} \ j =  r - \frac{1 \mp  1}{2},  i = r,  \vspace{2mm}  \\
0  & \text{otherwise}, 
\end{array}
\right.
\end{equation}
}
and
{\small
\begin{equation}\label{takaso}
_{\langle k, r \rangle}e^{(\e)}_{ij} = \left\{
\begin{array}{lll}
_{\langle k, r \rangle}e_{k, +}(\e) & = \, _{\langle k-1, r \rangle}e_{k, +}(\e)  \vspace{1mm}  \\  
&  \text{if} \ j =  k+1,   i = k,  \vspace{2mm}  \\
_{\langle k, r \rangle}e_{k, -}(\e) & =  \, _{\langle k-1, r \rangle}e_{k, -}(\e)   \vspace{1mm}  \\ 
& + \, _{\langle k-1, r \rangle}e_{k-1}(\e) \cdot \frac{_{\langle k-1, r \rangle}p_{k, -}(\e)}{_{\langle k-1, r \rangle}p_{k-1, +}(\e)}   \vspace{1mm} \\ 
&  \text{if} \ j =  k,  i = k,  \vspace{2mm}  \\ 
_{\langle k, r \rangle}e_{i, \pm}(\e) & = \, _{\langle k-1, r \rangle}e_{i, \pm}(\e)   \vspace{1mm} \\ 
 & \text{if} \ j = i \pm 1, k < i < r,   \vspace{2mm}  \\ 
_{\langle k, r \rangle}e_{r, \pm}(\e) & = \, _{\langle k-1, r \rangle}e_{r, \pm }(\e)   \vspace{1mm}    \\
 &  \text{if} \ j =  r - \frac{1 \mp  1}{2},  i = r,  \vspace{2mm}  \\
0  & \text{otherwise},
\end{array}
\right.
\end{equation}
}
or,  via transition probabilities $_{\langle k, r +1 \rangle}p^{(\e)}_{ij}$ and expectations of transition times $_{\langle k, r +1 \rangle}e^{(\e)}_{ij}$ for  the reduced embedded semi-Markov process $_{\langle k, r +1 \rangle}\eta^{(\e)}(t)$, for $0 \leq k \leq r \leq N -1$, and, for every $\e \in (0, \e_0]$, 
{\small 
\begin{equation}\label{askao}
_{\langle k, r \rangle}p^{(\e)}_{ij} = \left\{
\begin{array}{lll}
_{\langle k, r \rangle}p_{k, \pm}(\e)  & = \, _{\langle k, r +1 \rangle}p_{k, \pm }(\e)   \vspace{1mm} \\
& \quad  \text{if} \ j =  k +  \frac{1 \pm  1}{2},  i = k,  \vspace{2mm} \\  
_{\langle k, r \rangle}p_{i, \pm}(\e)  & = \, _{\langle k, r +1 \rangle}p_{i, \pm}(\e)   \vspace{1mm}  \\
& \quad  \ \text{if} \ j = i \pm 1,   k < i < r,    \vspace{2mm} \\ 
_{\langle k, r \rangle}p_{r, \pm}(\e)  & = \, _{\langle k, r +1 \rangle}p_{r, \pm}(\e)    \vspace{1mm}  \\
& \quad \  \text{if} \ j =  r - \frac{1 \mp  1}{2},  i = r,  \vspace{2mm}  \\
0  & \text{otherwise}.
\end{array}
\right.
\end{equation}
}
and
{\small
\begin{equation}\label{akaso}
_{\langle k, r \rangle}e^{(\e)}_{ij}   = \left\{
\begin{array}{lll}
_{\langle k, r \rangle}e_{k, \pm}(\e) & = \, _{\langle k, r+1 \rangle}e_{k, \pm}(\e)  \vspace{1mm}  \\  
&  \text{if} \ j =  k+ \frac{1 +1}{2},   i = k, \vspace{2mm}  \\
_{\langle k, r \rangle}e_{i, \pm }(\e) & = \, _{\langle k-1, r \rangle}e_{i, \pm}(\e)   \vspace{1mm} \\ 
 & \text{if} \ j = i \pm 1, k < i < r,    \vspace{2mm} \\ 
 _{\langle k, r \rangle}e_{r, +}(\e) & =  \, _{\langle k, r +1 \rangle}e_{r, +}(\e)    \vspace{1mm} \\ 
& + \, _{\langle k, r +1 \rangle}e_{r+1}(\e) \cdot \frac{_{\langle k, r +1 \rangle}p_{r, +}(\e)}{_{\langle k, r +1 \rangle}p_{r+1, -}(\e)}   \vspace{1mm}  \\ 
&  \text{if} \ j =  r,  i = r,  \vspace{2mm}  \\ 
_{\langle k, r \rangle}e_{r, -}(\e) & = \, _{\langle k, r +1 \rangle}e_{r, -}(\e)   \vspace{1mm}    \\
 &  \text{if} \ j =  r - 1,  i = r,  \vspace{2mm} \\
0  & \text{otherwise}, 
\end{array}
\right.
\end{equation}
}
where 
\begin{equation}\label{noply}
_{\langle k, r \rangle}e_{i}(\e) = \,  _{\langle k, r \rangle}e_{i, +}(\e)  + \, _{\langle k, r \rangle}e_{i, -}(\e) 
\end{equation}\vspace{2mm}

{\bf 4.5. Explicit formulas for expectations of hitting times for  \\
\makebox[14mm]{} birth-death type semi-Markov processes} \\

As was mentioned above the process $_{\langle 0, N \rangle}\eta^{(\e)}(t) =  \eta^{(\e)}(t)$. Also the process $_{\langle 1, N \rangle}\eta^{(\e)}(t) =  \, _0\eta^{(\e)}(t)$ and the 
process $_{\langle 0, N-1 \rangle}\eta^{(\e)}(t) =  \, _N\eta^{(\e)}(t)$. 

Thus, the relations (\ref{takaskao}) and (\ref{askao}) and relations (\ref{takaso}) and (\ref{akaso})  reduce to relations   (\ref{emsemitakaska}) and (\ref{emsemitakaskaba}), when computing, respectively, the transition probabilities $_{\langle 1, N \rangle}p_{k, \pm}(\e)$  and  $_{\langle 0, N -1 \rangle}p_{k, \pm}(\e)$. 

Also, relation (\ref{emsemitakaso}) reduces to relations  (\ref{emsemitakas}) and {\ref{emsemitakasba}), when computing,
respectively, expectation of transition times $_{\langle 1, N \rangle}e_{k, \pm}(\e)$ and   $_{\langle 0, N -1 \rangle}e_{k, \pm}(\e)$.
 
Let us get, for example,   quantities $_{\langle 2, N \rangle}p_{k, \pm}(\e)$ and $_{\langle 2, N \rangle}e_{k, \pm}(\e)$ for the process  $_{\langle 2, N \rangle}\eta^{(\e)}(t)$ expressed in terms of  transition characteristic for the initial semi-Markov process $\eta^{(\e)}(t)$. 

Using recurrent  relations   (\ref{takaskao}) and then (\ref{emsemitakaska}), we get the following relations, for every $\e \in (0, \e_0]$, 
\begin{equation}\label{emsemitalop}
_{\langle 2, N \rangle}p^{(\e)}_{ij} = \left\{
\begin{array}{llll} 
_{\langle 2, N \rangle}p_{2, \pm}(\e)  & = \, _{\langle 1, N \rangle}p_{2, \pm }(\e)  = p_{2, \pm }(\e)   \vspace{1mm} \\
& \quad  \text{if} \ j =  2 +  \frac{1 \pm  1}{2},  i = 2,  \vspace{2mm} \\  
_{\langle 2, N \rangle}p_{i, \pm}(\e)  & = \, _{\langle 1, N \rangle}p_{i, \pm}(\e) = p_{i, \pm}(\e)  \vspace{1mm}  \\
& \quad  \ \text{if} \ j = i \pm 1,   2 < i < N,    \vspace{2mm} \\ 
_{\langle 2, N \rangle}p_{N, \pm }(\e)  & = \, _{\langle 1, N \rangle}p_{N, \pm}(\e) = p_{N, \pm}(\e)    \vspace{1mm}  \\
& \quad \  \text{if} \ j =  N - \frac{1 \mp  1}{2},  i = N,  \vspace{2mm}  \\
0  & \quad s \text{otherwise}. 
\end{array}
\right.
\end{equation}
and
{\small
\begin{equation}\label{emkas}
_{\langle 2, N \rangle}e^{(\e)}_{ij} = \left\{
\begin{array}{llll}
_{\langle 2, N \rangle}e_{2, +}(\e)  & = \, _{\langle 1, N \rangle}e_{1, +}(\e) = e_{1, +}(\e)   \vspace{1mm} \\
& \quad  \text{if}  \ j =  2,   i = 1,  \vspace{2mm} \\
_{\langle 2, N \rangle}e_{1, -}(\e)  & =   \, _{\langle 1, N \rangle}e_{2, -}(\e)  
  + \, _{\langle 1, N \rangle}e_{1}(\e) \cdot \frac{_{\langle 1, N \rangle}p_{2, -}(\e)}{_{\langle 1, N \rangle}p_{1, +}(\e)}  \vspace{1mm} \\
&  =  e_{2, -}(\e) +  e_{1}(\e) \cdot \frac{p_{2, -}(\e)}{p_{1, +}(\e)} +  e_{0}(\e) \cdot \frac{p_{1, -}(\e)p_{2, -}(\e)}{p_{0, +}(\e)p_{1, +}(\e)} \vspace{1mm}  \\ 
&   \quad \text{if} \ j =  1,  i = 1, \vspace{2mm}  \\    
_{\langle 2, N \rangle}e_{i, \pm}(\e) & = \, _{\langle 1, N \rangle}e_{i, \pm}(\e) = e_{i, \pm}(\e)   \vspace{1mm} \\
&  \text{if} \ j = i \pm 1, 1 < i < N,  \vspace{2mm}  \\ 
_{\langle 2, N \rangle}e_{N, \pm}(\e)   & = \, _{\langle 1, N \rangle}e_{N, \pm}(\e)  =  e_{N, \pm}(\e)  \vspace{1mm} \\
&  \quad \text{if} \ j =  N - \frac{1 \mp  1}{2},  i = N, \vspace{2mm}  \\
0  &  \quad  \text{otherwise},
\end{array}
\right.
\end{equation}
}
since
\begin{align}\label{goptrew}
_{\langle 2, N \rangle}e_{2, +}(\e) & =  \, _{\langle 1, N \rangle}e_{2, -}(\e)  
 + \, _{\langle 1, N \rangle}e_{1}(\e) \cdot \frac{_{\langle 1, N \rangle}p_{2, -}(\e)}{_{\langle 1, N \rangle}p_{1, +}(\e)}  \nonumber \\
 & = e_{2, -}(\e) + \Big( \big( e_{1, -}(\e) + e_{0}(\e) \cdot \frac{p_{1, -}(\e)}{p_{0, +}(\e)} \big) + e_{1, +}(\e) \Big)  \cdot \frac{p_{2, -}(\e)}{p_{1, +}(\e)}  \nonumber \\
 & =  e_{2, -}(\e) +  e_{1}(\e) \cdot \frac{p_{2, -}(\e)}{p_{1, +}(\e)} +  e_{0}(\e) \cdot \frac{p_{1, -}(\e)p_{2, -}(\e)}{p_{0, +}(\e)p_{1, +}(\e)}.  
\end{align}

By iterating recurrent formulas (\ref{takaskao})  -- (\ref{takaso}) and  (\ref{askao})  -- (\ref{akaso}) in the way shown in relations (\ref{emsemitalop}) -- (\ref{emkas}), we  get the following explicit formulas for probabilities transition probabilities $_{\langle k, r \rangle}p^{(\e)}_{ij}$ and expectations of transition times $_{\langle k, r \rangle}e^{(\e)}_{ij}$ for the reduced embedded semi-Markov process  $_{\langle k, r \rangle}\eta^{(\e)}(t)$ expressed in terms of the transition characteristic for the initial semi-Markov process $\eta^{(\e)}(t)$, for $0 \leq k \leq r \leq N$ and, for every $\e \in (0, \e_0]$, 
\begin{equation}\label{emkaskao}
_{\langle k, r \rangle}p^{(\e)}_{ij} = \left\{
\begin{array}{lll}
_{\langle k, r \rangle}p_{k, \pm}(\e)  & = p_{k, \pm }(\e)   \vspace{1mm} \\
& \quad  \text{if} \ j =  k +  \frac{1 \pm  1}{2},  i = k,  \vspace{2mm} \\  
_{\langle k, r \rangle}p_{i, \pm}(\e)  & = p_{i, \pm}(\e)   \vspace{1mm}  \\
& \quad  \ \text{if} \ j = i \pm 1,   k < i < r,    \vspace{2mm} \\ 
_{\langle k, r \rangle}p_{r, +}(\e)  & = p_{r, \pm}(\e)    \vspace{1mm}  \\
& \quad \  \text{if} \ j =  r - \frac{1 \mp  1}{2},  i = r,  \vspace{2mm}  \\
0  & \quad s\text{otherwise}.
\end{array}
\right.
\end{equation}
and 
{\small
\begin{equation}\label{emsemitakaso}
_{\langle k, r \rangle}e^{(\e)}_{ij} = \left\{
\begin{array}{lll}
_{\langle k, r \rangle}e_{k, +}(\e) &  = e_{k, +}(\e) \vspace{1mm}  \\  
&  \text{if} \ j =  k+1,   i = k, \vspace{2mm}  \\
_{\langle k, r \rangle}e_{k, -}(\e) & =   e_{k, -}(\e) +   e_{k -1}(\e)\cdot  \frac{p_{k, -}(\e)}{p_{k-1, +}(\e)}  \vspace{1mm}  \\
& +  \cdots + e_{0}(\e)\cdot \frac{p_{1, -}(\e) \cdots p_{k, -}(\e)}{p_{0, +}(\e) \cdots p_{k -1, +}(\e)}  \vspace{1mm}   \\ 
&  \text{if} \ j =  k,  i = k,  \vspace{2mm}\\ 
_{\langle k, r \rangle}e_{i, +}(\e) & = e_{i, \pm}(\e) \vspace{1mm}  \\ 
 & \text{if} \ j = i \pm 1, k < i < r, \vspace{2mm}  \\
 _{\langle k, r \rangle}e_{r, +}(\e) & =   e_{r, +}(\e) +   e_{r +1}(\e)\cdot  \frac{p_{r, +}(\e)}{p_{r+1, -}(\e)}  \vspace{1mm}  \\
& +  \cdots + e_{N}(\e)\cdot \frac{p_{N -1, +}(\e) \cdots p_{r, +}(\e)}{p_{N, -}(\e) \cdots p_{r+1, -}(\e)}  \vspace{1mm}   \\ 
&  \text{if} \ j =  r,  i = r,  \vspace{2mm}\\  
_{\langle k, r \rangle}e_{r, -}(\e) & = e_{r, -}(\e)    \\
 &  \text{if} \ j =  r - 1,  i = r, \vspace{2mm} \\
0  &  \text{otherwise}.
\end{array}
\right.
\end{equation}
}

Let us denote by $_{\langle k, r \rangle}\tau_j^{(\e)}$ the hitting time for the state $j \in \, _{\langle k, r \rangle}\XX$ for the reduced 
semi-Markov process $_{\langle k, r \rangle}\eta^{(\e)}(t)$.

By Theorem 1, the following relation takes place, for $i, j \in \, _{\langle k, r \rangle}\XX$ and, for every $\e \in (0, \e_0]$, 
\begin{equation}\label{hitrerva}
\EE_i \tau_j^{(\e)} = \EE_i \, _{\langle k, r \rangle}\tau_j^{(\e)}.  
\end{equation}

Let us now choose $k = r = m = i \in \XX$. In this case the reduced phase space $_{\langle i, i \rangle}\XX = \{ i \}$ is o one-state set. In this case, the process $_{\langle i, i \rangle}\eta^{(\e)}(t)$ returns to  the state $m$ after every jump.  This implies that, in this case,  for every $\e \in (0, \e_0]$, 
\begin{equation}\label{hitrerva}
E_{ii}(\e) = \EE_i \tau_i^{(\e)} = \EE_i \, _{\langle i, i \rangle}\tau_i^{(\e)} = \, _{\langle i, i \rangle}e_{i}(\e).  
\end{equation}
 
The following formulas takes place, for every $i \in \XX$ an, for every $\e \in (0, \e_0]$, 
%\begin{equation*}
\begin{align}\label{opada}
E_{ii}(\e)   & = e_i(\e) \nonumber \\
& +  e_{i-1}(\e) \frac{p_{i, -}(\e)}{p_{i-1, +}(\e)} +   e_{i-2}(\e) \frac{p_{i-1, -}(\e) p_{i, -}(\e)}{p_{i-2, +}(\e) p_{i-1, +}(\e)} \nonumber \\ 
&  + \cdots + e_{0}(\e) \frac{p_{1, -}(\e) p_{2, -}(\e) \cdots p_{i, -}(\e)}{p_{0, +}(\e)p_{1, +}(\e) \cdots p_{i-1, +}(\e)} \nonumber \\ 
& +  e_{i+1}(\e) \frac{p_{i, +}(\e)}{p_{i+1, -}(\e)} +   e_{i+2}(\e) \frac{p_{i+1, +}(\e) p_{i, +}(\e)}{p_{i+2, -}(\e) p_{i+1, -}(\e)} \nonumber \\ 
%\end{aligned}
%\end{equation*}
%\begin{align}
&  + \cdots + e_{N}(\e) \frac{p_{N -1, +}(\e) p_{N -2, +}(\e) \cdots p_{i, +}(\e)}{p_{N, -}(\e)p_{N -1, -}(\e) \cdots p_{i+1, -}(\e)} 
\end{align}

In particular, 
\begin{align}\label{opadada}
E_{00}(\e)   & = e_0(\e)  +  e_{1}(\e) \frac{p_{0, +}(\e)}{p_{1, -}(\e)} +   e_{2}(\e) \frac{p_{1, +}(\e) p_{0, +}(\e)}{p_{2, -}(\e) p_{1, -}(\e)} \nonumber \\ 
& \quad  + \cdots + e_{N}(\e) \frac{p_{N -1, +}(\e) p_{N -2, +}(\e) \cdots p_{0, +}(\e)}{p_{N, -}(\e)p_{N -1, -}(\e) \cdots p_{1, -}(\e)}.
\end{align}

Formulas (\ref{stationary}), (\ref{opada}) and (\ref{quasiF2}),  (\ref{quasiF3}) yield, in an obvious way, explicit formulas  for 
stationary and conditional quasi-stationary distributions for birth-death-type semi-Markov processes.  

It should be noted that such formulas for stationary distributions of birth-death-type Markov chains are well known and can be found, for example, in Feller  (1968). In context of our studies, a special value has the presented above recurrent algorithm for getting such formulas,  based on sequential reduction of the phase space for birth-death-type semi-Markov processes. \\
 
{\bf 5. First and second order asymptotic expansions for stationary \\ 
\makebox[10.5mm]{} and conditional quasi-stationary distributions} \\

In this section, we give explicit formulas for the coefficients in the first and the second order asymptotic expansions for stationary and conditional quasi-stationary distributions.

The results of the present section are based on the explicit formula (\ref{opada}) for expected return times and the expressions which connect these quantities with stationary and conditional quasi-stationary distributions. We obtain the first and second order asymptotic expansions from these formulas by using operational rules for Laurent asymptotic expansions presented in Silvestrov, D. and Silvestrov, S. (2015, 2016). Some of these operational rules which are relevant for this paper can be found in Subsection 6.1.

Let us here mention that recurrent algorithms for computation of the corresponding higher order asymptotic expansions are given in Section 6. The methods used in Section 6, which in particular give the first and the second order asymptotic expansions, are more convenient and efficient for computer programs. However, the method used in this section is interesting in its own right since it gives a more explicit description which helps us to better understand the asymptotic properties of our models. Moreover, having two different ways of computing the first and the second order coefficients may help us to detect possible errors or numerical instability in the corresponding computer programs.

In this section, it will be convenient to use the following notation,
\begin{equation} \label{defGammaij}
\Gamma_{i,j,\pm}(\e) = p_{i,\pm}(\e) p_{i+1,\pm}(\e) \cdots p_{j,\pm}(\e), \ 0 \leq i \leq j \leq N.
\end{equation}

Using (\ref{defGammaij}), we can write formula (\ref{opada}) as
\begin{align} \label{formulaEii}
E_{ii}(\e) & = e_i(\e) + \sum_{k=0}^{i-1} e_k(\e) \frac{\Gamma_{k+1,i,-}(\e)}{\Gamma_{k,i-1,+}(\e)} \nonumber \\ 
& \quad + \sum_{k=i+1}^N e_k(\e) \frac{\Gamma_{i,k-1,+}(\e)}{\Gamma_{i+1,k,-}(\e)}, \ i \in \XX.
\end{align}

In particular, we have
\begin{equation} \label{formulaE00}
E_{00}(\e) = e_0(\e) + \sum_{k \in \, _0\XX} e_k(\e) \frac{\Gamma_{0,k-1,+}(\e)}{\Gamma_{1,k,-}(\e)}, 
\end{equation}
and
\begin{equation} \label{formulaENN}
E_{NN}(\e) = e_N(\e) + \sum_{k \in \, _N\XX} e_k(\e) \frac{\Gamma_{k+1,N,-}(\e)}{\Gamma_{k,N-1,+}(\e)}.
\end{equation}

We will compute the desired asymptotic expansions by applying operational rules for Laurent asymptotic expansions in relations (\ref{formulaEii})--(\ref{formulaENN}). In order for the presentation to not be too repetitive, we will directly compute the second order asymptotic expansions which contain the first order asymptotic expansions as special cases. In particular, this gives us limits for stationary and conditional quasi-stationary distributions.

The formulas for computing the asymptotic expansions are different depending on whether condition ${\bf H_1}$, ${\bf H_2}$, or ${\bf H_3}$ holds. We consider these three cases in Subsections 5.1, 5.2, and 5.3, respectively. 

Each of these sections will have the same structure: First, we present a lemma which successively constructs asymptotic expansions for the quantities given in relations (\ref{formulaEii})--(\ref{formulaENN}). Then, using these expansions, we construct the first and the second order asymptotic expansions for stationary (Subsections 5.1--5.3) and conditional quasi-stationary distributions (Subsections 5.2--5.3). \\

{\bf 5.1. First and second order asymptotic expansions for \\
\makebox[14mm]{} stationary distributions under condition ${\bf H_1}$} \\

In the case where condition ${\bf H_1}$ holds, the semi-Markov process has no asymptotically absorbing states. In this case, all quantities in relations (\ref{formulaEii})--(\ref{formulaENN}) are of order $O(1)$ and the construction of asymptotic expansions are rather straightforward.

In the following lemma we successively construct asymptotic expansions for the quantities given in relations (\ref{formulaEii})--(\ref{formulaENN}).
\begin{lemma} \label{lemmaH1}
Assume that conditions  ${\bf A}$--${\bf C}$, ${\bf D}_1$--${\bf F}_1$, ${\bf G}$  and ${\bf H_1}$ hold. Then:
\begin{enumerate}
\item[$\mathbf{(i)}$] For $i \in \XX$, we have
\begin{equation*}
e_i(\e) = b_i[0] + b_i[1] \e + \dot{o}_i(\e), \ \e \in (0,\e_0],
\end{equation*}
where $\dot{o}_i(\e) / \e \to 0$ as $\e \to 0$ and
\begin{equation*}
b_i[0] = b_{i,-}[0] + b_{i,+}[0] > 0, \quad b_i[1] = b_{i,-}[1] + b_{i,+}[1].
\end{equation*}
\item[$\mathbf{(ii)}$] For $0 \leq i \leq j \leq N$, we have
\begin{equation*}
\Gamma_{i,j,\pm}(\e) = A_{i,j,\pm}[0] + A_{i,j,\pm}[1] \e + o_{i,j,\pm}(\e), \ \e \in (0,\e_0],
\end{equation*}
where $o_{i,j,\pm}(\e) / \e \to 0$ as $\e \to 0$ and
\begin{equation*}
A_{i,j,\pm}[0] = a_{i,\pm}[0] a_{i+1,\pm}[0] \cdots a_{j,\pm}[0] > 0,
\end{equation*}
\begin{equation*}
A_{i,j,\pm}[1] = \sum_{n_i + n_{i+1} + \cdots + n_j = 1} a_{i,\pm}[n_i] a_{i+1,\pm}[n_{i+1}] \cdots a_{j,\pm}[n_j].
\end{equation*}
\item[$\mathbf{(iii)}$] For $0 \leq k \leq i-1, \ i \in \, _0\XX$, we have
\begin{equation*}
\frac{\Gamma_{k+1,i,-}(\e)}{\Gamma_{k,i-1,+}(\e)} = A_{k,i}^*[0] + A_{k,i}^*[1] \e + o_{k,i}^*(\e), \ \e \in (0,\e_0],
\end{equation*}
where $o_{k,i}^*(\e) / \e \to 0$ as $\e \to 0$ and
\begin{equation*}
A_{k,i}^*[0] = \frac{A_{k+1,i,-}[0]}{A_{k,i-1,+}[0]} > 0,
\end{equation*}
\begin{equation*}
A_{k,i}^*[1] = \frac{A_{k+1,i,-}[1] A_{k,i-1,+}[0] - A_{k+1,i,-}[0] A_{k,i-1,+}[1]}{A_{k,i-1,+}[0]^2}.
\end{equation*}
\item[$\mathbf{(iv)}$] For $i+1 \leq k \leq N, \ i \in \, _N\XX$, we have
\begin{equation*}
\frac{\Gamma_{i,k-1,+}(\e)}{\Gamma_{i+1,k,-}(\e)} = A_{k,i}^*[0] + A_{k,i}^*[1] \e + o_{k,i}^*(\e), \ \e \in (0,\e_0],
\end{equation*}
where $o_{k,i}^*(\e) / \e \to 0$ as $\e \to 0$ and
\begin{equation*}
A_{k,i}^*[0] = \frac{A_{i,k-1,+}[0]}{A_{i+1,k,-}[0]} > 0,
\end{equation*}
\begin{equation*}
A_{k,i}^*[1] = \frac{A_{i,k-1,+}[1] A_{i+1,k,-}[0] - A_{i,k-1,+}[0] A_{i+1,k,-}[1]}{A_{i+1,k,-}[0]^2}.
\end{equation*}
\item[$\mathbf{(v)}$] For $i \in \XX$, we have
\begin{equation*}
E_{ii}(\e) = B_{ii}[0] + B_{ii}[1] \e + \dot{o}_{ii}(\e), \ \e \in (0,\e_0],
\end{equation*}
where $\dot{o}_{ii}(\e) / \e \to 0$ as $\e \to 0$ and
\begin{equation*}
B_{ii}[0] = b_i[0] + \sum_{k \in \, _i\XX} b_k[0] A_{k,i}^*[0] > 0,
\end{equation*}
\begin{equation*}
B_{ii}[1] = b_i[1] + \sum_{k \in \, _i\XX} ( b_k[0] A_{k,i}^*[1] + b_k[1] A_{k,i}^*[0] ).
\end{equation*}
\end{enumerate}
\end{lemma}

\begin{proof}
Since $e_i(\e) = e_{i,-}(\e) + e_{i,+}(\e)$, $i \in \XX$, part $\mathbf{(i)}$ follows immediately from condition ${\bf E}_1$.

For the proof of part $\mathbf{(ii)}$ we notice that it follows from the definition (\ref{defGammaij}) of $\Gamma_{i,j,\pm}(\e)$ and condition ${\bf D}_1$ that
\begin{equation*}
\Gamma_{i,j,\pm}(\e) = \prod_{k=i}^j ( a_{k,\pm}[0] + a_{k,\pm}[1] \e + o_{k,\pm}(\e) ), \ 0 \leq i \leq j \leq N.
\end{equation*}

By applying the multiple product rule for asymptotic expansions, we obtain the asymptotic relation given in part $\mathbf{(ii)}$ where the coefficients $A_{i,j,\pm}[0]$, $0 \leq i \leq j \leq n$, are positive since condition ${\bf H_1}$ holds.

In order to prove parts $\mathbf{(iii)}$ and $\mathbf{(iv)}$ we use the result in part $\mathbf{(ii)}$. For $0 \leq k \leq i-1$, $i \in \, _0\XX$, this gives us
\begin{equation} \label{lmmH1quotient1}
\frac{\Gamma_{k+1,i,-}(\e)}{\Gamma_{k,i-1,+}(\e)} = \frac{A_{k+1,i,-}[0] + A_{k+1,i,-}[1] \e + o_{k+1,i,-}(\e)}{A_{k,i-1,+}[0] + A_{k,i-1,+}[1] \e + o_{k,i-1,+}(\e)},
\end{equation}
and, for $i+1 \leq k \leq N$, $i \in \, _N\XX$, we get
\begin{equation} \label{lmmH1quotient2}
\frac{\Gamma_{i,k-1,+}(\e)}{\Gamma_{i+1,k,-}(\e)} = \frac{A_{i,k-1,+}[0] + A_{i,k-1,+}[1] \e + o_{i,k-1,+}(\e)}{A_{i+1,k,-}[0] + A_{i+1,k,-}[1] \e + o_{i+1,k,-}(\e)}.
\end{equation}

Using the division rule for asymptotic expansions in relations (\ref{lmmH1quotient1}) and (\ref{lmmH1quotient2}) we get the asymptotic expansions given in parts $\mathbf{(iii)}$ and $\mathbf{(iv)}$.

Finally, we can use relation (\ref{formulaEii}) to prove part $\mathbf{(v)}$. This relation together with the results in parts $\mathbf{(i)}$, $\mathbf{(iii)}$, and $\mathbf{(iv)}$ yield
\begin{align*}
E_{ii}(\e) &= b_i[0] + b_i[1] \e + \dot{o}_i(\e) \\
&+ \sum_{k \in \, _i\XX} (b_k[0] + b_k[1] \e + \dot{o}_k(\e)) \\ 
& \times (A_{k,i}^*[0] + A_{k,i}^*[1] \e + o_{k,i}^*(\e)), \ i \in \XX.
\end{align*}
A combination of the product rule and the multiple summation rule for asymptotic expansions gives the asymptotic relation in part $\mathbf{(v)}$.
\end{proof}

The following theorem gives second order asymptotic expansions for stationary probabilities. In particular, this theorem shows that there exist limits for stationary probabilities, $\pi_i(0) = \lim_{\e \to 0} \pi_i(\e)$, $i \in \XX$, where $\pi_i(0) > 0$, $i \in \XX$. 

The corresponding higher order asymptotic expansions are given in Theorem 5.

\begin{theorem}
Assume that conditions  ${\bf A}$--${\bf C}$, ${\bf D}_1$--${\bf F}_1$, ${\bf G}$  and ${\bf H_1}$ hold. Then, we have the following asymptotic relation for the stationary probabilities $\pi_i(\e)$, $i \in \XX$,
\begin{equation*}
\pi_i(\e) = c_i[0] + c_i[1] \e + o_i(\e), \ \e \in (0,\e_0],
\end{equation*}
where $o_i(\e) / \e \to 0$ as $\e \to 0$ and
\begin{equation*}
c_i[0] = \frac{b_i[0]}{B_{ii}[0]} > 0, \quad c_i[1] = \frac{b_i[1] B_{ii}[0] - b_i[0] B_{ii}[1]}{B_{ii}[0]^2},
\end{equation*}
where the $B_{ii}[0]$, $B_{ii}[1]$, $i \in \XX$, can be computed from the formulas given in Lemma \ref{lemmaH1}.
\end{theorem}

\begin{proof}
It follows from condition ${\bf E}$ and part $\mathbf{(v)}$ of Lemma \ref{lemmaH1} that, for $i \in \XX$,
\begin{equation*}
\pi_i(\e) = \frac{e_i(\e)}{E_{ii}(\e)} = \frac{b_i[0] + b_i[1] \e + \dot{o}_i(\e)}{B_{ii}[0] + B_{ii}[1] \e + \dot{o}_{ii}(\e)}.
\end{equation*}
The result now follows from the division rule for asymptotic expansions.
\end{proof}  
\vspace{3mm}

{\bf 5.2. First and second order asymptotic expansions for \\ 
\makebox[14mm]{} stationary and conditional quasi-stationary distributions  \\ 
\makebox[14mm]{} under condition ${\bf H_2}$} \\

In the case where condition ${\bf H_2}$ holds, the semi-Markov process has one asymptotically absorbing state, namely state $0$. This means that $p_{0,+}(\e) \sim O(\e)$ and since this quantity is involved in relations (\ref{formulaEii})--(\ref{formulaENN}), the pivotal properties of the expansions are less obvious. Furthermore, since some terms now tends to infinity, we partly need to operate with Laurent asymptotic expansions.

In order to separate cases where $i=0$ or $i \in \, _0\XX$ we will use the indicator function $\gamma_i = I(i=0)$, that is, $\gamma_0 = 1$ and $\gamma_i = 0$ for $i \in \, _0\XX$.

The following lemma gives asymptotic expansions for quantities in relations (\ref{formulaEii})--(\ref{formulaENN}).

\begin{lemma} \label{lemmaH2}
Assume that conditions ${\bf A}$--${\bf C}$, ${\bf D}_1$--${\bf F}_1$, ${\bf G}$ and ${\bf H_2}$ hold. Then:
\begin{enumerate}
\item[$\mathbf{(i)}$] For $i \in \XX$, we have
\begin{equation*}
e_i(\e) = b_i[0] + b_i[1] \e + \dot{o}_i(\e), \ \e \in (0,\e_0],
\end{equation*}
where $\dot{o}_i(\e) / \e \to 0$ as $\e \to 0$ and
\begin{equation*}
b_i[0] = b_{i,-}[0] + b_{i,+}[0] > 0, \quad b_i[1] = b_{i,-}[1] + b_{i,+}[1].
\end{equation*}
\item[$\mathbf{(ii)}$] For $0 \leq i \leq j \leq N$, we have,
\begin{equation*}
\Gamma_{i,j,+}(\e) = A_{i,j,+}[\gamma_i] \e^{\gamma_i} + A_{i,j,+}[\gamma_i+1] \e^{\gamma_i+1} + o_{i,j,+}(\e^{\gamma_i+1}), \ \e \in (0,\e_0],
\end{equation*}
where $o_{i,j,+}(\e^{\gamma_i+1}) / \e^{\gamma_i+1} \to 0$ as $\e \to 0$ and
\begin{equation*}
A_{i,j,+}[\gamma_i] = a_{i,+}[\gamma_i] a_{i+1,+}[0] \cdots a_{j,+}[0] > 0,
\end{equation*}
\begin{equation*}
A_{i,j,+}[\gamma_i+1] = \sum_{n_i + n_{i+1} + \cdots + n_j = 1} a_{i,+}[\gamma_i + n_i] a_{i+1,+}[n_{i+1}] \cdots a_{j,+}[n_j].
\end{equation*}
\item[$\mathbf{(iii)}$] For $0 \leq i \leq j \leq N$, we have
\begin{equation*}
\Gamma_{i,j,-}(\e) = A_{i,j,-}[0] + A_{i,j,-}[1] \e + o_{i,j,-}(\e), \ \e \in (0,\e_0],
\end{equation*}
where $o_{i,j,-}(\e) / \e \to 0$ as $\e \to 0$ and
\begin{equation*}
A_{i,j,-}[0] = a_{i,-}[0] a_{i+1,-}[0] \cdots a_{j,-}[0] > 0,
\end{equation*}
\begin{equation*}
A_{i,j,-}[1] = \sum_{n_i + n_{i+1} + \cdots + n_j = 1} a_{i,-}[n_i] a_{i+1,-}[n_{i+1}] \cdots a_{j,-}[n_j].
\end{equation*}
\item[$\mathbf{(iv)}$] For $0 \leq k \leq i-1, \ i \in \, _0\XX$, we have
\begin{align*}
\frac{\Gamma_{k+1,i,-}(\e)}{\Gamma_{k,i-1,+}(\e)} &= A_{k,i}^*[-\gamma_k] \e^{-\gamma_k} + A_{k,i}^*[-\gamma_k+1] \e^{-\gamma_k+1} \\
&+ o_{k,i}^*(\e^{-\gamma_k+1}), \ \e \in (0,\e_0],
\end{align*}
where $o_{k,i}^*(\e^{-\gamma_k+1}) / \e^{-\gamma_k+1} \to 0$ as $\e \to 0$ and
\begin{equation*}
A_{k,i}^*[-\gamma_k] = \frac{A_{k+1,i,-}[0]}{A_{k,i-1,+}[\gamma_k]} > 0,
\end{equation*}
\begin{equation*}
A_{k,i}^*[-\gamma_k+1] = \frac{A_{k+1,i,-}[1] A_{k,i-1,+}[\gamma_k] - A_{k+1,i,-}[0] A_{k,i-1,+}[\gamma_k+1]}{A_{k,i-1,+}[\gamma_k]^2}.
\end{equation*}
\item[$\mathbf{(v)}$] For $i+1 \leq k \leq N, \ i \in \, _N\XX$, we have
\begin{align*}
\frac{\Gamma_{i,k-1,+}(\e)}{\Gamma_{i+1,k,-}(\e)} &= A_{k,i}^*[\gamma_i] \e^{\gamma_i} + A_{k,i}^*[\gamma_i+1] \e^{\gamma_i+1} \\
&+ o_{k,i}^*(\e^{\gamma_i+1}), \ \e \in (0,\e_0],
\end{align*}
where $o_{k,i}^*(\e^{\gamma_i+1}) / \e^{\gamma_i+1} \to 0$ as $\e \to 0$ and
\begin{equation*}
A_{k,i}^*[\gamma_i] = \frac{A_{i,k-1,+}[\gamma_i]}{A_{i+1,k,-}[0]} > 0,
\end{equation*}
\begin{equation*}
A_{k,i}^*[\gamma_i+1] = \frac{A_{i,k-1,+}[\gamma_i+1] A_{i+1,k,-}[0] - A_{i,k-1,+}[\gamma_i] A_{i+1,k,-}[1]}{A_{i+1,k,-}[0]^2}.
\end{equation*}
\item[$\mathbf{(vi)}$] For $i \in \XX$, we have
\begin{equation*}
E_{ii}(\e) = B_{ii}[\gamma_i-1] \e^{\gamma_i-1} + B_{ii}[\gamma_i] \e^{\gamma_i} + \dot{o}_{ii}(\e^{\gamma_i}), \ \e \in (0,\e_0],
\end{equation*}
where $\dot{o}_{ii}(\e^{\gamma_i}) / \e^{\gamma_i} \to 0$ as $\e \to 0$ and
\begin{equation*}
B_{00}[0] = b_0[0] > 0, \quad B_{00}[1] = b_0[1] + \sum_{k \in \, _0\XX} b_k[0] A_{k,0}^*[1],
\end{equation*}
\begin{equation*}
B_{ii}[-1] = b_0[0] A_{0,i}^*[-1] > 0, \ i \in \, _0\XX,
\end{equation*}
\begin{equation*}
B_{ii}[0] = b_0[1] A_{0,i}^*[-1] + b_i[0] + \sum_{k \in \, _i\XX}b_k[0] A_{k,i}^*[0], \ i \in \, _0\XX.
\end{equation*}
\end{enumerate}
\end{lemma}

\begin{proof}
Let us first note that the quantities in parts $\mathbf{(i)}$ and $\mathbf{(iii)}$ do not depend on $p_{0,+}(\e)$, so the proofs for these parts are the same as the proofs for parts $\mathbf{(i)}$ and $\mathbf{(ii)}$ in Lemma \ref{lemmaH1}, respectively.

We now prove part $\mathbf{(ii)}$. Notice that it follows from conditions ${\bf D}_1$ and ${\bf H_2}$ that $p_{i,+}(\e) = a_{i,+}[\gamma_i] \e^{\gamma_i} + a_{i,+}[\gamma_i+1] \e^{\gamma_i+1} + o_{i,+}(\e^{\gamma_i+1})$, $i \in \XX$. Using this and the definition (\ref{defGammaij}) of $\Gamma_{i,j,+}(\e)$ gives
\begin{align*}
\Gamma_{i,j,+}(\e) &= ( a_{i,+}[\gamma_i] \e^{\gamma_i} + a_{i,+}[\gamma_i+1] \e^{\gamma_i+1} + o_{i,+}(\e^{\gamma_i+1}) ) \\
&\times ( a_{i+1,+}[0] + a_{i+1,+}[1] \e + o_{i+1,+}(\e) ) \\
&\times \cdots \times \\
&\times ( a_{j,+}[0] + a_{j,+}[1] \e + o_{j,+}(\e) ), \ 0 \leq i \leq j \leq N.
\end{align*}

An application of the multiple product rule for asymptotic expansions shows that part $\mathbf{(ii)}$ holds.

Now, using the results in parts $\mathbf{(ii)}$ and $\mathbf{(iii)}$ we get, for $0 \leq k \leq i-1$, $i \in \, _0\XX$,
\begin{equation} \label{lmmH2quotient1}
\frac{\Gamma_{k+1,i,-}(\e)}{\Gamma_{k,i-1,+}(\e)} = \frac{A_{k+1,i,-}[0] + A_{k+1,i,-}[1] \e + o_{k+1,i,-}(\e)}{A_{k,i-1,+}[\gamma_k] \e^{\gamma_k} + A_{k,i-1,+}[\gamma_k+1] \e^{\gamma_k+1} + o_{k,i-1,+}(\e^{\gamma_k+1})},
\end{equation}
and, for $i+1 \leq k \leq N$, $i \in \, _N\XX$,
\begin{equation} \label{lmmH2quotient2}
\frac{\Gamma_{i,k-1,+}(\e)}{\Gamma_{i+1,k,-}(\e)} = \frac{A_{i,k-1,+}[\gamma_i] \e^{\gamma_i} + A_{i,k-1,+}[\gamma_i+1] \e^{\gamma_i+1} + o_{i,k-1,+}(\e^{\gamma_i+1})}{A_{i+1,k,-}[0] + A_{i+1,k,-}[1] \e + o_{i+1,k,-}(\e)}.
\end{equation}

Notice that it is possible that the quantity in equation (\ref{lmmH2quotient1}) tends to infinity as $\e \to 0$. Applying the division rule for Laurent asymptotic expansions in relations (\ref{lmmH2quotient1}) and (\ref{lmmH2quotient2}) yields the asymptotic relations given in parts $\mathbf{(iv)}$ and $\mathbf{(v)}$.

In order to prove part $\mathbf{(vi)}$, we consider the cases $i=0$ and $i \in \, _0\XX$ separately. First note that it follows from relation (\ref{formulaE00}) and the results in parts $\mathbf{(i)}$ and $\mathbf{(iv)}$ that
\begin{align*}
E_{00}(\e) &= b_0[0] + b_0[1] \e + \dot{o}_0(\e) \\
&+ \sum_{k \in \, _0\XX} (b_k[0] + b_k[1] \e + \dot{o}_k(\e)) (A_{k,0}^*[1] \e + A_{k,0}^*[2] \e^2 + o_{k,0}^*(\e^2)).
\end{align*}

Using the product rule and the multiple summation rule for asymptotic expansions we obtain the asymptotic relation in part $\mathbf{(vi)}$ for the case $i=0$.

If $i \in \, _0\XX$, relation (\ref{formulaEii}) implies together with parts $\mathbf{(i)}$, $\mathbf{(iv)}$, and $\mathbf{(v)}$ that
\begin{align*}
E_{ii}(\e) &= b_i[0] + b_i[1] \e + \dot{o}_i(\e) \\
&+ (b_0[0] + b_0[1] \e + \dot{o}_0(\e)) (A_{0,i}^*[-1] \e^{-1} + A_{0,i}^*[0] + o_{0,i}^*(1)) \\
&+ \sum_{k \in \, _{0,i}\XX} (b_k[0] + b_k[1] \e + \dot{o}_k(\e)) (A_{k,i}^*[0] + A_{k,i}^*[1] \e + o_{k,i}^*(\e)).
\end{align*}

Notice that the term corresponding to $k=0$ is of order $O(\e^{-1})$ while all other terms in the sum are of order $O(1)$. We can again apply the product rule and multiple summation rule for Laurent asymptotic expansions and in this case, the asymptotic relation in part $\mathbf{(vi)}$ is obtained for $i \in \, _0\XX$.
\end{proof}

The following theorem gives second order asymptotic expansions for stationary and conditional quasi-stationary probabilities. In particular, part $\mathbf{(i)}$ of this theorem shows that there exist limits for stationary probabilities, $\pi_i(0) = \lim_{\e \to 0} \pi_i(\e)$, $i \in \XX$, where $\pi_0(0) = 1$ and $\pi_i(0) = 0$ for $i \in \, _0\XX$. Furthermore, part $\mathbf{(ii)}$ of the theorem shows, in particular, that there exist limits for conditional quasi-stationary probabilities $\tilde{\pi}_i(0) = \lim_{\e \to 0} \tilde{\pi}_i(\e)$, $i \in \, _0\XX$, where $\tilde{\pi}_i(0) > 0$, $i \in \, _0\XX$. The corresponding higher order asymptotic expansions are given in Theorem 5.
\begin{theorem} \label{theoremH2}
Assume that conditions  ${\bf A}$--${\bf C}$, ${\bf D}_1$--${\bf F}_1$, ${\bf G}$ and ${\bf H_2}$ hold. Then:
\begin{enumerate}
\item[$\mathbf{(i)}$] We have the following asymptotic relation for the stationary probabilities $\pi_i(\e)$, $i \in \XX$,
\begin{equation*}
\pi_i(\e) = c_i[\tilde{l}_i] \e^{\tilde{l}_i} + c_i[\tilde{l}_i+1] \e^{\tilde{l}_i+1} + o_i(\e^{\tilde{l}_i+1}), \ \e \in (0,\e_0],
\end{equation*}
where $\tilde{l}_i = I(i \neq 0)$, $o_i(\e^{\tilde{l}_i+1}) / \e^{\tilde{l}_i+1} \to 0$ as $\e \to 0$, and
\begin{equation*}
c_i[\tilde{l}_i] = \frac{b_i[0]}{B_{ii}[-\tilde{l}_i]} > 0, \quad c_i[\tilde{l}_i+1] = \frac{b_i[1] B_{ii}[-\tilde{l}_i] - b_i[0] B_{ii}[-\tilde{l}_i+1]}{B_{ii}[-\tilde{l}_i]^2},
\end{equation*}
where $B_{ii}[-1]$, $i \in \, _0\XX$, $B_{ii}[0]$, $i \in \XX$, and $B_{00}[1]$, can be computed from the formulas given in Lemma \ref{lemmaH2}.
\item[$\mathbf{(ii)}$] We have the following asymptotic relation for the conditional quasi-stationary probabilities $\tilde{\pi}_i(\e)$, $i \in \, _0\XX,$
\begin{equation*}
\tilde{\pi}_i(\e) = \tilde{c}_i[0] + \tilde{c}_i[1] \e + \tilde{o}_i(\e), \ \e \in (0,\e_0],
\end{equation*}
where $\tilde{o}_i(\e) / \e \to 0$ as $\e \to 0$ and
\begin{equation*}
\tilde{c}_i[0] = \frac{c_i[1]}{d[1]} > 0, \quad \tilde{c}_i[1] = \frac{c_i[2] d[1] - c_i[1] d[2]}{d[1]^2},
\end{equation*}
where $d[l] = \sum_{j \in \, _0\XX} c_i[l], \ l=1,2$.
\end{enumerate}
\end{theorem}

\begin{proof}
It follows from parts $\mathbf{(i)}$ and $\mathbf{(vi)}$ in Lemma \ref{lemmaH2} that, for $i \in \XX$,
\begin{equation} \label{thmH2quotient1}
\pi_i(\e) = \frac{e_i(\e)}{E_{ii}(\e)} = \frac{b_i[0] + b_i[1] \e + \dot{o}_i(\e)}{B_{ii}[\gamma_i-1] \e^{\gamma_i-1} + B_{ii}[\gamma_i] \e^{\gamma_i} + \dot{o}_{ii}(\e^{\gamma_i})}.
\end{equation}

We also have $\gamma_i = I(i=0) = 1-I(i \neq 0) = 1-\tilde{l}_i$. By changing the indicator function and then using the division rule for Laurent asymptotic expansions in relation (\ref{thmH2quotient1}), we obtain the asymptotic expansion given in part $\mathbf{(i)}$.

In order to prove part $\mathbf{(ii)}$ we first use part $\mathbf{(i)}$ for $i \in \, _0\XX$ to get
\begin{equation*}
\tilde{\pi}_i(\e) = \frac{\pi_i(\e)}{\sum_{j \in \, _0\XX} \pi_j(\e)} = \frac{c_i[1] \e + c_i[2] \e^2 + o_i(\e^2)}{\sum_{j \in \, _0\XX} ( c_j[1] \e + c_j[2] \e^2 + o_j(\e^2) )},
\end{equation*}
and then we apply the multiple summation rule and the division rule for asymptotic expansions in this relation.
\end{proof} 
\vspace{3mm}

{\bf 5.3. First and second order asymptotic expansions for  \\
\makebox[14mm]{} stationary and conditional quasi-stationary distributions \\
\makebox[14mm]{} under condition ${\bf H_3}$} \\

In the case where condition ${\bf H_3}$ holds, both state $0$ and state $N$ are asymptotically absorbing for the semi-Markov process. This means that $p_{0,+}(\e) \sim O(\e)$ and $p_{N,-}(\e) \sim O(\e)$ which makes the asymptotic analysis of relations (\ref{formulaEii})--(\ref{formulaENN}) even more involved.

The following lemma gives asymptotic expansions for quantities given in relations (\ref{formulaEii})--(\ref{formulaENN}).
\begin{lemma}\label{lemmaH3}
Assume that conditions  ${\bf A}$--${\bf C}$, ${\bf D}_1$--${\bf F}_1$, ${\bf G}$  and ${\bf H_3}$ hold. Then:
\begin{enumerate}
\item[$\mathbf{(i)}$] For $i \in \XX$, we have
\begin{equation*}
e_i(\e) = b_i[0] + b_i[1] \e + \dot{o}_i(\e), \ \e \in (0,\e_0],
\end{equation*}
where $\dot{o}_i(\e) / \e \to 0$ as $\e \to 0$ and
\begin{equation*}
b_i[0] = b_{i,-}[0] + b_{i,+}[0] > 0, \quad b_i[1] = b_{i,-}[1] + b_{i,+}[1].
\end{equation*}
\item[$\mathbf{(ii)}$] For $0 \leq i \leq j \leq N$, we have
\begin{equation*}
\Gamma_{i,j,+}(\e) = A_{i,j,+}[\gamma_i] \e^{\gamma_i} + A_{i,j,+}[\gamma_i+1] \e^{\gamma_i+1} + o_{i,j,+}(\e^{\gamma_i+1}), \ \e \in (0,\varepsilon_0],
\end{equation*}
where $o_{i,j,+}(\e^{\gamma_i+1}) / \e^{\gamma_i+1} \to 0$ as $\e \to 0$ and
\begin{equation*}
A_{i,j,+}[\gamma_i] = a_{i,+}[\gamma_i] a_{i+1,+}[0] \cdots a_{j,+}[0] > 0,
\end{equation*}
\begin{equation*}
A_{i,j,+}[\gamma_i+1] = \sum_{n_i + n_{i+1} + \cdots + n_j = 1} a_{i,+}[\gamma_i + n_i] a_{i+1,+}[n_{i+1}] \cdots a_{j,+}[n_j].
\end{equation*}
\item[$\mathbf{(iii)}$] For $0 \leq i \leq j \leq N$, we have
\begin{equation*}
\Gamma_{i,j,-}(\e) = A_{i,j,-}[\beta_j] \e^{\beta_j} + A_{i,j,-}[\beta_j+1] \e^{\beta_j+1} + o_{i,j,-}(\e^{\beta_j+1}), \ \e \in (0,\e_0],
\end{equation*}
where $o_{i,j,-}(\e^{\beta_j+1}) / \e^{\beta_j+1} \to 0$ as $\e \to 0$ and
\begin{equation*}
A_{i,j,-}[\beta_j] = a_{i,-}[0] \cdots a_{j-1,-}[0] a_{j,-}[\beta_j] > 0,
\end{equation*}
\begin{equation*}
A_{i,j,-}[\beta_j+1] = \sum_{n_i + \cdots + n_{j-1} + n_j = 1} a_{i,-}[n_i] \cdots a_{j-1,-}[n_{j-1}] a_{j,-}[\beta_j + n_j].
\end{equation*}
\item[$\mathbf{(iv)}$] For $0 \leq k \leq i-1, \ i \in \, _0\XX$, we have
\begin{align*}
\frac{\Gamma_{k+1,i,-}(\e)}{\Gamma_{k,i-1,+}(\e)} &= A_{k,i}^*[\beta_i-\gamma_k] \e^{\beta_i-\gamma_k} + A_{k,i}^*[\beta_i-\gamma_k+1] \e^{\beta_i-\gamma_k+1} \\
&+ o_{k,i}^*(\e^{\beta_i-\gamma_k+1}), \ \e \in (0,\e_0],
\end{align*}
where $o_{k,i}^*(\e^{\beta_i-\gamma_k+1}) / \e^{\beta_i-\gamma_k+1} \to 0$ as $\e \to 0$ and
\begin{equation*}
A_{k,i}^*[\beta_i-\gamma_k] = \frac{A_{k+1,i,-}[\beta_i]}{A_{k,i-1,+}[\gamma_k]} > 0,
\end{equation*}
\begin{align*}
&A_{k,i}^*[\beta_i-\gamma_k+1] \\
&\quad \quad = \frac{A_{k+1,i,-}[\beta_i+1] A_{k,i-1,+}[\gamma_k] - A_{k+1,i,-}[\beta_i] A_{k,i-1,+}[\gamma_k+1]}{A_{k,i-1,+}[\gamma_k]^2}.
\end{align*}
\item[$\mathbf{(v)}$] For $i+1 \leq k \leq N, \ i \in \, _N\XX$, we have
\begin{align*}
\frac{\Gamma_{i,k-1,+}(\e)}{\Gamma_{i+1,k,-}(\e)} &= A_{k,i}^*[\gamma_i-\beta_k] \e^{\gamma_i-\beta_k} + A_{k,i}^*[\gamma_i-\beta_k+1] \e^{\gamma_i-\beta_k+1} \\
&+ o_{k,i}^*(\e^{\gamma_i-\beta_k+1}), \ \e \in (0,\e_0],
\end{align*}
where $o_{k,i}^*(\e^{\gamma_i-\beta_k+1}) / \e^{\gamma_i-\beta_k+1} \to 0$ as $\e \to 0$ and
\begin{equation*}
A_{k,i}^*[\gamma_i-\beta_k] = \frac{A_{i,k-1,+}[\gamma_i]}{A_{i+1,k,-}[\beta_k]} > 0,
\end{equation*}
\begin{align*}
&A_{k,i}^*[\gamma_i-\beta_k+1] \\
&\quad \quad = \frac{A_{i,k-1,+}[\gamma_i+1] A_{i+1,k,-}[\beta_k] - A_{i,k-1,+}[\gamma_i] A_{i+1,k,-}[\beta_k+1]}{A_{i+1,k,-}[\beta_k]^2}.
\end{align*}
\item[$\mathbf{(vi)}$] For $i \in \XX$, we have
\begin{equation*}
E_{ii}(\e) = B_{ii}[\gamma_i+\beta_i-1] \e^{\gamma_i+\beta_i-1} + B_{ii}[\gamma_i+\beta_i] \e^{\gamma_i+\beta_i} + \dot{o}_{ii}(\e^{\gamma_i+\beta_i}), \ \e \in (0,\e_0],
\end{equation*}
where $\dot{o}_{ii}(\e^{\gamma_i+\beta_i}) / \e^{\gamma_i+\beta_i} \to 0$ as $\e \to 0$ and
\begin{equation*}
B_{ii}[0] = b_i[0] + b_{N-i}[0] A_{N-i,i}^*[0] > 0, \ i=0,N,
\end{equation*}
\begin{equation*}
B_{ii}[1] = b_{N-i}[1] A_{N-i,i}^*[0] + b_i[1] + \sum_{k \in \, _i\XX} b_k[0] A_{k,i}^*[1], \ i=0,N,
\end{equation*}
\begin{equation*}
B_{ii}[-1] = b_0[0] A_{0,i}^*[-1] + b_N[0] A_{N,i}^*[-1] > 0, \ i \in \, _{0,N}\XX,
\end{equation*}
\begin{equation*}
B_{ii}[0] = b_0[1] A_{0,i}^*[-1] + b_N[1] A_{N,i}^*[-1] + b_i[0] + \sum_{k \in \, _i\XX} b_k[0] A_{k,i}^*[0], \ i \in \, _{0,N}\XX.
\end{equation*}
\end{enumerate}
\end{lemma}

\begin{proof}
We first note that the quantities in parts $\mathbf{(i)}$ and $\mathbf{(ii)}$ do not depend on $p_{N,-}(\e)$, so the proofs for these parts are the same as the proofs for parts $\mathbf{(i)}$ and $\mathbf{(ii)}$ in Lemma \ref{lemmaH2}, respectively.

In order to prove part $\mathbf{(iii)}$ we notice that it follows from conditions ${\bf D}_1$ and ${\bf H_3}$ that $p_{i,-}(\e) = a_{i,-}[\beta_i] \e^{\beta_i} + a_{i,-}[\beta_i+1] + \e^{\beta_i+1} + o_{i,-}(\e^{\beta_i+1})$, $i \in \XX$. From this and the definition (\ref{defGammaij}) of $\Gamma_{i,j,-}(\e)$ we get, for $0 \leq i \leq j \leq N$,
\begin{align*}
\Gamma_{i,j,-}(\e) &= ( a_{i,-}[0] + a_{i,-}[1] \e + o_{i,-}(\e) ) \\
&\times \cdots \times \\
&\times ( a_{j-1,-}[0] + a_{j-1,-}[1] \e + o_{j-1,-}(\e) ) \\
&\times ( a_{j,-}[\beta_j] \e^{\beta_j} + a_{j,-}[\beta_j+1] + \e^{\beta_j+1} + o_{j,-}(\e^{\beta_j+1}) ).
\end{align*}

By applying the multiple product rule for asymptotic expansions we obtain the asymptotic relation given in part $\mathbf{(iii)}$.

From parts $\mathbf{(ii)}$ and $\mathbf{(iii)}$ it follows that we have, for $0 \leq k \leq i-1$, $i \in \, _0\XX$,
{\small
\begin{equation} \label{lmmH3quotient1}
\frac{\Gamma_{k+1,i,-}(\e)}{\Gamma_{k,i-1,+}(\e)} = \frac{A_{k+1,i,-}[\beta_i] \e^{\beta_i} + A_{k+1,i,-}[\beta_i+1] \e^{\beta_i+1} + o_{k+1,i,-}(\e^{\beta_i+1})}{A_{k,i-1,+}[\gamma_k] \e^{\gamma_k} + A_{k,i-1,+}[\gamma_k+1] \e^{\gamma_k+1} + o_{k,i-1,+}(\e^{\gamma_k+1})}, 
\end{equation}}
and, for $i+1 \leq k \leq N$, $i \in \, _N\XX$,
{\small 
\begin{equation} \label{lmmH3quotient2}
\frac{\Gamma_{i,k-1,+}(\e)}{\Gamma_{i+1,k,-}(\e)} = \frac{A_{i,k-1,+}[\gamma_i] \e^{\gamma_i} + A_{i,k-1,+}[\gamma_i+1] \e^{\gamma_i+1} + o_{i,k-1,+}(\e^{\gamma_i+1})}{A_{i+1,k,-}[\beta_k] \e^{\beta_k} + A_{i+1,k,-}[\beta_k+1] \e^{\beta_k+1} + o_{i+1,k,-}(\e^{\beta_k+1})}.
\end{equation}}

Notice that both in relation (\ref{lmmH3quotient1}) and (\ref{lmmH3quotient2}) it is possible that the corresponding quantity tends to infinity as $\e \to 0$. 

The asymptotic relations given in parts $\mathbf{(iv)}$ and $\mathbf{(v)}$ are obtained by using the division rule for Laurent asymptotic expansions in relations (\ref{lmmH3quotient1}) and (\ref{lmmH3quotient2}).  

We finally give the proof of part $\mathbf{(vi)}$. For the case $i=0$, it follows from relation (\ref{formulaE00}) and parts $\mathbf{(i)}$ and $\mathbf{(v)}$ that
\begin{align*}
E_{00}(\e) &= b_0[0] + b_0[1] \e + \dot{o}_0(\e) \\
&+ ( b_N[0] + b_N[1] \e + \dot{o}_N(\e) ) ( A_{N,0}^*[0] + A_{N,0}^*[1] \e + o_{N,0}^*(\e) ) \\
&+ \sum_{k \in \, _{0,N}\XX} ( b_k[0] + b_k[1] \e + \dot{o}_k(\e) ) ( A_{k,0}^*[1] \e + A_{k,0}^*[2] \e^2 + o_{k,0}^*(\e^2) ).
\end{align*}

The product rule and multiple summation rule for asymptotic expansions now proves part $\mathbf{(vi)}$ for the case $i=0$.

If $i=N$, it follows from relation (\ref{formulaENN}) and parts $\mathbf{(i)}$ and $\mathbf{(iv)}$ that
\begin{align*}
E_{NN}(\e) &= b_N[0] + b_N[1] \e + \dot{o}_N(\e) \\
&+ ( b_0[0] + b_0[1] \e + \dot{o}_0(\e) ) ( A_{0,N}^*[0] + A_{0,N}^*[1] \e + o_{0,N}^*(\e) ) \\
&+ \sum_{k \in \, _{0,N}\XX} ( b_k[0] + b_k[1] \e + \dot{o}_k(\e) ) ( A_{k,N}^*[1] \e + A_{k,N}^*[2] \e^2 + o_{k,N}^*(\e^2) ).
\end{align*}

Again, we can use the product rule and multiple summation rule in order to prove part $\mathbf{(vi)}$, in this case, for $i=N$.

For the case where $i \in \, _{0,N}\XX$, we use relation (\ref{formulaEii}) and parts $\mathbf{(i)}$, $\mathbf{(iv)}$, and $\mathbf{(v)}$ to get
\begin{align*}
E_{ii}(\e) &= b_i[0] + b_i[1] \e + \dot{o}_i(\e) \\
&+ \sum_{k \in \{0,N\}} (b_k[0] + b_k[1] \e + \dot{o}_k(\e)) (A_{k,i}^*[-1] \e^{-1} + A_{k,i}^*[0] + o_{k,i}^*(1)) \\
&+ \sum_{k \in \, _{0,i,N}\XX} (b_k[0] + b_k[1] \e + \dot{o}_k(\e)) (A_{k,i}^*[0] + A_{k,i}^*[1] \e + o_{k,i}^*(\e)).
\end{align*}

Here we can note that the terms corresponding to $k \in \{0,N\}$ are of order $O(\e^{-1})$ while all other terms are of order $O(1)$. By using the product rule and multiple summation rule for Laurent asymptotic expansions, we conclude that the asymptotic relation given in part $\mathbf{(vi)}$ also holds for $i \in \, _{0,N}\XX$.
\end{proof}

The following theorem gives second order asymptotic expansions for stationary and conditional quasi-stationary probabilities. In particular, part $\mathbf{(i)}$ of this theorem shows that there exist limits for stationary probabilities, $\pi_i(0) = \lim_{\e \to 0} \pi_i(\e)$, $i \in \XX$, where $\pi_0(0) > 0$, $\pi_N(0) > 0$, and $\pi_i(0) = 0$ for $i \in \, _{0,N}\XX$. Furthermore, part $\mathbf{(ii)}$ of the theorem shows, in particular, that there exist limits for conditional quasi-stationary probabilities, $\hat{\pi}_i(0) = \lim_{\e \to 0} \hat{\pi}_i(\e)$, $i \in \, _{0,N}\XX$, where $\hat{\pi}_i(0) > 0$, $i \in \, _{0,N}\XX$. The corresponding higher order asymptotic expansions are given in Theorem 5.
\begin{theorem}
Assume that conditions  ${\bf A}$--${\bf C}$, ${\bf D}_1$--${\bf F}_1$, ${\bf G}$  and ${\bf H_3}$ hold. Then:
\begin{enumerate}
\item[$\mathbf{(i)}$] We have the following asymptotic relation for the stationary probabilities $\pi_i(\e)$, $i \in \XX$,
\begin{equation*}
\pi_i(\e) = c_i[\hat{l}_i] \e^{\hat{l}_i} + c_i[\hat{l}_i+1] \e^{\hat{l}_i+1} + o_i(\e^{\hat{l}_i+1}), \ \e \in (0,\e_0],
\end{equation*}
where $\hat{l}_i = I(i \neq 0,N)$, $o_i(\e^{\hat{l}_i+1}) / \e^{\hat{l}_i+1} \to 0$ as $\e \to 0$, and
\begin{equation*}
c_i[\hat{l}_i] = \frac{b_i[0]}{B_{ii}[-\hat{l}_i]} > 0, \quad c_i[\hat{l}_i+1] = \frac{b_i[1] B_{ii}[-\hat{l}_i] - b_i[0] B_{ii}[-\hat{l}_i+1]}{B_{ii}[-\hat{l}_i]^2},
\end{equation*}
where $B_{ii}[-1]$, $i \in \, _{0,N}\XX$, $B_{ii}[0]$, $i \in \XX$, and $B_{ii}[1]$, $i=0,N$, can be computed from the formulas given in Lemma \ref{lemmaH3}.
\item[$\mathbf{(ii)}$] We have the following asymptotic relation for the conditional quasi-stationary probabilities, $\hat{\pi}_i(\e)$, $i \in \, _{0,N}\XX$,
\begin{equation*}
\hat{\pi}_i(\e) = \hat{c}_i[0] + \hat{c}_i[1] \e + \hat{o}_i(\e), \ \e \in (0,\e_0],
\end{equation*}
where $\hat{o}_i(\e) / \e \to 0$ as $\e \to 0$ and
\begin{equation*}
\hat{c}_i[0] = \frac{c_i[1]}{d[1]} > 0, \quad \hat{c}_i[1] = \frac{c_i[2] d[1] - c_i[1] d[2]}{d[1]^2},
\end{equation*}
where $d[l] = \sum_{j \in \, _{0,N}\XX} c_i[l], \ l=1,2$.
\end{enumerate}
\end{theorem}

\begin{proof}
It follows from parts $\mathbf{(i)}$ and $\mathbf{(vi)}$ in Lemma \ref{lemmaH3} that, for $i \in \XX$,
\begin{equation} \label{thmH3quotient1}
\pi_i(\e) = \frac{b_i[0] + b_i[1] \e + \dot{o}_i(\e)}{B_{ii}[\gamma_i+\beta_i-1] \e^{\gamma_i+\beta_i-1} + B_{ii}[\gamma_i+\beta_i] \e^{\gamma_i+\beta_i} + \dot{o}_{ii}(\e^{\gamma_i+\beta_i})}.
\end{equation}

We also have
\begin{equation*}
\gamma_i + \beta_i = I(i=0) + I(i=N) = 1 - I(i \neq 0,N) = 1 - \hat{l}_i.
\end{equation*}

Using this relation for indicator functions and the division rule for Laurent asymptotic expansions in relation (\ref{thmH3quotient1}) we obtain the asymptotic relation given in part $\mathbf{(i)}$.

For the proof of part $\mathbf{(ii)}$, we first use part $\mathbf{(i)}$ for $i \in \, _{0,N}\XX$ to get
\begin{equation*}
\hat{\pi}_i(\e) = \frac{\pi_i(\e)}{\sum_{j \in \, _{0,N}\XX} \pi_j(\e)} = \frac{c_i[1] \e + c_i[2] \e^2 + o_i(\e^2)}{\sum_{j \in \, _{0,N}\XX} ( c_j[1] \e + c_j[2] \e^2 + o_j(\e^2) )},
\end{equation*}
and then we apply the multiple summation rule and the division rule for asymptotic expansions in this relation.
\end{proof} \vspace{3mm}

{\bf 6. Recurrent algorithms for asymptotic expansions for \\ 
\makebox[10.5mm]{} stationary and conditional quasi-stationary distributions} \\

In this section, we generalize results given in Section 5 and  describe general recurrent algorithms for construction of asymptotic expansions for stationary and conditional quasi-stationary distributions for perturbed birth-death semi-Markov processes based on sequential reduction of phase spaces. \\

{\bf 6.1 Laurent asymptotic expansions}. \\   

In this subsection, we present some operational rules for Laurent asymptotic expansions given in Silvestrov, D and Silvestrov S. (2015, 2016) and used in the present  paper for constructions of asymptotic expansions for stationary and conditional quasi-stationary distributions of perturbed semi-Markov processes.

A real-valued function $A(\e)$, defined on an interval $(0, \e_0]$ for some $0 < \e_0 \leq 1$,  is a Laurent asymptotic expansion 
if it can be represented in the following form, $A(\e) = a_{h_A}\e^{h_A} + \cdots + a_{k_A}\e^{k_A} + 
o_A(\e^{k_A}), \e \in (0, \e_0]$, where {\bf (a)} $- \infty < h_A \leq k_A < \infty$ are integers, {\bf
(b)} coefficients $a_{h_A}, \ldots, a_{k_A}$ are real numbers, {\bf (c)} function $o_A(\e^{k_A})/\e^{k_A} \rightarrow 0$ as $\e \rightarrow 0$. Such expansion ais  pivotal if it is known that $a_{h_A} \neq 0$.

The  above paper presents operational rules for Laurent asymptotic expansions. Let shortly formulate some of these rules, in particular, for summation, multiplication, division of  Laurent asymptotic expansions. \vspace{2mm}

\begin{lemma} Let $A(\e) = a_{h_A}\e^{h_A} + \cdots + a_{k_A}\e^{k_A} + o_A(\e^{k_A})$ and $B(\e) = b_{h_B}\e^{h_B} + \cdots + b_{k_B}\e^{k_B} + o_B(\e^{k_B})$ be two pivotal Laurent asymptotic expansions. Then{\rm :} \vspace{2mm}

{\bf (i)} $C(\e) = cA(\e) = c_{h_{C}}\e^{h_{C}} + \cdots + c_{k_C}\e^{k_C} + o_C(\e^{k_C})$, where a constant $c \neq 0$,  is a pivotal Laurent asymptotic expansion and  $h_C = h_A, k_C = k_A,  c_{h_C + r}  = c a_{h_C + r}, r = 0,  \ldots, k_C - h_C$, \vspace{2mm}

{\bf (ii)} $D(\e) = A(\e) + B(\e) = d_{h_{D}}\e^{h_{D}} + \cdots + d_{k_D}\e^{k_D} + o_D(\e^{k_D})$ is a pivotal Laurent asymptotic expansion and $h_D = h_A \wedge h_B, k_D = k_A \wedge k_B, d_{h_D + r} = a_{h_D + r}  + b_{h_D + r}, r = 0, \ldots, k_D - h_D$, where $a_{h_D + r} = 0, r < h_A - h_D, b_{h_D + r} = 0, 
r < h_B - h_D$, \vspace{2mm}
 
{\bf (iii)}   $E(\e) = A(\e) \cdot B(\e) = e_{h_{E}}\e^{h_{E}} + \cdots + e_{k_E}\e^{k_E} + 
o_E(\e^{k_E})$  is a pivotal Laurent asymptotic expansion and $h_E = h_A + h_B, k_E = (k_A + h_B) \wedge (k_B + h_A), e_{h_E + r} = \sum_{l = 0}^r a_{h_A + l}  \cdot b_{h_B + r - l}, r = 0, \ldots, k_E - h_E$,  \vspace{2mm}

{\bf (iv)}  $F(\e) = A(\e)/B(\e) = f_{h_{F}}\e^{h_{F}} + \cdots + f_{k_F}\e^{k_F} + o_F(\e^{k_F})$ ia a pivotal Laurent asymptotic expansion and   $h_F = h_A - h_B, \, k_F = (k_A - h_B)  \wedge (k_B - 2 h_B + h_A)$, $f_{h_F + r} = \frac{1}{b_{h_B}}( a_{h_A + r} +  \sum_{l = 1}^r b_{h_B + l}  \cdot f_{h_F + r - l}), r = 0, \ldots, k_F - h_F$. 
\end{lemma}

The following lemma presents useful multiple generalizations of summation and multiplication rules  given in Lemma 4. 
\begin{lemma} Let $A_i(\e) = a_{i, h_{A_i}}\e^{h_A} + \cdots + a_{i, k_{A_i}}\e^{k_{A_i}} + o_{A_i}(\e^{k_{A_i}}), i = 1, \ldots, m$  be  pivotal Laurent asymptotic expansions. Then{\rm :} \vspace{2mm}
 
 {\bf (i)} $D(\e) = \sum_{i = 1}^m A_i(\e) = d_{h_{D}}\e^{h_{D}} + \cdots + d_{k_{D}}\e^{k_{D}} + o_{D}(\e^{k_{D}})$ is a pivotal Laurent asymptotic expansion and  $h_{D} = \min_{1 \leq l \leq m} h_{A_l},  \, k_{D} =  \min_{1 \leq l \leq m} k_{A_l}$, $d_{h_{D} + l} =  a_{1, h_{D} + l} + \cdots + a_{m, h_{D} + l}, \  
 l = 0, \ldots, k_{D}  - h_{D}$, where  $a_{i, h_{D} + l} = 0$ for $0 \leq l < h_{A_i} - h_{D}, i = 1, \ldots, m$, \vspace{2mm}

 {\bf (ii)}   $E(\e) = \prod_{i = 1}^m A_i(\e)  = e_{h_{E}}\e^{h_{E}} + \cdots + e_{k_E}\e^{k_E} + 
o_E(\e^{k_E})$  is a pivotal Laurent asymptotic expansion and 
$h_{E} =  \sum_{l = 1}^m h_{A_l}, \,  k_{E} =  \min_{1 \leq l \leq m}(k_{A_l} + \sum_{1 \leq r \leq m, r \neq l} h_{A_r})$, 
 $e_{h_{E} + l} =  \sum_{l_1 + \cdots + l_m = l, 0 \leq l_i \leq k_{A_i}  - h_{A_i}, i = 1, \ldots, m}  \, \prod_{1 \leq i \leq m} a_{i, h_{A_{i}} + l_i}$, \,  $l = 0, \ldots, k_{E}  - h_{E}$. 
\end{lemma}

{\bf 6.2. Laurent asymptotic expansions for  transition \mbox{characteristics}  \\
\makebox[13.5mm]{} of perturbed  birth-death semi-Markov processes  based on \\ 
\makebox[13.5mm]{} sequential reduction of phase spaces} \\

In this subsection, we show how  recurrent algorithms for constructions of Laurent asymptotic expansions for stationary and conditional quasi-stationary distributions of perturbed semi-Markov processes developed in Silvestrov, D and Silvestrov S. (2015, 2016) and shortly presented in Lemmas 4 and 5 can be applied to the model of perturbed birth-death semi-Markov processes, which is studied in the present paper.

The operational rules for Laurent asymptotic expansions formulated in Lemmas  4 and 5 can be applied to recurrent formulas for computing transition probabilities and expectations of transition times  for reduced perturbed birth-death semi-Markov processes (\ref{takaskao}) -- (\ref{noply}).

In this context, it is convenient to   re-write the  asymptotic expansions penetrating perturbation conditions ${\bf D}_L$ and ${\bf E}_L$ in the equivalent pivotal form  $p_{i, \pm}(\e) =  \sum_{l = l_{i, \pm}}^{L  + l_{i, \pm}} a_{i, \pm}[l] \e^l + o_{i, \pm}(\e^{L + l_{i, \pm}})$ and $e_{i, \pm}(\e) =  \sum_{l = l_{i, \pm}}^{L  + l_{i, \pm}} b_{i, \pm}[l] \e^l + \dot{o}_{i, \pm}(\e^{L + l_{i, \pm}}), i \in \XX$. 

These asymptotic expansions  play the role of boundary conditions, which give explicit formulas for parameters  and coefficients of the asymptotic expansions 
$_{\langle 0, N \rangle}p_{i, \pm}(\e) = p_{i, \pm}(\e) =  \sum_{l = l_{i, \pm}}^{L  + l_{i, \pm}} \, _{\langle 0, N \rangle}a_{i, \pm}[l] \e^l$ $+ \, _{\langle 0, N \rangle}o_{i, \pm}(\e^{L + l_{i, \pm}})$ and $_{\langle 0, N \rangle}e_{i, \pm}(\e) \ = \ e_{i, \pm}(\e) \  =  \ \sum_{l = l_{i, \pm}}^{L  + l_{i, \pm}} \, _{\langle 0, N \rangle}b_{i, \pm}[l] \e^l \ +$ $_{\langle 0, N \rangle}\dot{o}_{i, \pm}(\e^{L + l_{i, \pm}})$, for transition characteristics of the initial semi-Markov processes $ _{\langle 0, N \rangle}\eta^{(\e)}(t) = \eta^{(\e)}(t)$, namely, 
$_{\langle 0, N \rangle}a_{i, \pm}[l] = a_{i, \pm}[l], \  _{\langle 0, N \rangle}b_{i, \pm}[l] = b_{i, \pm}[l], \ l = l_{i, \pm}, \ldots, L + l_{i, \pm}, i \in \XX$.

All recurrent formulas in relations  (\ref{takaskao}) -- (\ref{noply}), except  the formulas for expectations $_{\langle k, r \rangle}e_{k, -}(\e), \, _{\langle k, r \rangle}e_{r, +}(\e), 1 \leq k \leq r 
\leq N$ given, respectively,  in the third lines of relations (\ref{takaso}) and (\ref{akaso}), have the form of simple identities of the type $_{\langle k, r \rangle}e_{i, \pm}(\e)  = \, _{\langle k-1, r \rangle}e_{i, \pm}(\e)$ and $_{\langle k, r \rangle}e_{i, \pm}(\e)  = \, _{\langle k, r +1 \rangle}e_{i, \pm}(\e)$. This   implies equalities for the corresponding parameters and coefficients of  asymptotic expansions for transition characteristics on the left and right hand sides of the corresponding identities. 

As far as the identities for expectations $_{\langle k, r \rangle}e_{k, -}(\e), \, _{\langle k, r \rangle}e_{r, +}(\e), 1 \leq k \leq r 
\leq N$ are concerned, the operational rules for Laurent asymptotic expansions mentioned above should be applied to the expressions on the right hand side of the corresponding identities. 
For example, let us describe the recurrent algorithm for computing asymptotic expansions for expectations $_{\langle k, r \rangle}e_{k, -}(\e), 1 \leq k \leq r \leq N$.

First, the division rule {\bf (iv)} given in Lemma 4  should be applied to the quotients $\frac{_{\langle k-1, r \rangle}p_{k, -}(\e)}{_{\langle k-1, r \rangle}p_{k-1, +}(\e)}$, $1 \leq k \leq r \leq N$. Second summation rule {\bf (ii)} given in Lemma 4    should be applied to sums $_{\langle k-1, r \rangle}e_{k-1}(\e) = \,  _{\langle k-1, r \rangle}e_{k-1, +}(\e)  + \, _{\langle k-1, r \rangle}e_{k-1, -}(\e), 1 \leq k \leq N$. Third, the multiplication rule {\bf (iii)} given in Lemma 4    should be applied  to the products  $_{\langle k-1, r \rangle}e_{k-1}(\e) \cdot \frac{_{\langle k-1, r \rangle}p_{k, -}(\e)}{_{\langle k-1, r \rangle}p_{k-1, +}(\e)},1 \leq k \leq r \leq N$. Fourth, the summation rule {\bf (ii)} given in Lemma 4   should be applied to the sum  $_{\langle k, r \rangle}e_{k, -}(\e)  =  \, _{\langle k-1, r \rangle}e_{k, -}(\e)   + \, _{\langle k-1, r \rangle}e_{k-1}(\e) \cdot \frac{_{\langle k-1, r \rangle}p_{k, -}(\e)}{_{\langle k-1, r \rangle}p_{k-1, +}(\e)}$,  $1 \leq k \leq r \leq N$.
This will yield the corresponding asymptotic expansion and recurrent formulas for parameters and coefficients of the Laurent asymptotic expansions for 
expectations $_{\langle k, r \rangle}e_{k, -}(\e), 1 \leq k \leq r \leq N$.

The asymptotic expansions, which are based on the identities  $_{\langle k, r \rangle}e_{r, +}(\e)  =$  $_{\langle k, r +1 \rangle}e_{r, +}(\e)  + \, _{\langle k, r +1 \rangle}e_{r+1}(\e) \cdot \frac{_{\langle k, r +1 \rangle}p_{r, +}(\e)}{_{\langle k, r +1 \rangle}p_{r+1, -}(\e)}, 0 \leq k \leq r \leq N-1$,  can be obtained in analogous way.

The Laurent asymptotic expansion for the expectations of return times $E_{ii}(\e) =  \, _{\langle i, i \rangle}e_{i}(\e) =  \, _{\langle i, i \rangle}e_{i, +}(\e) + \, _{\langle i, i \rangle}e_{i, -}(\e), i \in \XX$  can be obtained  by application of the summation rule {\bf (ii)} given in Lemma 4   to the above sums $_{\langle i, i \rangle}e_{i, +}(\e) + \, _{\langle i, i \rangle}e_{i, -}(\e)$.

Alternatively, the asymptotic expansions for the expectations of return times $E_{ii}(\e)$ can be obtained, first, by application of the multiple multiplication rule given in Lemma 5 to the products representing numerators and denominators in fractions penetrating expressions for these expectations given in relation
(\ref{opada}), second,  the division rule given in Lemma 4  to these fractions and, third, the
multiple summation rule given in Lemma 5  to the sums of the above fractions representing expectations $E_{ii}(\e)$
in relation (\ref{opada}).

Both algorithms give the same  Laurent asymptotic expansions for expectations $E_{ii}(\e)$. The only difference is in forms of recurrent formulas for coefficients in  the corresponding expansions. \\
\vspace{1mm} 

{\bf 6.3. Laurent asymptotic expansions for stationary and \\ 
\makebox[14mm]{} conditional quasi-stationary distributions of  \\ 
\makebox[14mm]{} perturbed  birth-death semi-Markov processes} \\ 

At the final step, the asymptotic expansion for stationary probabilities $\pi_i(\e) = \frac{e_i(\e)}{E_{ii}(\e)}, \ i \in \XX$ and then conditional quasi-stationary probabilities $\tilde{\pi}_i(\e) = \frac{\pi_i(\e)}{1 - \pi_0(\e)}, \ i \in \, _0\XX$ and $\hat{\pi}_i(\e) = \frac{\pi_i(\e)}{1 - \pi_0(\e) -  \pi_N(\e)}, \ i \in \, _{0, N}\XX$ can be obtained by application of the division rule  {\bf (v)} given in Lemma 4  to the quotients defining these quantities.

Explicit expressions for coefficients of  the Laurent asymptotic expansions for  the expectations of return times $E_{ii}(\e)$ and then for stationary and conditional quasi-stationary distributions can be obtained using recurrent formulas for coefficients of  Laurent asymptotic expansions for transition characteristics of perturbed reduced semi-Markov processes. This can be done in the same way, as the explicit expressions for expectations $E_{ii}(\e)$ itself have been obtained with the use of recurrent formulas (\ref{emsemitalop}) -- (\ref{opadada})  for  transition characteristics of perturbed reduced semi-Markov processes.  
 
The corresponding explicit recurrent formulas for coefficients of the corresponding asymptotic expansions are different for the cases, where 
condition ${\bf H_1}$, ${\bf H_2}$ or ${\bf H_3}$ hold. We omit the long technical calculations and just formulate the corresponding final 
result.  
\begin{theorem} Let conditions   ${\bf A}$ --  ${\bf C}$, ${\bf D}_L$ -- ${\bf F}_L$ and ${\bf G}$ hold. In this case{\rm :} \vspace{1mm}

{\bf (i)} Condition ${\bf H_1}$ implies that  the following asymptotic expansions for stationary probabilities $\pi_i(\e), i \in \XX$ take place,
 \begin{equation}\label{statar}
 \pi_i(\e) =   \sum_{l = 0}^{L} c_{i}[l]\e^l + o_{i}(\e^{L}), \, \e \in (0, \e_0], 
 \end{equation}
 where{\rm :}   {\bf (b)}  $|c_{i}[l]|  < \infty, 0 \leq l \leq L, i \in \XX;$   
 {\bf (b)} $\pi_i(0) =  c_{i}[0]  > 0, i \in \XX$ and $\sum_{i \in \XX} \pi_{i}(0)  = 1;$ {\bf (c)}  $\sum_{i \in \XX} c_{i}[l]  = 0, 1 \leq l \leq L;$ {\bf (e)} $ o_{i}(\e^{L })/\e^{L} \to 0$ as $\e \to 0$, for $i \in \XX$. \vspace{1mm}
 
{\bf (ii)} Condition  ${\bf H_2}$ implies that  the following asymptotic expansions for for stationary probabilities $\pi_i(\e), i \in \XX$ take place,
 \begin{equation}\label{statara}
 \pi_i(\e) =   \sum_{l = \tilde{l}_i}^{L +  \tilde{l}_i} c_{i}[l]\e^l + o_{i}(\e^{L +  \tilde{l}_i}), \, \e \in (0, \e_0], 
 \end{equation}
 where: {\bf (a)} $\tilde{l}_i = I(i \neq 0), i \in \XX;$   {\bf (b)}  $|c_{i}[l]|  < \infty, \tilde{l}_i  \leq l \leq L + \tilde{l}_i , i \in \XX;$   
 {\bf (c)}  $\pi_0(0) =  c_{0}[0]  = 1$ and $\pi_i(0) =  c_{0}[0] = 0, c_{0}[1] > 0, i \in \, _0\XX;$ {\bf (d)}   $\sum_{i \in \XX} c_{i}[l]  = 0, 1 \leq l \leq L;$ {\bf (e)} $ o_{i}(\e^{L  + \tilde{l}_i})/\e^{L+  \tilde{l}_i} \to 0$ 
  as $\e \to 0$, for $i \in \XX$. \vspace{1mm}

{\bf (iii)} Condition  ${\bf H_2}$ implies that  the following asymptotic expansions for conditional quasi-stationary probabilities $\tilde{\pi}_i(\e), i \in \, _0\XX$ takes place,
 \begin{equation}\label{statarb}
\tilde{\pi}_i(\e) =   \sum_{l = 0}^{L} \tilde{c}_{i}[l]\e^l + \tilde{o}_{i}(\e^{L}), \, \e \in (0, \e_0], 
 \end{equation}
 where: {\bf (a)} $| \tilde{c}_{i}[l]|  < \infty, 0 \leq l \leq L, i \in \, _0\XX$;  {\bf (b)} $\tilde{\pi}_i(0)  =  \tilde{c}_{i}[0]  > 0, i \in \, _0\XX$ and  $\sum_{i \in \, _0\XX}  \tilde{\pi}_{i}(0)  = 1$;   {\bf (c)}  $\sum_{i \in \, _0\XX}  \tilde{c}_{i}[l]  = 0, 1 \leq l \leq L$; {\bf (d)} $ \tilde{o}_{i}(\e^{L})/\e^{L} \to 0$ as $\e \to 0$, for $i \in \, _0\XX$. \vspace{1mm}
 
 {\bf (iv)} Condition  ${\bf H_3}$ implies that  the following asymptotic expansions for for stationary probabilities $\pi_i(\e), i \in \XX$ take place,
 \begin{equation}\label{statarc}
 \pi_i(\e) =   \sum_{l = \tilde{l}_i}^{L +  \hat{l}_i} c_{i}[l]\e^l + o_{i}(\e^{L +  \tilde{l}_i}), \, \e \in (0, \e_0], 
 \end{equation}
 where: {\bf (a)} $\hat{l}_i = I(i \neq 0, N), i \in \XX;$   {\bf (b)}  $|c_{i}[l]|  < \infty, \hat{l}_i  \leq l \leq L + \hat{l}_i , i \in \XX;$   
 {\bf (c)}  $\pi_i(0) =  c_{i}[0]  > 0, i = 0, N$, $\pi_i(0) =  c_{i}[0]  = 0, c_{i}[1] > 0,   i \in \, _{0, N}\XX$ and  $\pi_0(0) + \pi_1(0) = 1;$ {\bf (d)}   $\sum_{i \in \XX} c_{i}[l]  = 0, 1 \leq l \leq L;$ {\bf (e)} $ o_{i}(\e^{L  + \hat{l}_i})/\e^{L+  \hat{l}_i} \to 0$ as $\e \to 0$, for $i \in \XX$. \vspace{1mm}

 {\bf (v)} Condition  ${\bf H_3}$  implies that  the following asymptotic expansions for conditional quasi-stationary probabilities 
 $\hat{\pi}_i(\e), i \in \, _{0, N}\XX$,
 \begin{equation}\label{statard}
\hat{\pi}_i(\e) =   \sum_{l = 0}^{L} \hat{c}_{i}[l]\e^l + \hat{o}_{i}(\e^{L}), \, \e \in (0, \e_0], 
 \end{equation}
 where: {\bf (a)} $| \hat{c}_{i}[l]|  < \infty, 0 \leq l \leq L, i \in \, _{0, N}\XX$;  {\bf (b)} $ \hat{\pi}_i(0)  =  \hat{c}_{i}[0]  > 0, i \in  \, _{0, N}\XX$ and  
 $\sum_{i \in  \, _{0, N}\XX}  \hat{\pi}_{i}(0)  = 1$;  
  {\bf (c)}  $\sum_{i \in \, _{0, N}\XX}  \hat{c}_{i}[l]  = 0, 1 \leq l \leq L$; {\bf (d)} $ \hat{o}_{i}(\e^{L})/\e^{L} \to 0$ as $\e \to 0$, for $i \in \, _{0, N}\XX$.
  \end{theorem}
  
Asymptotic expansions (\ref{statar}) -- (\ref{statard}) and explicit recurrent formulas for coefficients in these expansions  can be obtained by application of the recurrent algorithm presented in Subsection 6.1 and  based on application operational rules for Laurent asymptotic expansions given in Lemmas 4 and 5 to transition characteristics of reduced birth-death semi-Markov processes given by recurrent formulas  (\ref{takaskao}) -- (\ref{noply}) and quotient formulas for stationary and quasi-stationary probabilities given above.  

We omit the corresponding long technical calculations and just mention that the corresponding  explicit expressions for the   first and the second  coefficients in these expansions are given in  Section 5. \\

{\bf 7. Examples} \\

In this section, the results of the present paper are illustrated by numerical examples for some of the perturbed models of birth-death-type discussed in Section 3.

Let us first note that each model presented in Section 3 is defined in terms of intensities for a continuous time Markov chain and the perturbation scenarios considered give intensities which are linear functions of the perturbation parameter, that is,
\begin{equation} \label{lipmLinear}
\lambda_{i,\pm}(\e) = g_{i,\pm}[0] + g_{i,\pm}[1] \e, \ i \in \XX,
\end{equation}
where the coefficients $g_{i,\pm}[l]$ depend on the model under consideration. Consequently, the higher order ($l\ge 2$) terms in (\ref{lambdaexp}) all vanish.

In order to use the algorithm based on successive reduction of the phase space, we first need to calculate the coefficients in perturbation conditions ${\bf D}_L$ and ${\bf E}_L$. This can be done from relations (\ref{laDef}), (\ref{pDef}), and (\ref{eipm}) by applying the operational rules for Laurent asymptotic expansions given in Lemmas 4 and 5.

By relation (\ref{laDef}), we have $\lambda_i(\e) = \lambda_{i,-}(\e) + \lambda_{i,+}(\e)$, so it follows immediately from (\ref{lipmLinear}) that
\begin{equation} \label{liLinear}
\lambda_i(\e) = g_i[0] + g_i[1] \e, \ i \in \XX,
\end{equation}
where $g_i[l] = g_{i,-}[l] + g_{i,+}[l]$, $l=0,1$.

From (\ref{pDef}), (\ref{lipmLinear}), (\ref{liLinear}), and Lemma 4 we deduce the following asymptotic series expansions, for $i \in \XX$,
\begin{equation} \label{expansionp}
p_{i,\pm}(\e) = \frac{\lambda_{i,\pm}(\e)}{\lambda_i(\e)} = \frac{g_{i,\pm}[0] + g_{i,\pm}[1] \e}{g_i[0] + g_i[1] \e} = \sum_{l = l_{i,\pm}}^{L + l_{i,\pm}} a_{i,\pm}[l] \e^l + o_{i,\pm}(\e^{L + l_{i,\pm}}).
\end{equation}
The expansion (\ref{expansionp}) exists for any integer $L \geq 0$ and its coefficients can be calculated from the division rule for asymptotic expansions.

Then, using (\ref{eipm}), (\ref{liLinear}), (\ref{expansionp}), and Lemma 4, the following asymptotic series expansions can be constructed, for $i \in \XX$,
\begin{equation} \label{expansione}
e_{i,\pm}(\e) = \frac{p_{i,\pm}(\e)}{\lambda_i(\e)} = \sum_{l = l_{i,\pm}}^{L + l_{i,\pm}} b_{i,\pm}[l] \e^l + \dot{o}_{i,\pm}(\e^{L + l_{i,\pm}}).
\end{equation}

Once the coefficients in the expansions (\ref{expansionp}) and (\ref{expansione}) have been calculated for some integer $L \geq 0$, we can use the algorithm described in Section 6 in order to construct asymptotic expansions for stationary and conditional quasi-stationary probabilities.

The remainder of this section is organized as follows. In Section 7.1 we illustrate our results with numerical calculations for the perturbed models of population genetics discussed in Section 3.3. We first consider an example where condition ${\bf H_1}$ holds and then an example where condition ${\bf H_3}$ is satisfied. Numerical illustrations for the perturbed model of epidemics presented in Section 3.2 is given in Section 7.2. This provides an example where condition ${\bf H_2}$ holds. \\

{\bf 7.1. Numerical examples for perturbed models of  \\ 
\makebox[14mm]{} population genetics} \\

Recall that the perturbation conditions for the model in Section 3.3 are formulated in terms of the mutation parameters as
\begin{equation} \label{perturbationU}
U_1(\e) = C_1 + D_1 \e, \quad U_2(\e) = C_2 + D_2 \e.
\end{equation}

Additionally, the model depends on the size $N/2$ of the population and the selection parameters $S_1$ and $S_2$ which are assumed to be independent of $\e$. Thus, there are in total seven parameters to choose.

In our first example, we choose the following values for the parameters: $N = 100$, $C_1 = C_2 = 5$, $D_1 = 0$, $D_2 = N$, and $S_1 = S_2 = 0$. Recall that the mutation probabilities are related to the mutation parameters by $u_1(\e) = U_1(\e) / N$ and $u_2(\e) = U_2(\e) / N$. It follows from \eqref{perturbationU} that $u_1(\e) = 0.05$ and $u_2(\e) = 0.05 + \e$. Thus, in the limiting model, a chosen allele mutates with probability $0.05$ for both types $A_1$ and $A_2$. In this case, we have no absorbing states which means that condition ${\bf H_1}$ holds.

Since we have no selection, the stationary distribution for the limiting model will be symmetric around state $50$. The perturbation parameter $\e$ can be interpreted as an increase in the probability that a chosen allele of type $A_2$ mutates to an allele of type $A_1$. Increasing the perturbation parameter will shift the mass of the stationary distribution to the right.

With model parameters given above, we first used relations (\ref{laij}), (\ref{xast}), (\ref{xastast}), (\ref{SmallPar}), and (\ref{ue}) to calculate the coefficients in (\ref{lipmLinear}) for the intensities. Then, these coefficients were used to compute the coefficients in the perturbation conditions ${\bf D}_L$ and ${\bf E}_L$ as described above. After this, we used the algorithm outlined in Section 6 to calculate the asymptotic expansions for the stationary distribution given by (\ref{statar}) for $L = 3$. In particular, this also gives the corresponding asymptotic expansions for $L = 0,1,2$. Approximations for the stationary distribution based on these expansions were obtained by approximating the residual term by zero.

\begin{figure}
\centering
\includegraphics[page=1,scale=0.65]{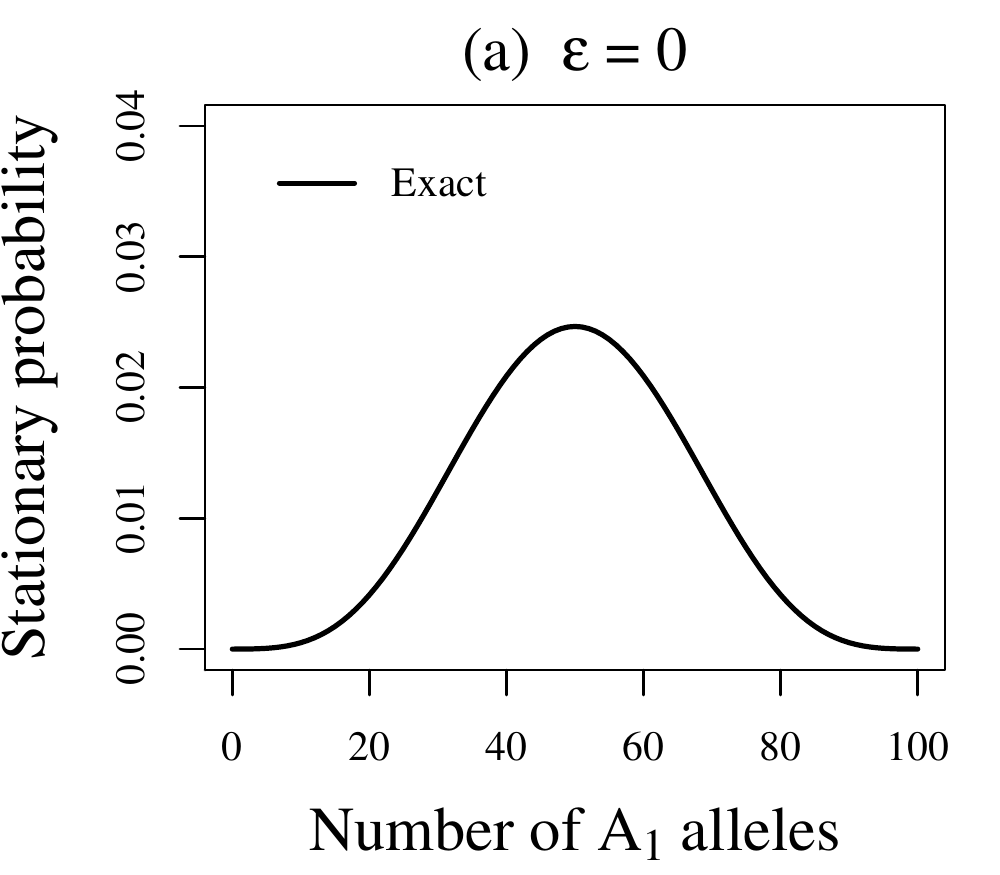}
\includegraphics[page=2,scale=0.65]{Rplots.pdf} \\
\vspace{5pt}
\includegraphics[page=3,scale=0.65]{Rplots.pdf}
\includegraphics[page=4,scale=0.65]{Rplots.pdf}
\caption{Comparison of the stationary distribution $\pi_i(\varepsilon)$ and some of its approximations for the population genetic example of Section 3.3.
The plots are functions of the number of $A_1$ alleles $i$, for different values of the perturbation parameter $\varepsilon$, with $N = 100$, $C_1 = C_2 = 5$, $D_1 = 0$, $D_2 = N$, and $S_1 = S_2 = 0$.} \label{figH1stationary}
\end{figure}

Let us first compare our approximations with the exact stationary distribution for some particular values of the perturbation parameter. Figure \ref{figH1stationary} (a) shows the stationary distribution for the limiting model ($\e = 0$) and, as already mentioned above, we see that it is symmetric around state $50$. The stationary distribution for the model with $\e = 0.01$ and the approximation corresponding to $L = 1$ are shown in Figure \ref{figH1stationary} (b). Here, the approximation seems the match the exact distribution very well. The approximation for $L = 2$ is not included here since it will not show any visible difference from the exact stationary distribution. In Figures \ref{figH1stationary} (c) and \ref{figH1stationary} (d), correspoding to the models where $\e = 0.02$ and $\e = 0.03$, respectively, we also include the approximations for $L = 2$. As expected, the approximations for the stationary distribution get worse as the perturbation parameter increases. However, it seems that even for higher values of the perturbation parameter, some parts of our approximations fit better to the exact stationary distribution. In this example, it seems that the approximations are in general better for states that belong to the right part of the distribution.

\begin{figure}
\centering
\includegraphics[page=5,scale=0.65]{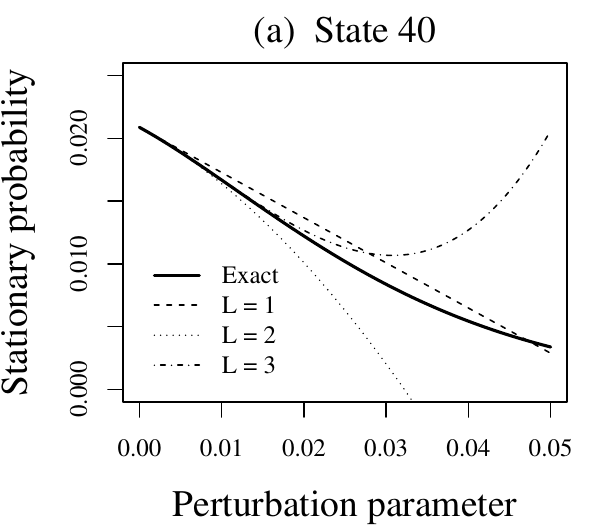}
\includegraphics[page=6,scale=0.65]{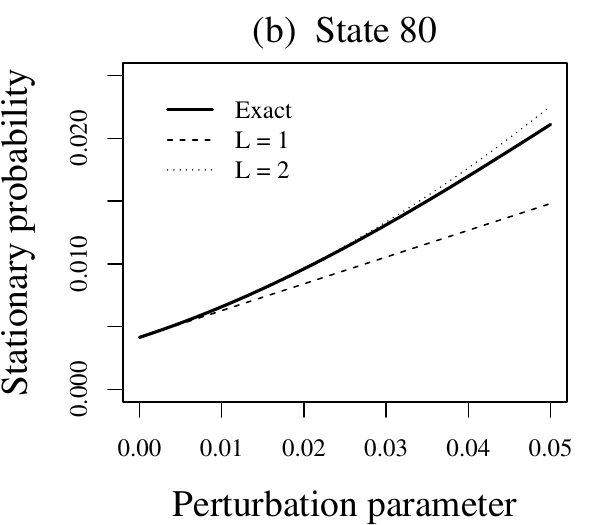}
\caption{Comparison of stationary probabilities $\pi_i(\varepsilon)$ for states $i=40$ and $i=80$ and some of its approximations considered as a function of the perturbation parameter $\varepsilon$. The model is based on the population genetic example of Section 3.3, with the same parameter values as in Figure 1.} \label{figH1states}
\end{figure}

In order to illustrate that the quality of the approximations differ depending on which states we consider, let us compare the stationary probabilities for the states $40$ and $80$. The stationary probabilities of these two states are approximately of the same magnitude and we can compare them in plots with the same scale on both the horizontal and the vertical axes. Figure \ref{figH1states} (a) shows the stationary probability for state $40$ as a function of the perturbation parameter and its approximations for $L = 1,2,3$. The corresponding quantities for state $80$ are shown in Figure \ref{figH1states} (b) where we have omitted the approximation for $L = 3$ since the approximation is very good already for $L = 2$. When $L = 2$, the approximation for state $80$ is clearly better compared to the approximation for state $40$.

Another point illustrated by Figures \ref{figH1stationary} and \ref{figH1states} is that for a fixed value of the perturbation parameter, the quality of an approximation based on a higher order asymptotic expansion is not necessarily better. For instance, in Figure \ref{figH1states} (a) we see that for $\e \in [0.04, 0.05]$ the approximations for $L = 1$ is better compared to both $L = 2$ and $L = 3$. However, asymptotically as $\e \rightarrow 0$, the higher order approximations are better. For example, we see in Figure \ref{figH1states} (a) that when $\e \in [0, 0.02]$ the approximations for $L = 3$ are the best.

Let us now consider a second example for the perturbed model of population genetics. We now choose the parameters as follows: $N = 100$, $C_1 = C_2 = 0$, $D_1 = D_2 = N$, and $S_1 = S_2 = 0$. In this case, both types of mutations have the same probabilities and are equal to the perturbation parameter, that is, $u_1(\e) = u_2(\e) = \e$. This means that both boundary states will be asymptotically absorbing, so condition ${\bf H_3}$ holds.

In this case, we calculated the asymptotic expansions for the stationary and conditional quasi-stationary stationary distribution, given by (\ref{statarc}) and (\ref{statard}), respectively.

Let us illustrate the numerical results obtained for conditional quasi-stationary distributions. Figure \ref{figH3quasi} (a) shows the conditional quasi-stationary distribution for $\e = 0.005$ and some of its approximations. Since it is quite hard to see the details near the boundary states for this plot, we also show the same curves restricted to the states $1$--$20$ in Figure \ref{figH3quasi} (b). As in the previous example, it can be seen that the qualities of the approximations differ between the states. In this case, we see that the approximations for states close to the boundary are not as good as for interior states. Similar type of behavior also appears for different choices of the selection parameters $S_1$ and $S_2$. We omit the plots showing this since they do not contribute with more understanding of the model.

\begin{figure}
\centering
\includegraphics[page=7,scale=0.65]{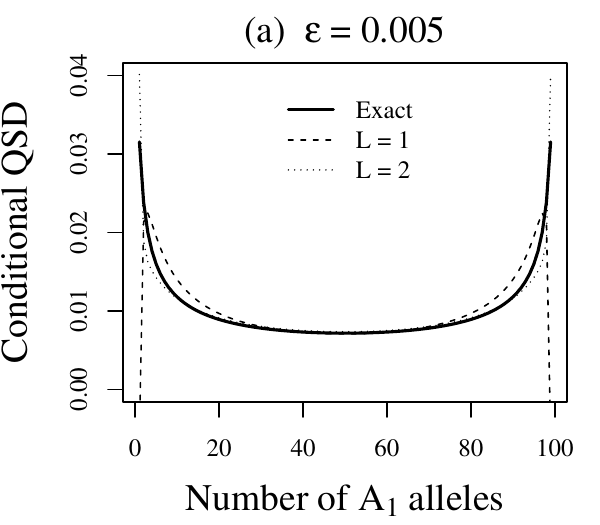}
\includegraphics[page=8,scale=0.65]{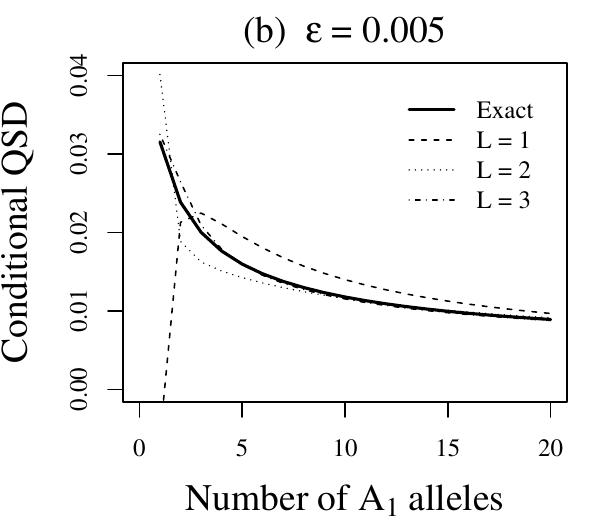}
\caption{The conditional quasi-stationary distribution $\hat{\pi}_i(\varepsilon)$ and some of its approximations for the population genetic example of Section 3.3.
The plots are functions of the number of $A_1$ alleles $i$, with the perturbation parameter $\e = 0005$ fixed. Plot (a) shows the distribution for all states while plot (b) is restricted to states $1$--$20$. The parameter values of the model are $N = 100$, $C_1 = C_2 = 0$, $D_1 = D_2 = N$, and $S_1 = S_2 = 0$.} \label{figH3quasi}
\end{figure}

Let us instead study the limiting conditional quasi-stationary distributions (\ref{CondQuasiLimit}) for some different values of the selection parameters $S_1$ and $S_2$. These types of distributions are interesting in their own right and are studied, for instance, by Allen and Tarnita (2014) where they are called \emph{rare-mutation dimorphic distributions}. In our example, if mutations are rare (i.e., $\e$ is very small), the probabilities of such a distribution can be interpreted as the likelihoods for different allele frequencies to appear during periods of competition which are separated by long periods of fixation.

Figure \ref{figH3limit} (a) shows the limiting conditional quasi-stationary distribution in the case $S_1 = S_2 = 0$, that is, for a selectively neutral model. Now, let the selection parameters be given by $S_1 = 10$ and $S_2 = -10$. In this case, the gene pairs with genotypes $A_1 A_1$, $A_2 A_2$, and $A_1 A_2$ have survival probabilities approximately equal to $0.37$, $0.30$, and $0.33$, respectively. Thus, allele $A_1$ has a selective advantage and this is reflected in Figure \ref{figH3limit} (b) where the limiting conditional quasi-stationary distribution is shown in this case. The mass of the distribution is now shifted to the right compared to a selectively neutral model. Next, we take the selection parameters as $S_1 = S_2 = 10$ which implies that gene pairs with genotypes $A_1 A_1$, $A_2 A_2$, and $A_1 A_2$ have survival probabilities approximately equal to $0.345$, $0.345$, and $0.31$, respectively. This means that we have a model with underdominance and we see in Figure \ref{figH3limit} (c) that the limiting conditional quasi-stationary distribution then has more of its mass near the boundary compared to a selectively neutral model. Finally, we set the selection parameters as $S_1 = S_2 = -10$. Then, gene pairs with genotypes $A_1 A_1$, $A_2 A_2$, and $A_1 A_2$ have survival probabilities approximately equal to $0.32$, $0.32$, and $0.36$, respectively. This gives us a model with overdominace or balancing selection and in this case we see in Figure \ref{figH3limit} (d) that the limiting conditional quasi-stationary distribution has more mass concentrated to the interior states compared to a selectively neutral model.

\begin{figure}
\centering
\includegraphics[page=9,scale=0.65]{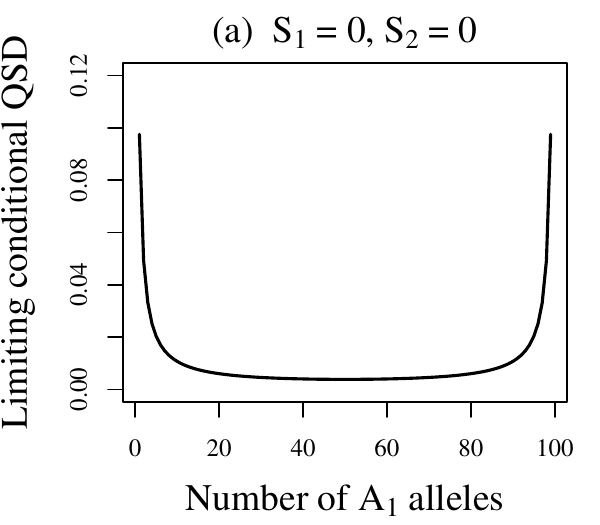}
\includegraphics[page=10,scale=0.65]{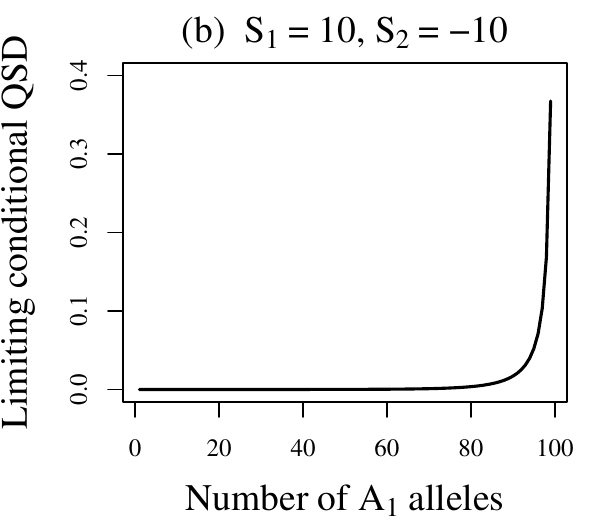} \\
\vspace{5pt}
\includegraphics[page=11,scale=0.65]{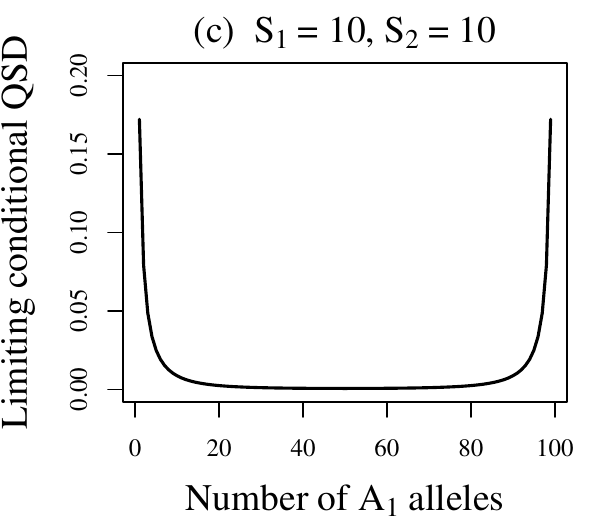}
\includegraphics[page=12,scale=0.65]{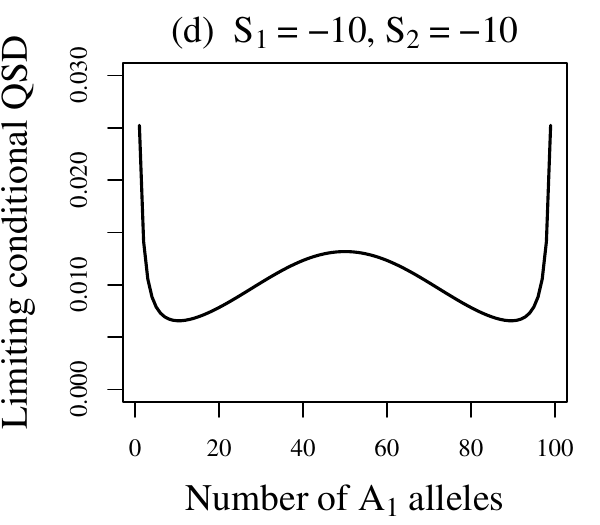}
\caption{Plots of the limiting conditional quasi-stationary distribution $\hat{\pi}_i(0)$ for the population genetic example of Section 3.3, as a function of the number of $A_1$-alleles $i$, for different values of the selection parameters.
The model parameters $N$, $C_1$, $C_2$, $D_1$ and $D_2$ are the same as in Figure 3. Note that the scales of the vertical axes differ between the plots.} \label{figH3limit}
\end{figure}

\par\bigskip
{\bf 7.2 Numerical examples for perturbed epidemic models}
\par\bigskip

In our last numerical example, we consider the perturbed epidemic model described in Section 3.2. Recall from (\ref{ExtRate}) that the contact rate $\nu$ for each individual and the group of infected individuals outside the population is considered as a perturbation parameter, that is, $\nu = \nu(\e) = \e$. In this case, state $0$ is asymptotically absorbing which means that condition ${\bf H_2}$ holds.

It follows directly from (\ref{la+Epid}) and (\ref{la-Epid}) that the intensities of the Markov chain describing the number of infected individuals are linear functions of $\e$ given by
\begin{equation*}
\lambda_{i,+}(\e) = \lambda i (1 - i/N) + (N - i) \e, \quad \lambda_{i,-}(\e) = \mu i, \ i \in \XX.
\end{equation*}

In this model, we only have three parameters to choose: $N$, $\lambda$, and $\mu$. As in the previous examples,  let us take $N = 100$ which here corresponds to the size of the population. Furthermore, we let $\mu = 1$ so that the expected time for an infected individual to be infectious is equal to one time unit. Numerical illustrations will be given for the cases where $\lambda = 0.5$ and $\lambda = 1.5$. For the limiting model, we have in the former case that the basic reproduction ratio $R_0 = 0.5$ and in the latter case $R_0 = 1.5$. The properties of the model is quite different depending on which of these two cases we consider.

For the two choices of model parameters given above we calculated asymptotic expansions for stationary and conditional quasi-stationary distributions given by (\ref{statara}) and (\ref{statarb}), respectively.

\begin{figure}
\centering
\includegraphics[page=13,scale=0.65]{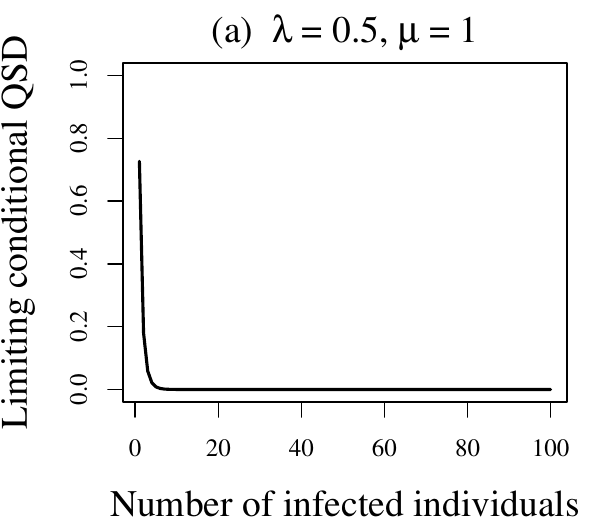}
\includegraphics[page=14,scale=0.65]{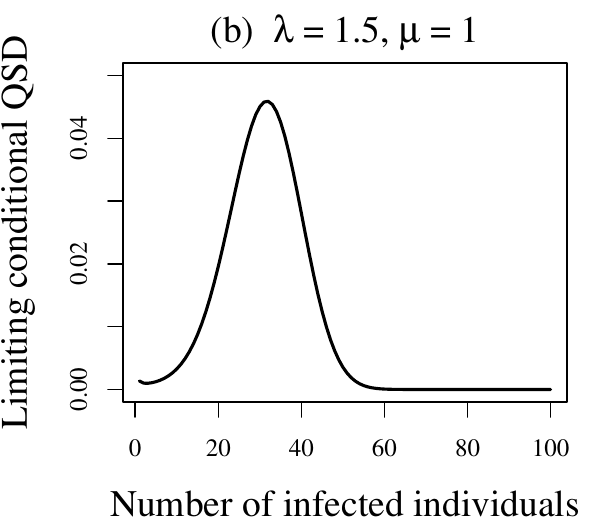}
\caption{Comparison of the limiting conditional quasi-stationary distribution $\tilde{\pi}_i(0)$ for the epidemic model of Section 3.2, as a function of the number of infected  individuals $i$, for a population of size $N=100$ with recovery rate $\mu=1$. The force of infection parameter is $\lambda=0.5$ in a) and $\lambda=1.5$ in b). Note that the scales of the vertical axes differ between the two plots.} \label{figH2limit}
\end{figure}

Let us first compare the limiting conditional quasi-stationary distributions (\ref{tpii0}). Figure \ref{figH2limit} (a) shows this distribution for the case where $\lambda = 0.5$ and $\mu = 1$ and in Figure \ref{figH2limit} (b) it is shown for the case where $\lambda = 1.5$ and $\mu = 1$. In the former case, the limiting conditional quasi-stationary distribution has most of its mass concentrated near zero and in the latter case the distribution has a shape which resembles a normal curve and most of its mass is distributed on the states between $0$ and $60$.

\begin{figure}
\centering
\includegraphics[page=15,scale=0.65]{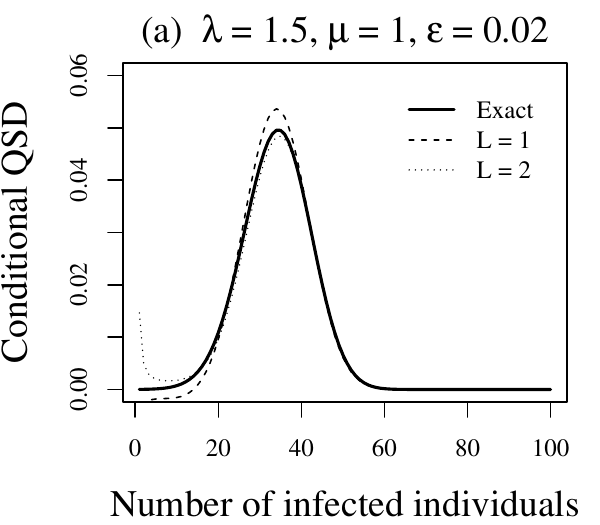}
\includegraphics[page=16,scale=0.65]{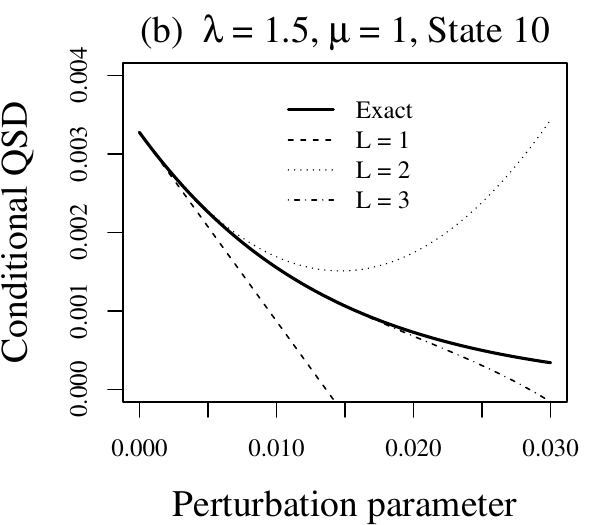}
\caption{Conditional quasi-stationary probabilities $\tilde{\pi}_i(\varepsilon)$ and some approximations for the epidemic model of Section 3.2, with
$N=100$, $\lambda = 1.5$ and $\mu = 1$. Note that the horizontal axes in the two plots represent different quantities; the number of infected individuals $i$ in a) and the perturbation parameter $\varepsilon$ in b).} \label{figH2quasi}
\end{figure}

We can also study plots of the type given in Figures \ref{figH1stationary}--\ref{figH3quasi}. Also in this example, intervals for the perturbation parameter where the approximations are good depend on which state is considered. In this case, states close to zero are more sensitive to perturbations. Let us here just show two of the plots for illustration. For the model with $\lambda = 1.5$ and $\mu = 1$, Figure \ref{figH2quasi} (a) shows the conditional quasi-stationary distribution for $\e = 0.02$ and the corresponding approximations for $L = 1$ and $L = 2$. For the same model parameters, the quasi-stationary probability for state $10$ is shown in Figure \ref{figH2quasi} (b) as a function of the perturbation parameter together with some of its approximations.

\begin{figure}
\centering
\includegraphics[page=17,scale=0.65]{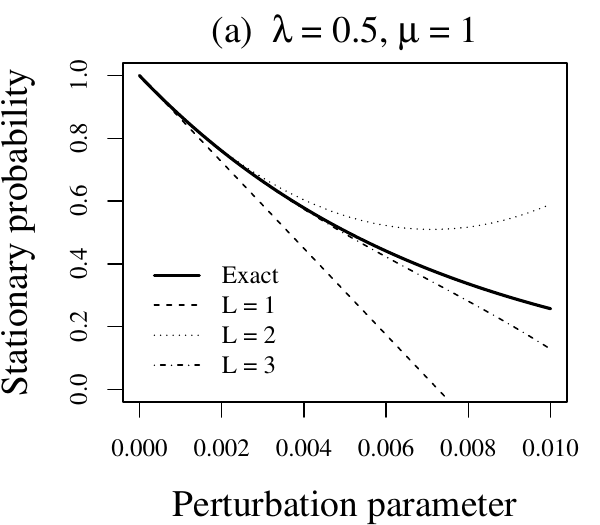}
\includegraphics[page=18,scale=0.65]{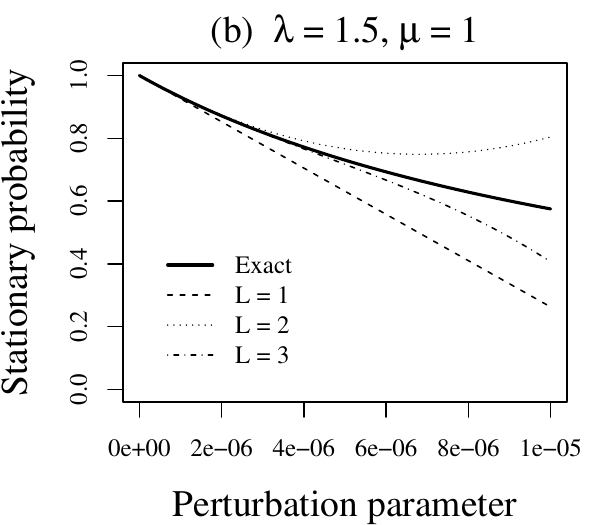}
\caption{Comparison of the stationary probability $\pi_i(\varepsilon)$ of state $i=0$ as a function of the perturbation parameter $\varepsilon$ for the epidemic model of Section 3.2 when $N=100$, $\mu=1$, and the contact rate parameter equals a) $\lambda=0.5$ and b) $\lambda=1.5$. Note that the scales of the horizontal axes differ between the two plots.} \label{figH2state0}
\end{figure}

Finally, let us compare the stationary probabilities for state $0$. Note that, despite that the limiting conditional quasi-stationary distribution is very different depending on whether $R_0 = 0.5$ or $R_0 = 1.5$ for the model with $\e = 0$, the limiting stationary distribution is concentrated at state $0$ in both these cases. Figure \ref{figH2state0} (a) shows the stationary probability of state $0$ as a function of the perturbation parameter and some of its approximations in the case where $\lambda = 0.5$ and $\mu = 1$. The corresponding quantities for the case where $\lambda = 1.5$ and $\mu = 1$ are shown in Figure \ref{figH2state0} (b). 

Qualitatively the plots show approximately the same behavior, but note that the scales on the horizontal axes are very different. We see that the stationary probability of state $0$ for the limiting model is much more sensitive to perturbations in the case where $R_0 = 1.5$. 

It follows from (\ref{NonExtinctProb}) that this is due to fact that the expected time $E_{10}(\varepsilon)$ for the infection to (temporarily) die out after one individual gets infected, is much larger for the model with $R_0=1.5$.  

\pagebreak

{\bf 8. Discussion} \\

The present paper is devoted to studies of asymptotic expansions for stationary and conditional quasi-stationary distributions for perturbed birth-death-type semi-Markov processes. The algorithms of sequential phase space reduction for perturbed semi-Markov processes combined with techniques of Laurent asymptotic expansions developed in the recent papers by Silvestrov, D. and Silvestrov, S. (2015, 2016) are applied to birth-death-type semi-Markov processes. In this model, the proposed algorithms of phase space reduction   preserve the birth-death structure for reduced semi-Markov processes.  This made it possible to get, in the present paper, explicit formulas for coefficients of the corresponding asymptotic expansions for stationary and quasi-stationary distributions. The asymptotic expansions may still be preferable though when the state space is large and (quasi-)stationary distributions are computed for several values of the perturbation parameter, since only the coefficients of the appropriate Laurent expansions are needed. 
We also apply the above results to perturbed models of population dynamics, epidemic models and models of population genetics and supplement theoretical results by computations illustrating numerical accuracy of the corresponding asymptotic expansions and diversity of shape forms for stationary and quasi-stationary distributions in the above perturbed models. 

Several extensions of our work are possible. We have considered semi-Markov processes defined on a finite and linearly ordered state space $\XX$, that is a subset of a one-dimensional lattice. We also confined ourselves to processes of birth-death type, where only jumps to neighboring states are possible. 

For population dynamics models, one needs to go beyond birth-death processes though and incorporate larger jumps in order to account for a changing environment, Lande et al. (2003). State spaces that are subsets of higher-dimensional lattices are of interest in a number of applications, for instance SIR-models of epidemic spread where some recovered individuals get immune, 
N{\aa}sell (2002), population genetic models with two sexes, Moran (1958b), H\"{o}ssjer and Tyvand (2016), and population dynamics or population genetics models with several species or subpopulations, Lande et al. (2003), H\"{o}ssjer et al. (2014). It is an interesting topic of further research to apply the methodology of this paper to such models.

\end{document}